\documentclass[12pt]{amsart}

\usepackage{amssymb}
\usepackage[all]{xy}
\usepackage{graphicx}
\usepackage{caption}
\usepackage{array,boldline,makecell,booktabs}
\newtheorem{theorem}{Theorem}[section]
\newtheorem{lemma}[theorem]{Lemma}

\newtheorem{cor}[theorem]{Corollary}
\theoremstyle{definition}
\newtheorem{definition}[theorem]{Definition}
\newtheorem{example}[theorem]{Example}

\theoremstyle{remark}
\newtheorem{remark}[theorem]{Remark}
\numberwithin{equation}{section}

\theoremstyle{conjecture}

\newcommand\Tstrut{\rule{0pt}{2.6ex}}         

\newcommand\B{\mathbb{B}}
\newcommand\C{\mathbb{C}}

\newcommand\Z{\mathbb{Z}}

\newcommand\T{\mathbb{T}}
\newcommand\cA{\mathcal{A}}
\newcommand\cB{\mathcal{B}}
\newcommand\cC{\mathcal{C}}
\newcommand\cD{\mathcal{D}}
\newcommand\cE{\mathcal{E}}
\newcommand\cF{\mathcal{F}}
\newcommand\cH{\mathcal{H}}

\newcommand\cO{\mathcal{O}}
\newcommand\cP{\mathcal{P}}

\newcommand\cU{\mathcal{U}}

\newcommand\cZ{\mathcal{Z}}

\newcommand\Aut{\operatorname{Aut}}

\newcommand\Inn{\operatorname{Inn}}
\newcommand\End{\operatorname{End}}
\newcommand\Tube{\operatorname{Tube}}
\newcommand\id{\mathrm{id}}
\newcommand\Ad{\mathrm{Ad}}

\newcommand{\Irr}{\operatorname{Irr}}
\newcommand{\Inv}{\operatorname{Inv}}

\newcommand\inpr[2]{\langle{#1,#2}\rangle}

\newcommand{\hG}{\hat{G}}
\newcommand{\tG}{\tilde{G}}

\newcommand{\tmu}{\widetilde{\mu}}
\newcommand{\tpi}{\widetilde{\pi}}
\newcommand{\tsigma}{\widetilde{\sigma}}
\newcommand{\tpsi}{\widetilde{\psi}}

\newcommand{\brho}{\overline{\rho}}

\title[Drinfeld centers of fusion categories]
{Drinfeld centers of fusion categories arising from generalized Haagerup subfactors} 

\author{Pinhas Grossman}
\address{School of Mathematics and Statistics, University of New South Wales,
Sydney NSW 2052, Australia}
\email{p.grossman@unsw.edu.au}

\author{Masaki Izumi}
\address{Department of Mathematics\\ Graduate School of Science\\
Kyoto University\\ Sakyo-ku, Kyoto 606-8502\\ Japan}
\email{izumi@math.kyoto-u.ac.jp}


\thanks{Supported in part by JSPS KAKENHI Grant Number JP15H03623 and ARC grants DP140100732 and DP170103265.}

\begin{document} 

\begin{abstract}
We consider generalized Haagerup categories such that $1 \oplus X$ admits a $Q$-system for every non-invertible simple object $X$. We show that in such a category, the group of order two invertible objects has size at most four. We describe the simple objects of the Drinfeld center and give partial formulas for the modular data. We compute the remaining corner of the modular data for several examples and make conjectures about the general case. We also consider several types of equivariantizations and de-equivariantizations of generalized Haagerup categories and describe their Drinfeld centers.

In particular, we compute the modular data for the Drinfeld centers of a number of examples of fusion categories arising in the classification of small-index subfactors: the Asaeda-Haagerup subfactor; the $3^{\Z_4} $ and $3^{\Z_2 \times \Z_2} $ subfactors; the $2D2$ subfactor; and the $4442$ subfactor.

The results suggest the possibility of several new infinite families of quadratic categories. A description and generalization of the modular data associated to these families in terms of pairs of metric groups is taken up in the accompanying paper \cite{GI19_2}.

\end{abstract}

\maketitle

\section{Introduction} 
In the 1990s Asaeda and Haagerup discovered two ``exotic'' subfactors, known as the Haagerup subfactor and the Asaeda-Haagerup subfactor \cite{MR1686551}. They called these subfactors exotic since they were the first examples of finite depth subfactors which did not arise from known symmetries of groups or quantum groups. The fusion category $ \mathcal{C}$ which is the principal even part of the Haagerup subfactor has  as its group of invertible objects $$G=\text{Inv}(\mathcal{C})=\Z_3 $$ and is tensor generated by a simple object $X$ satisfying the fusion rules
$$ g \otimes X \cong X\otimes g^{-1} \text{ and } X \otimes X=1 \oplus \bigoplus_{g \in G} \limits g \otimes X.$$
The Haagerup subfactor corresponds to the algebra object $1\oplus X$ in this category.

In \cite{MR1832764} the second named author gave a new construction of the Haagerup subfactor in which the simple objects of $\mathcal{C} $ are realized by certain endomorphisms of the Cuntz C$^*$-algebra $ \cO_4$ (extended to a von Neumann algebra closure). In this construction, the four Cuntz algebra generators correspond to embeddings of the four simple summands of $X \otimes X$. The structure constants for the action of these endomorphisms on the generators encode the associativity structure of the tensor category; these constants are in turn determined by certain polynomial equations. 

More generally, it was shown that for an arbitrary finite Abelian group $G$ of odd order, there is a similar system of polynomial equations whose solution gives an associativity structure for a tensor category satisfying the Haagerup fusion rules above, but with $G$ replacing $\Z_3 $. Such a generalized Haagerup category can then be realized via endomorphisms of the Cuntz algebra $ \cO_{|G|+1}$. There is an additional equation for the existence of a $Q$-system (an algebra structure with a unitarity condition \cite{MR1257245}) on $1\oplus X$; a subfactor corresponding to such a $Q$-system is called a generalized Haagerup subfactor for $G$, or a $3^G$ subfactor (after the shape of its principal graph). It was shown  in  \cite{MR1832764} that there is a unique $3^{\Z_5}$ subfactor (up to isomorphism of the standard invariant), but the general existence question was left open.

The Drinfeld center $\cZ(\mathcal{C}) $ of a fusion category is the category of half-braidings of $\mathcal{C} $ by objects $Y \in \mathcal{C} $. For a fusion category over $\mathbb{C} $, $\cZ(\mathcal{C}) $ is a non-degenerate braided fusion category. If $\mathcal{C} $  is spherical (in particular if $\mathcal{C} $ is unitary), then $\cZ(\mathcal{C}) $ has the structure of a modular tensor category \cite{MR1966525}. A modular tensor category gives rise to a projective unitary representation of the modular group $SL_2(\mathbb{Z}) $, with canonical generators mapped to a pair of matrices called the $S$ and $T$ matrices, also known as the modular data. The modular data encodes among other things the fusion rules of the category via the Verlinde formula \cite{MR954762}.

Modular tensor categories appear in a variety of contexts, including conformal field theory \cite{MR1153682}, quantum topology \cite{MR1186845}, and topological quantum computing \cite{MR1910832}. On the other hand, every fusion category can be realized as a category of modules over a commutative algebra in its Drinfeld center. Thus the Drinfeld center construction provides a bridge between the theory of ordinary fusion categories and that of modular tensor categories.

A useful feature of the Cuntz algebra approach to the construction of subfactors is that it comes with a simple formalism for computing arbitrary tensor products and compositions of morphisms in the tensor category. This was exploited in  \cite{MR1832764} to give an explicit description of the tube algebra of the Haagerup category, and thereby compute the modular data of its Drinfeld center.

In \cite{MR2837122}, Evans and Gannon found simpler formulas for the modular data of the Drinfeld center of the Haagerup category. They generalized these formulas to an infinite family of modular data, which they conjectured were realized by Drinfeld centers of generalized Haagerup categories. They also computed solutions to the polynomial equations for generalized Haagerup categories for a number of odd groups, 
and found numerical evidence for the existence of (non-unique in some cases) generalized Haagerup subfactors for all odd cyclic groups up to order $19$.

The fusion categories which are the even parts of the Asaeda-Haagerup subfactor have less symmetric fusion rules than the Haagerup category. The original construction of Asaeda and Haagerup used generalized open string bimodules, a generalization of Ocneanu's connection formalism for graphs, to describe the categories. Their calculations showed the existence of the categories, but did not give a practical framework for performing complicated computations within the category; in particular, a description of the Drinfeld center was not accessible.

In \cite{MR3859276}, the authors and Snyder gave a new construction of the Asaeda-Haagerup subfactor. The construction starts with a generalized Haagerup category for the group $G=\Z_4 \times \Z_2$. The original Asaeda-Haagerup categories are then shown to be Morita equivalent to a $\Z_2  $-de-equivariantization of this category. The system of polynomial equations for generalized Haagerup categories associated to even groups is considerably more subtle than for the odd case, and involves a collection $\epsilon_{(.)}(g) $ of characters  of $G_2 $, the group of order two elements of $G$, indexed by $G$   \cite{MR3827808}. 

Since the Drinfeld center is an invariant of Morita equivalence, this construction allows us to use the Cuntz algebra framework for generalized Haagerup categories to describe the Drinfeld center and compute the modular data of the Asaeda-Haagerup categories (which was first announced in \cite{MR3859276}). The original motivation of this paper was to describe this computation. 
 
However, it turns out that generalized Haagerup categories for groups of even order play a central role in the classification of small-index subfactors beyond the Asaeda-Haagerup case. Subfactors with index less than $4$ have index of the form $ 4\cos^2 \frac{\pi}{k}, k=3,4,5...$ \cite{MR696688}, and their principal graphs are simply laced Dynkin diagrams. In the 1990s Haagerup initiated the classification of subfactors with index slightly above $4$ by searching for admissible principal graphs (which is how the Haagerup and Asaeda-Haagerup subfactors were discovered)  \cite{MR1317352}.
This classification has now been extended to index $5$ \cite{MR3166042}, and then to index $5.25 $ \cite{1509.00038}, with only a small number of finite-depth examples appearing.

The most interesting index value between $5$ and $5.25$ is $3+\sqrt{5} $, which is the first composite index above $4$. There are exactly seven finite depth subfactor planar algebras at index $3+\sqrt{5}$ up to duality \cite{1509.00038}. Of these, two are the unique $3^{\Z_4}$ and $3^{\Z_2 \times \Z_2}$ subfactors; another one is the $2D2$ subfactor , which is related to the $3^{\Z_4}$ subfactor through a $\Z_2 $-de-equivariantization; and another one is the $4442$ subfactor, which is related to the $3^{\Z_2 \times \Z_2}$ subfactor through a $ \Z_3$-equivariantization (and another one is related to the $3^{\Z_4}$ subfactor by Morita equivalence) \cite{MR3827808,MR3314808,MR3394622,MR3402358}.

In addition to the Asaeda-Haagerup subfactor, we compute the modular data for these four subfactors with index $3+\sqrt{5}$, as well as for a $\Z_2 $-de-equivariantization of a $3^{\Z_8} $ subfactor. We also numerically compute modular data for $3^{\Z_6} $, $3^{\Z_8}$, and $3^{\Z_{10}} $ subfactors (of which there are two each for $\Z_6 $ and $\Z_{10} $).

More generally, we consider families of subfactors associated to these examples. The appearance of the characters $\epsilon_{(.)}(g) $ on $G_2$ makes the study of generalized Haagerup categories for even groups considerably more complicated than the odd case. Indeed, generalized Haagerup categories for odd groups can always be constructed as de-equivariantizations of near-group categories \cite{MR3635673}; this is not true for even groups. However, our first main result restricts the groups for which generalized Haagerup subfactors can exist, under an additional assumption. 
\begin{theorem}
Let $G$ be the group of invertible objects of a generalized Haagerup category, and suppose that $1\oplus (g \otimes X) $ admits a $Q$-system for all $g \in G$. Let $G_2 =(\Z_2)^m \subseteq G$ be the elementary $2$-group of order $2$ elements. Then $m \leq 2 $.
\end{theorem}

In light of this result, we consider generalized Haagerup categories associated to a group
$G=\Z_{2^m} \times \Z_{2^n} \times H   $ with $m \geq n \geq  0$ and $|H|$ odd. By analyzing the tube algebra, we can describe the simple objects in the Drinfeld center. All of the simple objects in the center contain invertible simple summands except for
$$\mu =\bigoplus_{g \in G}  \limits g \otimes X$$ (which has $\frac{|G|^2}{2} $ half-braidings); and, in the case that $|G_2|=2$, the objects $\nu_{\pm} $, where $\mu=\nu_+ \oplus \nu_-$ is the decomposition of $\mu $ into the direct sums of those $g \otimes X $ indexed by the two cosets of $2G $ in $G$ (which have two half-braidings each).  

We give formulas for the parts of the modular data that do not involve half-braidings of $\mu $ and $\nu $ in terms of the characters $\epsilon_{(.)}(g) $. In the case that $|G_2|=2 $, we can also find formulas for the parts of the modular data when only one of the indices is a half-braiding of $\mu $ or $\nu $. In the case that $|G_2|=4$, under the additional assumption that the braiding on $G_2$ is nondegenerate (which is a property of the characters $\epsilon_{(.)}(g) $), the modular data factors, and we have formulas for the modular data modulo the corner indexed by the $\frac{|G|^2}{8} $ half-braidings of $\mu $ in the commutant of $G_2$. 

To compute the missing corners of the modular data, it is necessary to perform detailed calculations in the tube algebra, and we do not have general formulas. However, we have conjectures based on computations for small examples.

For the Asaeda-Haagerup subfactor, we consider a generalized Haagerup category $\cC $ for $G=\Z_4 \times \Z_2$ with a certain form of $\epsilon$, as in \cite{MR3859276}. Then $\text{Vec}_{\Z_2 \times \Z_2}$  lifts to the center as a nondegenerate subcategory, and its commutant in  $\cZ(\cC)$ is exactly the Drinfeld center of the Asaeda-Haagerup categories. More generally, we can consider  $\Z_2$-de-equivariantizations for generalized Haagerup categories for $\mathbb{Z}_{4m} \times \Z_2 $ with a similar form of $\epsilon $. We compute the missing corner of the modular data for the Asaeda-Haagerup case $m=1$. 

The $2D2 $ subfactor arises from a $\Z_2$-de-equivariantization of a generalized Haagerup category for $\Z_4 $ with non-trivial $\epsilon $ (which is uniquely determined up to gauge equivalence). We can similarly consider $\Z_2$-de-equivariantizations of generalized Haagerup categories for $\Z_{4m} $ with non-trivial $\epsilon $. We obtain formulas for the modular data except for a corner, and compute the missing corner for the two known examples $m=1,2$.

Finally, we consider the $4442$ subfactor, which comes from a $\Z_3$-equivariantization of a generalized Haagerup category $ \mathcal{C}$ for $\Z_2 \times \Z_2 $. We compute the modular data from the tube algebra of the Morita equivalent crossed product category, which is a $\Z_3 $-graded extension of $\mathcal{C} $ and has a Cuntz algebra model. It may be possible to generalize this construction to generalized Haagerup categories of the form $\Z_2 \times \Z_2 \times H $ with $|H|$ odd, but we do not know of any examples for nontrivial $H$, and we do not attempt this generalization here.

In summary, these examples suggest several possibly infinite families of quadratic categories, which are distinguished by the $2$-subgroups of the associated finite group.

\begin{itemize}
\item Generalized Haagerup categories for $G=\Z_{4m} $, and their $\Z_2 $-de-equivariantizations (or more generally $\Z_{4m} \times H $ with $H$ odd) - known to exist for $m=1,2$ and $H$ trivial.
\item  Generalized Haagerup categories for $G=\Z_{4m+2} $ (or more generally $\Z_{4m+2} \times H $ with $H$ odd) - known to exist for $m=0,1,2 $ and $H$ trivial. (For $m=1,2$ these occur in pairs.)
\item Generalized Haagerup categories for $G= \Z_{4m} \times \Z_2 $ (or more generally $ \Z_{4m} \times \Z_2 \times H $ with $H$ odd), and their $\Z_2 $-deequivariantization - only known to exist when $m=1$ and $H$ is trivial.
\item Generalized Haagerup categories for $G=\Z_2 \times \Z_2 \times H$, with $H$ odd, and its $\Z_3 $-equivariantization - both only known to exist when $H$ is trivial.
\end{itemize}

An analysis of potential modular data associated to such families, as well as generalizations of this modular data, in terms of pairs of metric groups will appear in the accompanying paper \cite{GI19_2}.

To compute the missing corners of the modular data in examples, we use Mathematica to perform the arithmetic in the tube algebra. These calculations require formulas for the tube algebra operations. Such formulas are in principle straightforward to derive from the definition of the tube algebra, the Cuntz algebra model for generalized Haagerup categories, and the orbifold construction for (de)-equivariantization. In practice the formulas are laborious to write down, so we include them in an online appendix.

The paper is organized as follows.

In Section \ref{prelim} we review some preliminaries on fusion categories, tube algebras, and generalized Haagerup subfactors.

In Section \ref{restrict} we prove the restriction on $G_2$ for a certain type of generalized Haagerup category. Then we describe the tube algebra and give partial formulas for the modular data.

In Section \ref{corner} we describe how to compute the remaining corner of the modular data, and discuss several examples with $|G_2|$=2.

In Section \ref{factor} we consider the Muger factorization of the Drinfeld center for the case that $|G_2|=4 $ and the braiding on $G_2$ is non-degenerate. This case includes the Asaeda-Haagerup subfactor (via de-equivariantization).

In Section \ref{deeq} we consider a $\Z_2 $-de-equivariantization of a generalized Haagerup for $\Z_{4m} $. This case includes the $2D2$ subfactor. 

In Section \ref{4442} we describe the tube algebra and compute the modular data for the $\Z_3 $-equivariantization of a generalized Haagerup category for $\Z_2 \times \Z_2 $ (which is the even part of the $4442$ subfactor).

In an online appendix, included with the \texttt{arxiv} source, we give formulas for multiplication, involution, and rotation in the tube algebras of generalized Haagerup categories and their (de)-equivariantizations.

In an accompanying Mathematica notebook \texttt{solutions.nb}, also included with the \texttt{arxiv} source, we give the structure constants for generalized Haagerup categories for $\Z_6 $, $\Z_8 $, and $\Z_{10} $.

This paper replaces an earlier \texttt{arxiv} preprint titled ``Quantum doubles of generalized Haagerup subfactors and their orbifolds'', which described the computations of most of the main examples in this paper but without discussing the general cases.

\textbf{Acknowledgements.}  We would like to thank the American Institute of Mathematics for its generous hospitality during the SQuaRE Classifying Fusion Categories and the Isaac Newton Institute for its generous hospitality during the semester program Operator algebras: Subfactors and their Applications. This paper was originally motivated by the new construction of the Asaeda-Haagerup subfactor from a generalized Haagerup category, which was joint work with Noah Snyder \cite{MR3859276}. We would like to thank Marcel Bischoff for pointing out to us the realization of the  $3^{\Z_2 \times \Z_2} $ subfactor from a conformal inclusion in \cite{MR3764563}. We would like to thank Eric Rowell for pointing out to us the zesting construction in \cite{MR3641612}.

\section{Preliminaries} \label{prelim}

\subsection{Fusion categories and the category $\text{End}_0(M)$}
A fusion category over the complex numbers $\C$ is a rigid semisimple $\C$-linear tensor category with 
finitely many simple objects and finite dimensional morphism spaces such that the identity object is simple 
(see \cite{MR3242743}). 
In this paper we concentrate on unitary fusion categories embedded as full monoidal subcategories of the category of unital endomorphisms 
$\End(M)$ of a type III factor $M$. 
Note that such an embedding always exists under the assumption of unitarity (see \cite[Theorem 2.1]{MR3635673}). 
Our standard reference for the category $\End(M)$ is \cite{MR3308880}. We follow the convention of using Greek letters  for objects in $\text{End}(M) $.

For a Hilbert space $\cH$, we denote by $\B(\cH)$ the set of bounded operators on $\cH$, and by $\cU(\cH)$ the set of unitaries on $\cH$. 
The identity operator of $\cH$ is denoted by $1_{\cH}$ or simply by $1$. 
For a unital C$^*$-algebra $A$, we denote by $\cU(A)$ the set of unitaries in $A$. 
The unit of $A$ is denoted by $1_A$ or simply by $1$. 

Let $M$ be a type III factor. 
Then the set of unital endomorphisms $\End(M)$ forms a category with the morphism space from an object $\rho$ to another object $\sigma$ 
given by 
$$\text{Hom}(\rho,\sigma)=(\rho,\sigma)=\{t \in M;\; t \rho(x)=\sigma(x)t,\;\forall x\in M\}.$$
This category has a strict monoidal structure, with the monoidal product $\rho\otimes \sigma$ of 
two objects $\rho, \sigma\in \End(M)$ given by the composition $\rho\circ \sigma$, and the monoidal product between morphisms $t \in (\rho_1,\rho_2) $ and $s \in (\sigma_1, \sigma_2) $ given by 
$$t \otimes s =  t \rho_1(s)=\rho_2(s)t \in (\rho_1 \circ \sigma_1, \rho_2 \circ \sigma_2) .$$
The monoidal unit of $\End(M)$ is the identity automorphism $\id$ of $M$. The endomorphism space $(\rho,\rho)$ is just the relative commutant $M\cap \rho(M)'$; $\rho$ is simple, or irreducible, if this space consists only of the scalars. 
When discussing the monoidal category $\End(M) $, we will generally suppress the ``$\otimes$'' symbol, and refer directly to multiplication in $M$ and composition of endomorphisms.

The morphism space $(\rho,\sigma)$ inherits a Banach space structure from $M$, and the $*$-operation 
of $M$ maps $(\rho,\sigma)$ to $(\sigma,\rho)$, which makes $\End(M)$ a C$^*$-tensor category 
(see \cite[Section 1]{MR3308880}). 
Moreover, if $\rho$ is simple, the space $(\rho,\sigma)$ is a Hilbert space with an inner product 
given by $t_1^*t_2=\inpr{t_1}{t_2}1_M$ for $t_1,t_2\in (\rho,\sigma)$. 

For $\rho\in \End(M)$, its dimension $d(\rho)$ is defined by $[M:\rho(M)]_0^{1/2}$, 
where $[M:\rho(M)]_0 \in [1,\infty]$ is the minimum index of $\rho(M)$ in $M$. 
We denote by $\End_0(M)$ the set of $\rho\in \End(M)$ with finite $d(\rho)$. 
The dimension function $\End_0(M)\ni \rho\mapsto d(\rho)$ is additive with respect to 
the direct sum operation and multiplicative with respect to the monoidal product operation. 
The monoidal category $\End_0(M)$ is rigid, which means that objects have left and right duals. For any $\rho\in \End_0(M)$, 
there exists $\brho\in \End_0(M)$, called a dual, or conjugate, endomorphism of $\rho$, and two isometries 
$r_\rho\in (\id,\brho\circ \rho)$, $\overline{r}_\rho\in (\id,\rho\circ \brho)$ satisfying 
$$\overline{r}_\rho^*\rho(r_\rho)=r_\rho^*\brho(\overline{r}_\rho)=\frac{1}{d(\rho)}.$$
The evaluation morphism $\mathrm{ev}_\rho$ is identified with $\sqrt{d(\rho)}\overline{r}_\rho$, 
and the coevaluation morphism $\mathrm{coev}_\rho$ is identified with $\sqrt{d(\rho)} r_\rho^*$. 

If $\rho $ and $\sigma $ are isomorphic objects in $\End_0(M)$, then there exists a unitary $u\in \cU(M)$ 
satisfying $\rho=\Ad u\circ \sigma$, where $\Ad u$ is the inner automorphism of $M$ given by $\Ad u (x)=uxu^*$; we then say that $\rho $ and $\sigma $ are unitarily equivalent, or simply equivalent. We denote by $\Inn(M)$ the group of inner automorphisms of $M$.

For a fusion category $\cC\subset \End(M)$, the categorical dimension $d(\rho)$ coincides with the Frobenius-Perron dimension in $\cC $. We denote by $\Irr(\cC)$ the set of isomorphism classes of 
simple objects in $\cC$. 
We often identify an element in $\Irr(\cC)$ with one of its representatives if there is no possibility of confusion. 
The global dimension of $\cC$ is defined by 
$$\dim \cC=\sum_{\xi\in \Irr(\cC)}d(\xi)^2.$$

\subsection{The Drinfeld Center}

Let $\mathcal{C} $ be a monoidal category. A half-braiding for an object $X \in \mathcal{C} $ is a
natural isomorphism $\cE_X: X \otimes (-) \rightarrow (-) \otimes X$,  satisfying the hexagon and unit identities. If $\cC $ is strict, these identities reduce to
 
  $$\cE_X(Y \otimes Z)=(\id_Y \otimes \cE_X(Z)) \circ (e_X(Y) \otimes \id_Z),  \ \forall Y,Z \in \mathcal{C}.$$
and
$$\cE_X(1)=\id_{X} ,$$
respectively.

The Drinfeld center $\cZ(\mathcal{C}) $ is the
category whose objects are half-braidings $(X, \cE_X) $ of objects in $\mathcal {C}$ and whose morphisms are given by  $$\text{Hom}((X, \cE_X) ,(Y, \cE_Y) )=$$ $$\{ t \in \text{Hom}(X,Y): (\id_Z \otimes t) \circ \cE_X(Z)
=\cE_Y(Z) \circ (t \otimes \id_Z), \ \forall Z \in \mathcal{C} \} .$$ The Drinfeld center is a braided monoidal category, with tensor product of 
 $(X, \cE_X) $ and $(Y,\cE_Y) $ given by $$(X\otimes Y, \cE_{X \otimes Y} )$$
 with
$$\cE_{X \otimes Y}(Z) =(\cE_X(Z) \otimes \id_Y) \circ (\id_X \otimes \cE_Y(Z)),  \ \forall Z \in \mathcal{C}, $$
and braiding given by $$c_{X,Y}=\cE_X(Y)   .$$

If $\cC $ is semisimple, then a half-braiding $\cE_X $ is determined by $\cE_X(Y)$ as $Y$ ranges over representatives of isomorphism classes of simple objects. If $\cC $ is a fusion category, then $\cZ(\mathcal{C}) $ is a fusion category as well, and the braiding on $\cZ(\mathcal{C}) $ is non-degenerate. If $\cC$ is a unitary fusion category, the unitary Drinfeld center (where the  $e_X(Y) $ are required to be unitaries) is equivalent to the ordinary Drinfeld center \cite{MR1966525}. We will therefore only consider unitary half-braidings in this paper.

A modular tensor category is a non-degenerate braided spherical fusion category $ \mathcal{C}$. The $S$-matrix of a modular tensor category is defined by
$$S_{X,Y}= \frac{d_X d_Y}{ \sqrt{\text{dim}(\mathcal{C}) }} \text{tr}_{X \otimes Y} (c_{X,Y} \circ c_{Y,X})  ,$$
for simple objects $X$ and $Y$, where $ c$ is the braiding on $\mathcal{C} $, $d_X $ is the quantum dimension, $\text{dim} (\mathcal{C}) $ is the global dimension, and $\text{tr}$ is the normalized spherical trace on $\text{End} (X \otimes Y)$. The $S$-matrix is defined up to a choice of square root of the global dimension.

The $T$-matrix is defined by 
$$T_{X,Y}= \delta_{X,Y} d_X \text{tr}_{X \otimes X}(c_{X,X}) , $$ 
and the conjugation matrix $C$ is defined by $$C_{X,Y} =\delta_{X,\bar{Y}},$$
where $\bar{Y} $ is the dual object of $Y$.

For a modular tensor category over $\mathbb{C} $, $S$ is symmetric, $T$ is diagonal with finite order, and $S$ and $T$ are unitary \cite{MR2183279}. We have the relations $$\alpha(ST)^3=S^2=C =T^{-1}CT $$
for a scalar $\alpha $
\cite{MR1153682,MR1797619}. 

If $\mathcal{C} $ is a spherical fusion category over $\mathbb{C} $, then $\mathcal{Z}(\mathcal{C}) $ is a modular tensor category. 
We fix $\sqrt {\text{dim} (\mathcal{Z}(\mathcal{C})) }=\text{dim}(\mathcal{C})$, and then $\alpha=1 $ \cite{MR1966525}.

We now return to a unitary fusion category $\cC $ embedded in $\text{End}_0(M) $. Let $\sigma\in \cC$ be a (not necessarily simple) object in $\cC$. Then the data of a half-braiding for $\sigma$ is given by
 family of unitaries $\cE_\sigma=\{\cE_\sigma(\xi)\}_{\xi\in \Irr(\cC)}$ with
$\cE_\sigma(\xi)\in (\sigma\xi,\xi\sigma)$ such that for any $t \in (\zeta,\xi\eta)$ with $\xi,\eta,\zeta\in \Irr(\cC)$, 
we have
$$t \cE_\sigma(\zeta)=\xi(\cE_\sigma(\eta))\cE_\sigma(\xi)\sigma(t).$$ 

In general, a single object $\sigma$ may have several inequivalent half-braidings, and we introduce labelling  to distinguish them  
and use the notation $\widetilde{\sigma}^l=(\sigma,\cE_\sigma^l)$ for simplicity. 
We denote by $F$ the forgetful functor $\cF:\cZ(\cC)\to \cC$. 


\subsection{The tube algebra}\label{tube}
We summarize the basics of tube algebras following \cite{MR1782145}. 

The tube algebra for a fusion category $\cC\subset \End(M)$ is a finite dimensional C$^*$-algebra with 
underlying vector space 
$$\Tube\cC=\bigoplus_{\xi,\eta,\zeta\in \Irr(\cC)}(\xi\zeta,\zeta\eta).$$
An element $x \in (\xi\zeta,\zeta\eta)$ is denoted as an element of $\Tube \cC$ by 
$(\xi\;\zeta|x |\zeta\;\xi)$. 
The $*$-algebra operations of $\Tube \cC$ are defined by 
\begin{align*}
\lefteqn{(\xi\;\zeta|x|\zeta\;\eta)(\xi'\;\zeta'|y|\zeta'\;\eta')} \\
 &=\delta_{\eta,\xi'}\sum_{\nu\in \Irr(\cC)}\sum_{i=1}^{\dim(\nu,\zeta\zeta')}
 (\xi\;\nu|t(_{\zeta,\zeta'}^\nu)_i^*\zeta(y)x\xi(t(_{\zeta,\zeta'}^\nu)_i) |\nu\;\eta'),
\end{align*}
$$(\xi\;\zeta|x|\zeta\;\eta)^*=d(\zeta)(\eta\;\overline{\zeta}|\overline{\zeta}(\xi(\overline{r}_\zeta^*)x^*)r_\zeta|\overline{\zeta}\xi),$$
where $\{t(_{\zeta,\zeta'}^\nu)_i\}_{i=1}^{\dim(\nu,\zeta\zeta')}$ is an orthonormal basis of the space $(\nu,\zeta\zeta')$. 
We denote 
$$\cA_{\xi,\eta}=\bigoplus_{\zeta\in \Irr(\cC)}(\xi\zeta,\zeta\eta),$$
which is a subspace of $\Tube \cC$ satisfying 
$$\cA_{\xi,\eta}\cA_{\xi',\eta'}\subseteq \delta_{\eta,\xi'}\cA_{\xi,\eta'},\quad \cA_{\xi,\eta}^*=\cA_{\eta,\xi}.$$
In particular, $\cA_\xi:=\cA_{\xi,\xi}$ is a $*$-subalgebra of $\Tube \cC$ with unit 
$1_\xi=(\xi\;0|1|0\;\xi)$, 
where we denote $\id$ by $0$ for simplicity. 
We have $$\sum_{\xi\in \Irr(\cC)}\limits 1_\xi=1_{\Tube \cC}.$$
The algebra $\cA_\xi$ is a corner of $\Tube \cC$ in the sense that we have $1_\xi \Tube \cC 1_\xi=\cA_\xi$. 
In particular, every minimal projection in $\cA_\xi$ is minimal in $\Tube \cC$. 
The space $\cA_{\xi,\eta}$ is a $\cA_\xi$-$\cA_\eta$ bimodule with respect to the left and right multiplication 
of $\cA_\xi$ and $\cA_\eta$ respectively.

There is a one to one correspondence between the set of simple components of $\Tube \cC$ and $\Irr(\cZ(\cC))$, 
and we denote by $z(\widetilde{\sigma}^l)$ the minimal central projection in $\Tube \cC$ corresponding to $\widetilde{\sigma}^l$. 
Then the algebra $z(\widetilde{\sigma}^l)\Tube \cC$ is isomorphic to the full matrix algebra, and 
we can write down a system of matrix units for it in terms of $\cE_\sigma^l$ as follows. 
We choose an orthonormal basis $\{w_\sigma(\xi)_i\}_{i=1}^{\dim (\xi,\sigma)}$ of $(\xi,\sigma)$ for each 
$\xi\in \Irr(\cC)$, and set 
$$\cE_\sigma^l(\zeta)_{(\xi,i),(\eta,j)}=\zeta(w_\sigma(\eta)_j^*)\cE_\sigma^j(\zeta)w_\sigma(\xi)_i\in(\xi\zeta,\zeta\eta),$$
\begin{equation}\label{matuni}
e(\widetilde{\sigma}^l)_{(\xi,i),(\eta,j)}=\frac{d(\sigma)}{\Lambda\sqrt{d(\xi)d(\eta)}}\sum_{\zeta\in \Irr(\cC)}d(\zeta)
(\xi\;\zeta|\cE_\sigma(\zeta)_{(\xi,i),(\eta,j)}|\zeta\;\eta)\in \Tube \cC,
\end{equation}
where $\Lambda$ is the global dimension of $\cC$. 
Then $\{e(\widetilde{\sigma}^l)_{(\xi,i),(\eta,j)}\}_{(\xi,i),(\eta,j)}$ forms a system matrix units for the subalgebra 
$z(\widetilde{\sigma}^l)\Tube \cC$, and 
\begin{equation}\label{mincenpr}
z(\widetilde{\sigma}^l)=\sum_{(\xi,i)}e(\widetilde{\sigma}^l)_{(\xi,i),(\xi,i)}.
\end{equation}
In particular, the rank of $z(\widetilde{\sigma}^l)\Tube \cC$ is 
$$\sum_{\xi\in \Irr(\cC)}\dim(\xi,\sigma).$$
Eq.(\ref{matuni}) gives a very practical way to determine half-braidings from the algebra structure of $\Tube \cC$. 
Note that 
$$\sum_{i}e(\widetilde{\sigma}^l)_{(\xi,i),(\xi,i)}$$
is a minimal central projection in $\cA_\xi$, and the corresponding simple component of $\cA_\xi$ has rank $\dim (\xi,\sigma)$. 
Conversely every minimal central projection in $\cA_\xi$ is of this form. 
The element $e(\tsigma^l)_{(\xi,i),(\xi,i)}$ acts on $\cA_{\xi,\eta}$ as a projection of rank $\dim (\eta,\sigma)$. 

The modular data $(S,T)$ for the modular tensor category $\cZ(\cC)$ can be computed in terms of $\Tube \cC$ 
as follows. 
For $\xi\in \Irr(\cC)$, we set $$\mathbf{t}_{\xi}=d(\xi)(\xi\;\overline{\xi}|r_\xi\overline{r}_\xi^*|\overline{\xi}\;\xi).$$ 
Then $\mathbf{t}_\xi$ is a unitary central element of $\cA_\xi$ with adjoint 
$\mathbf{t}_\xi^*=(\xi\;\xi|1|\xi\;\xi)$. 
Let 
$$\mathbf{t}=\sum_{\xi\in \Irr(\cC)}\mathbf{t}_\xi.$$
Then $\mathbf{t}$ is a central unitary element in $\Tube\cC$, giving the $T$-matrix via 
$$\mathbf{t}z(\widetilde{\sigma}^l)=T_{\widetilde{\sigma}^l,\widetilde{\sigma}^l}z(\widetilde{\sigma}^l).$$
We can compute $T$ by computing the eigenvalues of $\mathbf{t}_{\xi}$ (or $\mathbf{t}_{\xi}^*$). 

We introduce a linear transformation $S_0$ of 
$$\bigoplus_{\xi}\cA_\xi$$
by 
$$S_0((\xi\;\eta|x|\eta\;\xi))
=d(\xi)(\overline{\eta}\;\xi|r_\eta^*\overline{\eta}(x\xi(\overline{r}_\eta))|\xi\overline{\eta}),$$
which can be thought of as rotation.
Then $S_0$ preserves the center of $\Tube \cC$, and the $S$-matrix is given by the matrix coefficients 
of $S_0$ with respect to the basis $\{\frac{\sqrt{\Lambda}}{d(\sigma)}z(\widetilde{\sigma}^l)\}$. 

We can extract these matrix coefficients using the linear functional on $\Tube\cC $
$$\psi ( \xi \; \zeta | x | \zeta \; \eta  )=d(\xi)^2\delta_{\xi, \eta}\delta_{\zeta,0}x .$$

Then
\begin{equation}\label{Sformula0}
S_{\widetilde{\sigma}^l,\widetilde{\mu}^m}=\frac{d(\sigma)d(\mu)}{\Lambda}\psi (S_0 (z(\widetilde{\sigma}^l)), z(\widetilde{\mu}^m) ).
\end{equation}

In practice, it is often easier to compute $S$ and $T$ from the following formulae via Eq.(\ref{matuni}) (see \cite[Lemma 5.3]{MR1782145}). 
Let $\phi_\xi(x)=r_\xi^*\xi(x)r_\xi$ be the standard left inverse of $\xi\in \Irr(\cC)$. Fix a simple summand $(\eta,j) $ of $\sigma $. Then

\begin{equation}\label{Sformula}
S_{\widetilde{\sigma}^l,\widetilde{\mu}^m}= \frac{d(\sigma)}{\Lambda}\sum_{\xi,i}d(\xi)\phi_\xi(\cE_\mu^m(\eta)_{(\xi,i),(\xi,i)}^*\cE_\sigma^l(\xi)_{(\eta,j),(\eta,j)}^*),
\end{equation}
where the sum is taken over simple summands of $ \mu$, and
\begin{equation}\label{Tformula}
T_{\widetilde{\sigma}^l,\widetilde{\sigma}^l}=d(\eta)\phi_\eta(\cE_\sigma^l(\eta)_{(\eta,j),(\eta,j)}). 
\end{equation}

Eq.(\ref{matuni}) implies the following observation, which allows us to determine $\Irr(\cZ(\cC))$ from the algebra structure of $\Tube \cC$. 

\begin{lemma}\label{bimodule} Let $\xi,\eta\in \Irr(\cC)$. 
\begin{itemize}  
\item[(1)]
The $\cA_\xi$-$\cA_\eta$ bimodule $\cA_{\xi,\eta}$ is decomposed as
$$\cA_{\xi,\eta}=\sum_{\mu\in \Irr(\cZ(\cC))}z(\mu)\cA_{\xi,\eta},$$
where each $z(\mu)\cA_{\xi,\eta}$ is, if it is not 0, an irreducible $\cA_\xi$-$\cA_\eta$ bimodule of dimension $\dim(\xi,\cF(\mu))\dim(\eta,\cF(\mu))$. 
If two objects $\mu,\nu\in \Irr(\cZ(\cC))$ are inequivalent, so are the corresponding bimodules $z(\mu)\cA_{\xi,\eta}$ and $z(\nu)\cA_{\xi,\eta}$. 
\item[(2)] Let $\mu\in \Irr(\cZ(\cC))$ with $\dim (\xi,\cF(\mu))\neq 0$. 
Then every minimal projection in $z(\mu)\cA_\xi$ acts on $\cA_{\xi,\eta}$ as a projection of rank 
$\dim(\eta,\cF(\mu))$. 
\end{itemize}
\end{lemma}

Apparently the following property of irreducible objects of fusion categories has never been observed before 
except in \cite[Proof of Theorem 6.4, Lemma 8.1,(7)]{MR1832764}. 
It often occurs in quadratic categories as a key feature that allows us to compute their tube algebras. 

\begin{definition} An irreducible object $\xi\in \cC$ is said to be \textit{multiplicity free} 
if the subalgebra $\cA_\xi$ of $\Tube \cC$ is abelian. 
\end{definition}

Eq.(\ref{matuni}) and Lemma \ref{bimodule} imply the following. 

\begin{lemma} \label{mulfree} Let $\xi,\eta\in \Irr(\cC)$. 
\begin{itemize}
\item[(1)] The object $\eta$ is multiplicity free if and only if for any $\mu\in \Irr(\cZ(\cC))$ the multiplicity of $\eta$ in 
$\cF(\mu)$ is at most one. 
\item[(2)] Assume that $\eta$ is multiplicity free. 
Then the multiplicity of every irreducible $\cA_\xi$ module contained in $\cA_{\xi,\eta}$ as a left $\cA_\xi$-module is at most one.  
\end{itemize}
\end{lemma}

\subsection{Generalized Haagerup categories} 

We now introduce generalized Haagerup categories embedded in $\End(M)$ 
(see \cite[Definition 2.7]{MR3827808} for the definition without unitarity). 
Let $\cC\subset \End(M)$ be a fusion category. 
Then the set of isomorphism classes of invertible objects forms a finite group, and we denote it by $G$. 
Throughout the paper we assume that $G$ is abelian, and we use additive notation for $G$. 
We choose a representative $\alpha_g\in \cC$ for each $g\in G$. 
Then by definition, we get 
$$[\alpha_g][\alpha_h]=[\alpha_{g+h}].$$
Thus the map $\alpha:G\to \Aut(M)$ is a $G$-kernel (see \cite{MR587749}), and it gives rise to an obstruction class 
in $H^3(G,\T)$, which is identified with the associator of the pointed category $\Inv(\cC)$. 
Furthermore, we assume that there exists a self-dual object $\rho\in \cC$ such that
\begin{equation}\label{irreducibles}
\Irr(\cC)=\{[\alpha_g]\}_{g\in G}\sqcup \{[\alpha_g\circ \rho]\}_{g\in G},
\end{equation}
with the fusion rules: 
\begin{equation}\label{fusion1}
[\alpha_g] [\rho]=[\rho][\alpha_{-g}],\quad g\in G,
\end{equation}
\begin{equation}\label{fusion2}
[\rho]^2=[\id]+ \sum_{g\in G}[\alpha_g\circ \rho].
\end{equation}

We consider a $\Z_2$-action on $G$ defined by multiplying by $-1$. 
Then the third cohomology obstruction $\mathfrak{c}^{0,3}(\cC)$ for $\alpha$ actually belongs to 
$H^3(G,\T)^{\Z_2}$ (see \cite[Lemma 2.5]{MR3827808}). 
Assume that it vanishes. 
Then we can choose $\alpha$ to be a $G$-action on $M$. 
From the fusion rules above, we see that there exist unitaries $u_g\in M$ for each $g\in G$ satisfying 
$$\Ad u_g\circ\alpha_g\circ \rho=\rho\circ \alpha_{-g}.$$
Since $\alpha$ is a $G$-action, we get 
\begin{align*}
\lefteqn{\Ad u_{g+h}\circ\alpha_{g+h}\circ \rho=\rho\circ \alpha_{g+h}=\rho\circ\alpha_g\circ\alpha_h} \\
 &=\Ad u_g\circ \alpha_g\circ\rho\circ \alpha_h=\Ad(u_g\alpha_g(u_h))\circ \alpha_g\circ \alpha_h\circ \rho 
=\Ad(u_g\alpha_g(u_h))\circ \alpha_{g+h}\circ \rho, 
\end{align*}
and there exists a 2-cocycle $\omega\in Z^2(G,\T)$ satisfying 
$$u_g\alpha_g(u_h)=\omega(g,h)u_{g+h}.$$
The cohomology class $[\omega]\in H^2(G,\T)$ in fact gives a class in $H^1(\Z_2,H^2(G,\T))$ (see \cite[Lemma 2.6]{MR3827808}), 
which we denote by $\mathfrak{c}^{1,2}(\cC)$. 

\begin{definition} Let $\cC\subset \End(M)$ be a fusion category with abelian $G$ satisfying 
Eq.(\ref{irreducibles}), (\ref{fusion1}), (\ref{fusion2}). 
If the cohomology classes $\mathfrak{c}^{0,3}(\cC)$ and $\mathfrak{c}^{1,2}(\cC)$ vanish, we say that $\cC$ is a 
generalized Haagerup category. 
\end{definition}

Generalized Haagerup categories are completely classified by the solutions of the following polynomial equations 
up to an appropriate equivalence relation (see \cite[Theorem 5.7]{MR3827808}). 
Let $d=\frac{|G|+\sqrt{|G|^2+4}}{2}$, where $|G|$ is the order of $G$. 
We consider $\epsilon_h(g)\in \{1,-1\}$, $\eta_g\in \T$, and $A_g(h,k)\in \C$ for $g,h,k\in G$ 
satisfying the following condition: 
\begin{equation}\label{cocycle}\epsilon_{h+k}(g)=\epsilon_h(g)\epsilon_k(g+2h),\quad \epsilon_h(0)=1,
\end{equation}
\begin{equation}\label{eta}
\eta_{g+2h}=\eta_g,\quad \eta_g^3=1,
\end{equation}
\begin{equation}\label{Orth1}
\sum_{h\in G}A_g(h,0)=-\frac{\overline{\eta_g}}{d},
\end{equation}
\begin{equation}\label{Orth2}
\sum_{h\in G}A_g(h-g,k)\overline{A_{g'}(h-g',k)}=\delta_{g,g'}-\frac{\overline{\eta_g}\eta_{g'}}{d}\delta_{k,0},
\end{equation}
\begin{equation}\label{2hshift}
A_{g+2h}(p,q)=\epsilon_h(g)\epsilon_h(g+p)\epsilon_h(g+q)\epsilon_h(g+p+q)A_g(p,q),
\end{equation}
\begin{equation}\label{CC}
A_g(k,h)=\overline{A_g(h,k)},
\end{equation}
\begin{align}\label{Z3symm}
A_g(h,k)&=A_g(-k,h-k)\eta_g\epsilon_{-k}(g+h)\epsilon_{-k}(g+k)\epsilon_{-k}(g+h+k)\\
&=A_g(k-h,-h)\overline{\eta_g}\epsilon_{-h}(g+h)\epsilon_{-h}(g+k)\epsilon_{-h}(g+h+k),\nonumber
\end{align}
\begin{align}\label{hkshift}
A_g(h,k)&=A_{g+h}(h,k)\eta_g\eta_{g+k}\overline{\eta_{g+h}}\overline{\eta_{g+h+k}}\epsilon_h(g)\epsilon_h(g+k)\\
&=A_{g+k}(h,k)\overline{\eta_g\eta_{g+h}}\eta_{g+k}\eta_{g+h+k}\epsilon_k(g)\epsilon_k(g+h),\nonumber
\end{align}
\begin{align}\label{AAA}
\lefteqn{
\sum_{l\in G}A_g(x+y,l)A_{g-p+x}(-x,l+p)A_{g-q+x+y}(-y,l+q)} \\
&=A_g(p+x,q+x+y)A_{g-p}(q+y,p+x+y)\nonumber\\
&\times \eta_g\eta_{g+q+x}\eta_{g+p+q+y}\overline{\eta_{g+p}\eta_{g+x+y}\eta_{g+q+x+y}}\nonumber \\ 
&\times \epsilon_p(g-p+x)\epsilon_{p+x}(g-p+q+y)\epsilon_q(g-q+x+y)\epsilon_{q+y}(g-q+x)\nonumber\\
&-\frac{\delta_{x,0}\delta_{y,0}}{d}\eta_g\eta_{g+p}\eta_{g+q}.\nonumber
\end{align}

These numerical invariants arise as follows (see \cite[Section 3]{MR3827808}). 
Let $\cC\subset \End(M)$ be a generalized Haagerup category with $G$. 
Since $\mathfrak{c}^{1,2}(\cC)=0$, we can choose $\rho$ and $\alpha_g$ satisfying the relation 
$$\alpha_g\circ \rho=\rho\circ \alpha_{-g}.$$
Let $G_2=\{g\in G;2g=0\}$. 
Then thanks to the above relation, we have $\alpha_g((\rho,\rho^2))=(\rho,\rho^2)$ for $g\in G_2$. 
By replacing $\rho$ with an equivalent endomorphism, we may further assume that $\alpha_g$ for $g\in G_2$ 
acts on $(\rho,\rho^2)$ trivially. 
Then we can choose bases of intertwiner spaces consisting of isometries $s\in (\id, \rho^2)$, $t_g\in (\alpha_g\circ \rho,\rho^2)$ 
satisfying 
\begin{equation}\label{alpha}
\alpha_h(s)=s,\quad \alpha_h(t_g)=\epsilon_h(g)t_{g+2h},
\end{equation}
\begin{equation}\label{rho1}
\rho(s)=\frac{1}{d}s+\frac{1}{\sqrt{d}}\sum_{g\in G}t_gt_g,
\end{equation}
\begin{equation}\label{rho2}
\alpha_g\circ \rho(t_g)=\eta_gt_gss^*+\frac{\overline{\eta_g}}{\sqrt{d}}st_g^*+
\sum_{h,k\in G}A_g(h,k)t_{g+h}t_{g+h+k}t_{g+k}^*.
\end{equation}

Recall that the Cuntz algebra $\cO_n$ for an integer $n$ larger than 1 is the universal C$^*$-algebra with generators 
$\{s_i\}_{i=1}^n$ satisfying the relations 
$$s_i^*s_j=\delta_{i,j}1,$$
$$\sum_{i=1}^ns_is_i^*=1.$$
Note that the above isometries $\{s\}\cup\{t_g\}_{g\in G}$ satisfy the $\cO_{|G|+1}$-relation. 

Conversely, assume that we are given a solution $(\epsilon_h(g),\eta_g,A_g(h,g))$ of Eq.(\ref{cocycle})-(\ref{AAA}) 
(without knowing that it comes from a generalized Haagerup category). 
We consider the Cuntz algebra $\cO_{|G|+1}$ with the canonical generators $\{s\}\cup\{t_g\}_{g\in G}$. 
Then we can introduce a $G$-action $\alpha$ on $\cO_{|G|+1}$ and an endomorphism $\rho$ of $\cO_{|G|+1}$ 
by Eq.(\ref{alpha})-(\ref{rho2}), which satisfy $\alpha_g\circ \rho=\rho\circ \alpha_{-g}$ and 
$$\rho^2(x)=sxs^*+\sum_{g\in G}t_g\alpha_g\circ \rho(x)t_g^*.$$ 
Taking the weak closure of $\cO_{|G|+1}$ in an appropriate representation, we get a type III factor $M$ 
and a generalized Haagerup category $\cC\subset \End(M)$ generated by (extensions of) $\alpha_g$ and $\rho$ 
(see \cite[Theorem 4.1, Theorem 10.2]{MR3827808}). 

From now on, whenever we discuss a generalized Haagerup category $\cC$ with a finite abelian group $G$, 
we choose and fix $\alpha_g$, 
$\rho$, and $\{s\}\cup\{t_g\}_{g\in G}\subset M$ satisfying Eq.(\ref{alpha})-(\ref{rho2}). 

When a generalized Haagerup category comes from a generalized Haagerup subfactor (called a $3^G$ subfactor in \cite{MR3827808}), the object 
$\id\oplus \rho$ has a $Q$-system. 
It is shown in \cite[Section 7]{MR1832764} that $\id\oplus \rho$ has a $Q$-system if and only if the following holds:
\begin{equation}\label{Q1}
A_0(h,0)=\delta_{h,0}-\frac{1}{d-1}.
\end{equation}

In concrete examples, it is often the case that a solution of 
(\ref{cocycle})-(\ref{AAA}) and (\ref{Q1}) automatically satisfies 
\begin{equation}\label{Q2}
A_g(h,0)=\delta_{h,0}-\frac{1}{d-1}, 
\end{equation}
for any $g,h\in G$. 
In other words, once $\id\oplus \rho$ has a $Q$-system, so does any other $\id\oplus \alpha_g\circ \rho$ 
in known examples.

Under Eq.(\ref{Q2}), we get $\eta_g=1$ from Eq.(\ref{Z3symm}), and Eq.(\ref{2hshift}) implies that 
the map $G\ni g\mapsto \epsilon_h(g)\in \{1,-1\}$ is a character for any $h\in G_2$. 
In particular, it makes sense to say that $(\epsilon_h(g),A_g(h,k))$ is a solution of Eq.(\ref{cocycle})-(\ref{AAA}) and Eq.(\ref{Q2}) 
omitting $\eta_g$. 

Solutions to Eq.(\ref{cocycle})-(\ref{AAA}) and Eq.(\ref{Q2})  have been computed for a number of small groups. We list the known solutions, up to group automorphism and gauge equivalence. For $G=\Z_2$ there is a unique solution corresponding to the even part of the $A_7$ subfactor. For $ \Z_3$ there is a unique solution, corresponding to the Haagerup subfactor. In \cite{MR1832764} a unique solution was found for $\Z_5 $. In \cite{MR2837122}  Evans and Gannon found a unique solution for $\Z_7 $, exactly two solutions for $\Z_9 $, and no solutions for $ \Z_3 \times \Z_3$; they also found numerical evidence for solutions for several larger odd cyclic groups. In \cite{MR3827808} unique solutions were found for $\Z_4 $ and $\Z_2 \times \Z_2 $. In \cite{MR3859276} a solution was found for $\Z_4 \times \Z_2 $. In working through the examples in this paper we found that there exactly two solutions each for $\Z_6 $ and $\Z_{10} $, as well as at least one solution for $\Z_8 $; these solutions are contained in the accompanying Mathematica notebook \texttt{solutions.nb}. 

\subsection{(de)-Equivariantization and orbifolds}
Let $G$ is a finite group acting on a fusion category $\cC $by tensor autoequivalences. Then one can define the notion of a $G$-equivariant object in $ \cC$. The category of $G$-equivariant objects is a fusion category, called an equivariantization of $\cC $ by $G$, and denoted $\cC^G $. There is an inverse construction called de-equivariantization, by which $\cC $ may be recovered from $\cC^G $. We refer the reader to \cite{MR2609644} for details.
In addition to the equivariantization (which can be thought of as taking the ``fixed points'' of the $G$-action), one can also construct the crossed product $\cC \rtimes G $. The crossed product is a quasi-trivial $G$-graded extenension of $\cC $ which is Morita equivalent to $\cC^G $ \cite{MR2480712}.

For a fusion category $ \cC$ embedded in $\End(M) $, both equivariantization and de-equivariantization can sometimes be realized by an orbifold construction, in which the von Neumann $M$ is enlarged to a crossed product by a group action, and the endomorphisms/objects of $ \cC$ are extended to the larger algebra.

For generalized Haagerup categories, two specific types of orbifolds were constructed in \cite[Section 8]{MR3827808}. We briefly recall these constructions, and refer the reader there for more details.

\subsubsection{De-equivariantization}
Let $\mathcal{C}$ be a generalized Haagerup category embedded in $\End(M)$. Suppose there is a $z \in G_2 $ such that $\epsilon_z(\cdot) $ is a character satisfying $\epsilon_z(z)=1 $. Let $P=M \rtimes_{\alpha_z} \Z_2  $ be the crossed
product of $M$ by $ \alpha_z$. Then $P$ is the von Neumann algebra generated by $M$ and a unitary $\lambda $ satisfying $\lambda^2=1 $ and $$\lambda x \lambda^{-1}=\alpha_z(x), \forall x \in M .$$  Each $ \alpha_g$ can be extended to an automorphism
$ \tilde{\alpha}_g$ of $P$ by setting 
$$ \tilde{\alpha}_g(\lambda)=\epsilon_z(g) \lambda .$$ 
Similarly, $\rho $ can be extended to an endomorphism $\tilde{\rho} $ of $P$ by 
setting $$\tilde{\rho} (\lambda) =\lambda.$$ Then $g \mapsto \tilde{\alpha}_g$ defines an action of $G$ on $P$,  
and we have $$\tilde{\alpha}_g \circ \tilde{\rho}=\tilde{\rho} \circ \tilde{\alpha}_{-g}, \ \forall g \in G .$$

Moreover, $$ [\tilde{\alpha}_g]=[\tilde{\alpha}_h] \text{ iff } g-h \in \{ 0,z\} ,$$ and if $G_0 \subset G$ is a set of representative elements
for the $\{0,z\} $-cosets of $G$, we have $$[\tilde{\rho}^2] = 
[id]\bigoplus_{g \in G_0} 2[\tilde{\alpha}_{g} \tilde{\rho}]  .$$ 

The fusion category in $\End(P) $ tensor generated by $\tilde{\rho} $ is a $\Z_2 $-de-equivariantization of $\cC$.

\subsubsection{Equivariantization} Let $\mathcal{C}$ again be a generalized Haagerup category, and let $\theta $ be an automorphism of $G$ which preserves the structure of $\cC $:
 $$\epsilon_{\theta(h)}(\theta(g))=\epsilon_h(g),  $$ $$\eta_{\theta(g)} =\eta(g),$$ $$  A_{\theta(g)}(\theta(h),\theta(k))=A_g(h,k), \ \forall g,h,k \in G .$$
Define an automorphism $\gamma $ on $M$ by $\gamma(s)=s$ and $\gamma(t_g)=t_{\theta(g)} $, $g \in G $.  Then $$\gamma \circ \rho = \rho \circ \gamma  $$ and $$\gamma \circ \alpha_g  = \alpha_{\theta(g)} \circ \gamma .$$ The automorphism $\gamma $ thus induces an action of $\Z_m $ on $\mathcal{C} $, where $m$ is the order of $\theta $. Let $P=M \rtimes_{\gamma} \Z_m $ be the crossed
product of $M$ by $ \gamma$. Then $P$ is the von Neumann algebra generated by $M$ and a unitary $\lambda $ satisfying $$\lambda^m=1 \text{ and }  \lambda x \lambda^{-1}=\gamma(x), \ \forall x \in M .$$ We can extend $\rho $ to an endomorphism $\tilde{\rho} $ of $P $ by setting $$\tilde{\rho}(\lambda)=\lambda .$$ Then the category $\mathcal{C}^{\gamma}$ in  $\End(P)$ tensor generated by $\tilde{\rho} $ is a $\Z_m$-equivariantization of $\mathcal{C}$. The corresponding crossed product category is the category in $\End(M) $ tensor generated by $ \cC$ and $\gamma $.

%
%

%
%

\section{Restriction of $G_2$} \label{restrict}
We fix a generalized Haagerup category $\cC\subset \End(M)$ with a finite abelian group $G$. 
We choose $\alpha$, $\rho$, $s$, and $t_g$ satisfying 
Eq.(\ref{cocycle})-(\ref{rho2}) for $\cC$. 

Recall $G_2=\{g\in G;\; 2g=0\}$. 
We denote by $\hG$ the dual group of $G$. 
Let $(\hG)_2=\{\chi\in \hG;\; 2\chi=0\}$. 
We choose $G_*\subset G$ and $\hG_*\subset \hG$ satisfying 
$$G=G_2\sqcup G_*\sqcup (-G_*),$$
$$\hG=(\hG)_2\sqcup \hG_*\sqcup (-\hG_*).$$
Let $\Lambda$ be the global dimension of $\cC$:
$$\Lambda=|G|+|G|d^2=|G|(2+|G|d)=|G|\frac{|G|^2+4+|G|\sqrt{|G|^2+4}}{2}.$$
Letting $a=1/|G|$ and $b=1/\sqrt{|G|^2+4}$, we get a simple expression 
$$\frac{1}{\Lambda}=\frac{a-b}{2}.$$
To avoid heavy notation, we often omit $\alpha$ and write $\cA_g$ for $\cA_{\alpha_g}$, and similarly
${}_g\rho$ for $\alpha_g\circ \rho$.

\begin{lemma} \label{mulfree1}
The object $\alpha_g\circ \rho$ is multiplicity free (i.e. the algebra $\cA_{_g\rho}$ is abelian), for any $g\in G$. 
\end{lemma}

\begin{proof} It suffices to show the statement for $g=0$ because the pair $(\cC,\alpha_g\circ\rho)$ satisfies the same condition 
as the pair $(\cC,\rho)$. 
Since $\rho$ is self-conjugate, the map $S_0^2$ restricted $\cA_\rho$ is regarded as an algebra isomorphism from $\cA_\rho$ to its opposite algebra. 
On the other hand, we claim that it is the identity on $\cA_\rho$, which would show that $\cA_\rho$ is abelian. 

Indeed $\alpha_k(s)=s$ implies that $S_0^2$ acts on $(\rho\;k|1|k\;\rho)$ and $(\rho\;_k\rho|ss^*|_k\rho\;\rho)$ as the identity. 
We also have 
\begin{align*}
\lefteqn{S_0^2((\rho\;_g\rho|t_{h+g}t_{h-g}^*|_g\rho\;\rho))} \\
 &=dS_0((_g\rho\;\rho|s^*\alpha_g\circ \rho(t_{h+g}t_{h-g}^*\rho(s))|\rho\;_g\rho)) \\
 &=\sqrt{d}S_0((_g\rho\;\rho|s^*\alpha_g\circ \rho(t_{h+g}t_{h-g})|\rho\;_g\rho)) \\
 &=d^{3/2}((\rho\;_g\rho|s^*\rho(s^*\alpha_g\circ \rho(t_{h+g}t_{h-g}s))|_g\rho\;\rho)),\\
\end{align*}
and

\begin{align*}
\lefteqn{d^{3/2}s^*\rho(s^*\alpha_g\circ \rho(t_{h+g}t_{h-g}s))} \\
 &=d^{3/2}\epsilon_{-h}(g+h)s^*\rho(s^*\alpha_{g-h}\circ\rho(t_{g-h})\alpha_g\circ\rho(t_{h-g}s)) \\
 &=d\epsilon_{-h}(g+h)s^*\rho(t_{g-h}^*)\rho\circ\alpha_g\circ\rho(t_{h-g}s)  \\
 &=d\epsilon_{-h}(g+h)\epsilon_{g-h}(h-g) s^*\alpha_{h-g}\circ\rho(t_{h-g}^*)\rho^2\circ\alpha_{-g}(t_{h-g}s) \\
 &=d\epsilon_g(h-g) s^*t_{h-g}^*\rho^2\circ\alpha_{-g}(t_{h-g}s)\\
 &=d\epsilon_g(h-g) s^*\alpha_h\circ\rho(t_{h-g}s)t_{h-g}^*\\
 &=d s^*\alpha_{h+g}\circ\rho(t_{h+g}s)t_{h-g}^*\\
 &=d^{1/2} t_{h+g}^*\rho(s)t_{h-g}^*\\
 &=t_{h+g}t_{h-g}^*.
\end{align*}
 Since elements of the above form span $\cA_{\rho}$, the claim is shown.
\end{proof}

In what follows, we assume Eq.(\ref{Q2}). 
It is straightforward to show the following.  

\begin{lemma}\label{lifts} For $k\in G_2$, the object $\alpha_k$ has a unique half-braiding $\cE_k$, and it is given by 
$\cE_k(g)=\epsilon_k(g)$ and $\cE_k(_g\rho)=\epsilon_k(g)$.  
\end{lemma}

The above lemma implies that $\alpha_k$ for $k\in G_2$ has a unique extension $\widetilde{\alpha_k}$ to the Drinfeld center $\cZ(\cC)$, 
which we often denote by $k$ for simplicity. 
Thanks to Eq.(\ref{mincenpr}), the corresponding minimal central projection in $\Tube \cC$ is   
$$z(k)=\frac{1}{\Lambda}\sum_{g\in G}(\epsilon_k(g)(k\;g|1|g\;k)+d\epsilon_k(g)(k\;_g\rho|1|_g\rho\;k))\in \cA_k.$$

We determine the algebra structure of $\cA_k$ for $k\in G_2$ and 
$$\cB_g:=\cA_g\oplus\cA_{g,-g}\oplus \cA_{-g,g}\oplus \cA_{-g}$$
for $g\in G_*$ following \cite{MR1832764}. 
For $g\in G$ and $\tau\in \hG$, let 
$$p(g,\tau)=\frac{1}{|G|}\sum_{g\in G}\inpr{h}{\tau}(g\;h|1|h\;g)\in \cA_g,$$
which is a projection in $\cA_g$.   
We have 
$$p(g,\tau)(g\;\rho|1|\rho\;-g)=(g\;\rho|1|\rho\;-g)p(-g,-\tau).$$
Recall that for $k\in G_2$, the map $G\ni g\mapsto \epsilon_k(g)\in \{1,-1\}$ is a character, 
and we can regard $\epsilon_k$ as an element in $(\hG)_2$. 
Thus we have the following expression 
$$z(k)=\frac{|G|}{\Lambda}p(k,\epsilon_k)(1_k+d(k\; \rho|1|\rho\;k))=\frac{|G|}{\Lambda}(1_k+d(k\; \rho|1|\rho\;k))p(k,\epsilon_k).$$

We first determine the structure of $\cA_0$. 
For $\chi\in (\hG)_2$, the projection $p(0,\chi)$ is central in $\cA_0$, and 
$p(0,\chi)\cA_0$ is a 2-dimensional algebra spanned by $p(0,\chi)$ and $p(0,\chi)(0\;\rho|1|\rho\;0)$ with 
\begin{align*}
\lefteqn{(p(0,\chi)(0\;\rho|1|\rho\;0))^2} \\
 &=p(0,\chi)+\sum_{g\in G}p(0,\chi)(0\;_g\rho|1|_g\rho\;0)\\
 &=p(0,\chi)+\sum_{g\in G}\inpr{g}{\chi}p(0,\chi)(0\;\rho|1|\rho\;0) \\
 &=
\left\{
\begin{array}{ll}
p(0,0)+|G|p(0,0)(0\;\rho|1|\rho\;0) , &\quad \chi=0 \\
p(0,\chi) , &\quad \chi\neq 0
\end{array}
\right..
\end{align*}

Let 
$$E(0,0)=\frac{|G|d}{\Lambda}p(0,0)(d1_0-(0\;\rho|1|\rho\;0)),$$
which is a minimal projection in $\cA_0$ orthogonal to $z(0)$, and to 
$p(0,\tau)$ for any $\tau\in \hG\setminus \{0\}$. 

For $\chi\in (\hG)_2\setminus\{0\}$, we set 
$$E(0,\chi)_{\pm}=\frac{1}{2}p(0,\chi)(1_0\pm (0\;\rho|1|\rho\;0)),$$
which are minimal central projections in $\cA_0$. 

For $\tau\in \hG_*$, we set 
$$E(0,\tau)_{11}=p(0,\tau),\quad E(0,\tau)_{22}=p(0,-\tau),$$
$$E(0,\tau)_{12}=p(0,\tau)(0\;\rho|1|\rho\;0),$$
$$E(0,\tau)_{21}=p(0,-\tau)(0\;\rho|1|\rho\;0),$$
and set $\cA_0^\tau=\mathrm{span}\{E(0,\tau)_{ij}\}_{1\leq i,j,\leq 2}$. 
Then $\cA_0^{\tau}$ is isomorphic to the 2 by 2 matrix algebra with a system 
of matrix units $\{E(0,\tau)_{ij}\}$. 
Then we get the following decomposition of $\cA_0$ as an algebra: 
$$\cA_{0}=\C z(0)\oplus \C E(0,0)\oplus 
\bigoplus_{\chi\in (\hG)_2\setminus\{0\}}(\C E(0,\chi)_+\oplus \C E(0,\chi)_{-})
\oplus \bigoplus_{\tau\in G_*}\cA^\tau_0.$$

For $k\in G_2$, the algebra $\cA_k$ has a similar decomposition as $k$ extends to $\cZ(\cC)$. 
For $\tau\in \hG_*$, we set 
$$E(k,\tau)_{11}=p(k,\tau),\quad E(k,\tau)_{22}=p(k,-\tau),$$
$$E(k,\tau)_{12}=p(k,\tau)(k\;\rho|1|\rho\;k),$$
$$E(k,\tau)_{21}=p(k,-\tau)(k\;\rho|1|\rho\;k),$$
and set $\cA_k^\tau=\mathrm{span}\{E(k,\tau)_{ij}\}_{1\leq i,j,\leq 2}$. 
Then $\cA_k^{\tau}$ is a simple component of $\cA_k$ isomorphic to the 2 by 2 matrix algebra with a system 
of matrix units $\{E(k,\tau)_{ij}\}_{ij}$. 
 
For $g\in G\setminus G_2$, we have decomposition
$$\cA_g=\bigoplus_{\tau \in \hG}\C p(g,\tau).$$
For $\tau\in \hG$, we set 
$$E(g,\tau)_{11}=p(g,\tau),\quad E(g,\tau)_{22}=p(-g,-\tau),$$
$$E(g,\tau)_{12}=p(g,\tau)(g\;\rho|1|\rho\;-g)=(g\;\rho|1|\rho\;-g)p(-g,-\tau),$$
$$E(g,\tau)_{21}=p(-g,-\tau)(-g\;\rho|1|\rho\;g)=(-g\;\rho|1|\rho\;g)p(g,\tau).$$
Let $\cB^\tau_g=\mathrm{span}\{E(g,\tau)_{ij}\}_{ij}$. 
Then $\cB^\tau_g$ is isomorphic to the 2 by 2 matrix algebra with a system of matrix units $\{E(g,\tau)_{ij}\}_{ij}$, and 
$$\cB_g=\bigoplus_{\tau\in \hG}\cB^\tau_g.$$

Next we determine the action of $\cA_0$ on $\cA_{0,_g\rho}$. 
Since $\cA_{_g\rho}$ is multiplicity free, Lemma \ref{mulfree1} shows that every irreducible $\cA_0$-module 
can appear in $\cA_{0,_g\rho}$ with multiplicity at most one. 
Thus it suffices to known whether each simple component of $\cA_0$ acts on $\cA_{0,_g\rho}$ trivially or not.   
Since $\cA_{0,_{g\rho}}$ has a basis $\{(0\;h|1|h\;0)(0\;\rho|t_{-g}|\rho\;_{g}\rho)\}_{h\in G}$, 
it also has a basis $\{p(0,\tau)(0\;\rho|t_{-g}|\rho\; _g\rho)\}_{\tau \in \hG}$. 
Since $z(0)$ is central in $\Tube \cC$, it acts on $\cA_{0,_g\rho}$ trivially. 
The algebras $\C E(0,0)$ and $\cA_0^\tau$ for $\tau\in \hG_*$ act on $\cA_{0,_g\rho}$ non-trivially. 
It remains to show how $\C E(0,\chi)_{\pm}$ acts on $p(0,\chi)(0\;\rho|t_{-g}|\rho\; _g\rho)$ for $\chi\in (\hG)_2\setminus \{0\}$. 

\begin{lemma} For $\chi\;\in (\hG)_2\setminus \{0\}$, 
$$(0\;\rho|1|\rho\;0)p(0,\chi)(0\;\rho|t_{-g}|\rho\; _g\rho)=\inpr{g}{\chi}p(0,\chi)
(0\;\rho|t_{-g}|\rho\; _g\rho).$$
\end{lemma}

\begin{proof} 
Since $(0\;\rho|0|\rho\;0)$ commutes with $p(0,\chi)$ as $\chi \in (\hG)_2$, 
\begin{align*}
\lefteqn{(0\;\rho|1|\rho\;0)p(0,\chi)(0\;\rho|t_{-g}|\rho\; _g\rho)} \\
 &=p(0,\chi)(0\;\rho|1|\rho\;0)(0\;\rho|t_{-g}|\rho\; _g\rho) \\
 &=p(0,\chi)\sum_{h\in G}(0\;_h\rho|t_h^*\rho(t_{-g})t_h|_h\rho\;_g\rho)\\
 &=p(0,\chi)\sum_{h\in G}A_g(h-g,h-g)\epsilon_{-g}(g)(0\;_h\rho|t_{2h-g}|_h\rho\;_g\rho)\\
 &=p(0,\chi)\sum_{h\in G}A_g(h-g,h-g)\epsilon_{-g}(g)\epsilon_{h}(-g)(0\;h|1|h\;0)(0\;\rho|t_{-g}|\rho\;_g\rho)\\
 &=\sum_{h\in G}A_g(h-g,h-g)\epsilon_{h-g}(g)\inpr{h}{\chi}p(0,\chi)(0\;\rho|t_{-g}|\rho\;_g\rho)\\
 &=\inpr{g}{\chi}\sum_{h\in G}A_g(h,h)\epsilon_{h}(g)\inpr{h}{\chi}p(0,\chi)(0\;\rho|t_{-g}|\rho\;_g\rho).
\end{align*}
Eq.(\ref{Z3symm}) and Eq.(\ref{Q2}) imply 
$$\sum_{h\in G}A_g(h,h)\epsilon_{h}(g)\inpr{h}{\chi}=\sum_{h\in G}(\delta_{h,0}-\frac{1}{d-1})\inpr{h}{\chi}=1.$$
\end{proof}

The above lemma shows 
$$E(0,\chi)_{\inpr{g}{\chi}}p(0,\chi)(0\;\rho|t_{-g}|\rho\; _g\rho)=p(0,\chi)(0\;\rho|t_{-g}|\rho\; _g\rho),$$
and 
$$E(0,\chi)_{-\inpr{g}{\chi}}\cA_{0,_g\rho}=\{0\}.$$

For $h\in G\setminus G_2$, we can easily determine the $\cA_h$-action on $\cA_{h,_g\rho}$ from Lemma \ref{mulfree} 
as we have $\dim \cA_h=\dim \cA_{h,_g\rho}=|G|$. 
Namely, every simple component $\C p(h,\tau)$ acts on $\cA_{h,_g\rho}$ with multiplicity one. 

Summing up the above argument, we get 
\begin{lemma} Let the notation be as above, and let $k\in G_2$. 
\begin{itemize}
\item [(1)] Let 
$$\pi=\id\oplus \bigoplus_{g\in G}\alpha_g\circ \rho.$$
The object $\pi$ has a unique half-braiding, which gives $e(\widetilde{\pi})_{0,0}=E(0,0)$, and 
$$\cE_\pi(h)_{0,0}=1,$$
$$\cE_\pi(_h\rho)_{0,0}=-\frac{1}{d^2}.$$
More generally $\alpha_k\pi$ has a unique half-braiding $\cE_{k\pi}(\xi)=\cE_k(\xi)\alpha_k(\cE_{\pi}(\xi))$.  

\item [(2)] 
For $\chi\in (\hG)_2\setminus\{0\}$, let  
$$\varphi_{\chi,\pm}=\id\oplus\bigoplus_{\inpr{g}{\chi}=\pm1}\alpha_g\circ \rho.$$
The object $\varphi_{\chi,+}$ has a unique half-braiding, which gives $e(\widetilde{\varphi_{\chi,+}})_{0,0}=E(0,\chi)_+$, 
and $\varphi_{\chi,-}$ has a unique half-braiding, which gives $e(\widetilde{\varphi_{\chi,-}})_{0,0}=E(0,\chi)_-$.
The corresponding half-braidings are  
$$\cE_{\varphi_{\chi,\pm}}(h)_{0,0}=\inpr{h}{\chi},$$
$$\cE_{\varphi_{\chi,\pm}}(_h\rho)_{0,0}=\frac{\pm\inpr{h}{\chi}}{d}.$$
More generally $\alpha_k\varphi_{\chi,\pm}$ has a unique half-braiding $\cE_{k\varphi_{\xi,\pm}}(\xi)=\cE_k(\xi)\alpha_k(\cE_{\varphi_{\chi,\pm}}(\xi))$. 

\item [(3)] Let 
$$\sigma_k=\alpha_k\oplus \alpha_k \oplus \bigoplus_{g\in G}\alpha_g\circ \rho.$$
The object $\sigma_k$ has exactly $(|G|-|G_2|)/2$ half-braidings parametrized by $\tau\in \hG_*$, 
which give $e(\widetilde{\sigma_k}^\tau)_{(k,i),(k,j)}=E(k,\tau)_{ij}$, and 
$$\cE_{\sigma_k}^\tau(h)_{(k,1),(k,1)}=\inpr{h}{\tau},$$
$$\cE_{\sigma_k}^\tau(_h\rho)_{(k,i),(k,i)}=0.$$ 

\item [(4)] For $g\in G_*$, let 
$$\sigma_g=\alpha_g\oplus \alpha_{-g}\oplus\bigoplus_{h\in G}\alpha_h\circ \rho.$$
The object $\sigma_g$ has exactly $|G|$ half-braidings parametrized by $\tau\in \hG$, which give 
$e(\widetilde{\sigma_g}^\tau)_{g,g}=p(g,\tau)$, and 
$$\cE_{\sigma_g}^\tau(h)_{g,g}=\inpr{h}{\tau},$$
$$\cE_{\sigma_g}^\tau(_h\rho)_{g,g}=0.$$ 
\end{itemize}
\end{lemma}

The algebra $\cA_\rho$ has a basis 
$$\{(\rho\;k|1|k\;\rho)\}_{k\in G_2}\cup\{(\rho\;_k\rho|ss^*|_k\rho\;\rho)\}_{k\in G}\cup
\{(\rho\;_g\rho|t_{h+g}t_{h-g}^*|_g\rho\;\rho)\}_{g,h},$$
and $\dim \cA_\rho=|G|^2+2|G_2|$. 
In a similar way, we can see 
$$\dim\cA_{_g\rho,\rho}=\left\{
\begin{array}{ll}
|G|^2+2|G_2| , &\quad g\in 2G  \\
|G|^2 , &\quad g\in G\setminus 2G
\end{array}
\right..
$$
There exists an invertible element $x=(\rho\;-g|1|-g\;_{2g}\rho)\in \cA_{\rho,_{2g}\rho}$ in the sense that it satisfies 
$x^*x=1_{_{2g}\rho}$ and $xx^*=1_\rho$. 

\begin{theorem} Let $\cC\subset \End(M)$ be a generalized Haagerup category with a finite abelian group $G$ 
satisfying Eq.(\ref{Q2}). 
Then $|G_2|\leq 4$. 
\end{theorem}

\begin{proof} If $G$ is an odd group, we have nothing to prove, and we assume that $G$ is an even group. 
Let
\begin{align*}
\lefteqn{Z_G} \\
 &=\sum_{k\in G_2}\big(z(k)+z(\widetilde{k\pi})+\sum_{\chi\in ((\hG)_2)\setminus \{0\}}
(z(\widetilde{k\varphi_{\chi,+}})+z(\widetilde{k\varphi_{\chi,-}})\big)+\sum_{\tau\in \hG_*}z(\widetilde{\sigma_k}^\tau)\big) \\
 &+\sum_{g\in G_*}\sum_{\tau\in \hG}z(\widetilde{\sigma_g}^\tau).
\end{align*}
Then 
\begin{align*}
\lefteqn{\dim (1-Z_G)\cA_\rho=\dim \cA_\rho-|G_2|(1+|(\hG)_2|-1+|\hG_*|)-|G_*||\hG|} \\
 &=|G|^2+2|G_2|-(|G_2|^2+(|G|+|G_2|)\frac{|G|-|G_2|}{2})=|G|^2+2|G_2|-\frac{|G|^2+|G_2|^2}{2} \\
 &=\frac{|G|^2}{2}-\frac{|G_2|^2}{2}+2|G_2|.
\end{align*}
On the other hand, if $g\in G\setminus 2G$, we get 
\begin{align*}
\lefteqn{\dim (1-Z_G)\cA_{_g\rho,\rho}}\\
&=\dim \cA_{_g\rho,\rho}
-|G_2|(1+\#\{\chi\in (\hG)_2\setminus \{0\};\;\inpr{g}{\chi}=\inpr{0}{\chi}\}+|\hG_*|)-|G_*||\hG| \\
 &=|G|^2-|G_2|(1+\#\{\chi\in (\hG)_2\setminus \{0\};\;\inpr{g}{\chi}=1\}-\frac{|G|^2-|G_2|^2}{2}\\
 &=\frac{|G|^2+|G_2|^2}{2}-|G_2|(1+\#\{\chi\in (\hG)_2\setminus \{0\};\;\inpr{g}{\chi}=1\}.
\end{align*}
Note that the group $(\hG)_2$ is identified with the dual group of $G/2G\cong \Z_2^m$, and 
$g+2G$ in $G/2G$ is not 0. 
Thus we get 
$$\#\{\chi\in (\hG)_2\setminus \{0\};\;\inpr{g}{\chi}=1\}=\frac{|(\hG)_2|}{2}-1=\frac{|G_2|}{2}-1,$$
and $\dim (1-Z_G)\cA_{_g\rho,\rho}=\frac{|G|^2}{2}$. 
Since $\rho$ and $\alpha_g\circ \rho$ are multiplicity free, Lemma \ref{mulfree} implies  
$$\dim(1-Z_G)\cA_{_g\rho,\rho}\leq \dim (1-Z_G)\cA_\rho.$$ 
This is possible only if $|G_2|\leq 4$. 
\end{proof}

The above computation shows the following.  
If $|G_2|=2$, we have 
$$\dim(1-Z_G)\cA_\rho=\frac{|G|^2}{2}+2,$$
and for $g\in G\setminus 2G$, we have 
$$\dim (1-Z_G)\cA_{_g\rho,\rho}=\frac{|G|^2}{2}.$$
This means that if we set 
$$\mu=\bigoplus_{g\in G}\alpha_g\rho,$$
$$\nu_{+}=\bigoplus_{g\in 2G}\alpha_g\rho,$$
$$\nu_{-}=\bigoplus_{g\in G\setminus 2G}\alpha_g\rho,$$
The object $\mu$ has exactly $|G|^2/2$ half-braidings $\{\cE_\mu^i\}_{i\in I}$, and each of $\nu_+$ and $\nu_{-}$ has exactly 
two half-braidings $\{\cE_{\nu_\pm}^j\}_{j=0,1}$.

If $|G_2|=4$, we have 
$$\dim (1-Z_G)\cA_\rho=\dim (1-Z_G)\cA_{_g\rho,\rho}=\frac{|G|^2}{2},$$
for any $g\in G$. 
Thus $\mu$ has exactly $|G|^2/2$ half-braiding $\{\cE_\mu^i\}_{i\in I}$.

So far we have decided the algebra structure of $\Tube \cC$.  
Now we compute $S$ and $T$. 

\begin{lemma}
$$S_{0,0}=\frac{1}{\Lambda}=\frac{a-b}{2},\quad S_{0,\widetilde{\pi}}=\frac{1+|G|d}{\Lambda}=\frac{a+b}{2},$$
$$S_{0,\widetilde{\varphi_{\chi,\pm}}}=\frac{1+\frac{|G|}{2}d}{\Lambda}=\frac{a}{2},\quad S_{0,\widetilde{\sigma_g}^\tau}=\frac{2+|G|d}{\Lambda}=a,$$
$$S_{\widetilde{\pi},\widetilde{\pi}}=\frac{1}{\Lambda}=\frac{a-b}{2},\quad S_{\widetilde{\pi},\widetilde{\varphi_{\chi,\pm}}}=\frac{a}{2},\quad 
S_{\widetilde{\pi},\widetilde{\sigma_g}^\tau}=a,$$
$$S_{\widetilde{\varphi_{\chi_1,\varepsilon_1}},\widetilde{\varphi_{\chi_2,\varepsilon_2}}}
=\frac{a}{2}+\frac{\varepsilon_1\varepsilon_2\delta_{\chi_1,\chi_2}}{4},\quad 
S_{\widetilde{\varphi_{\chi,\pm}},\widetilde{\sigma_g}^\tau}
=a\inpr{g}{\chi},$$
$$S_{\widetilde{\sigma_g}^\tau,\widetilde{\sigma_h}^\theta}=a(\inpr{h}{\tau}\inpr{g}{\theta}+
\overline{\inpr{h}{\tau}\inpr{g}{\theta}}).$$
$$T_{0,0}=T_{\widetilde{\pi},\widetilde{\pi}}=T_{\widetilde{\varphi_{\chi,\pm}},\widetilde{\varphi_{\chi,\pm}}}=1,\quad 
T_{\widetilde{\sigma_g}^\tau,\widetilde{\sigma_g}^\tau}=\inpr{g}{\tau}.$$
\end{lemma}

\begin{proof} The computation not involving $\varphi_{\chi,\pm}$ is the same as in \cite{MR1832764}, and we compute only 
those involving $\varphi_{\chi,\pm}$. 
\begin{align*}
\lefteqn{S_{\widetilde{\varphi_{\chi,\pm}},\widetilde{\pi},}=\frac{d(\varphi_{\chi,\pm})}{\Lambda}\sum_{\xi}d(\xi)
\phi_\xi(\cE_{\pi}(0)_{\xi,\xi}^*\cE_{\varphi_{\chi,\pm}}(\xi)_{0,0}^*)} \\
 &=\frac{a}{2}(1+d\sum_{g\in G} \phi_{_g\rho}(\cE_{\pi}(0)_{_g\rho,_g\rho}^*
 \cE_{\varphi_{\chi,\pm}}(_g\rho)_{0,0}^*)\\
 &=\frac{a}{2}(1+d\sum_{g\in G} \frac{\pm \inpr{g}{\chi}}{d})=\frac{a}{2}.
\end{align*}

\begin{align*}
\lefteqn{S_{\widetilde{\varphi_{\chi_1,\varepsilon_1}},\widetilde{\varphi_{\chi_2,\varepsilon_2}},}
=\frac{d(\varphi_{\chi_1,\varepsilon_1})}{\Lambda}\sum_{\xi}d(\xi)
\phi_\xi(\cE_{\varphi_{\chi_2,\varepsilon_2}}(0)_{\xi,\xi}^*\cE_{\varphi_{\chi_1,\varepsilon_1}}(\xi)_{0,0}^*)} \\
 &=\frac{a}{2}(1+d\sum_{\inpr{g}{\chi_2}=\varepsilon_21} \phi_{_g\rho}
 (\cE_{\varphi_{\chi_2,\varepsilon_2}}(0)_{_g\rho,_g\rho}^*
 \cE_{\varphi_{\chi_1,\varepsilon_1}}(_g\rho)_{0,0}^*)\\
 &=\frac{a}{2}(1+\varepsilon_1\sum_{\inpr{g}{\chi_2}=\varepsilon_21} \inpr{g}{\chi_1}).
\end{align*}
Note that we have 
$$\sum_{\inpr{g}{\chi_2}=1}\inpr{g}{\chi_1}+\sum_{\inpr{g}{\chi_2}=-1}\inpr{g}{\chi_1}=0,$$
and 
$$\varepsilon_1\sum_{\inpr{g}{\chi_2}=\varepsilon_21} \inpr{g}{\chi_1}
=\varepsilon_1\varepsilon_2\sum_{\inpr{g}{\chi_2}=1} \inpr{g}{\chi_1}=\varepsilon_1\varepsilon_2\delta_{\chi_1,\chi_2}\frac{|G|}{2}.$$

\begin{align*}
\lefteqn{S_{\widetilde{\varphi_{\chi,\pm}},\widetilde{\sigma_0}^\tau,}
=\frac{d(\varphi_{\chi,\pm})}{\Lambda}\sum_{\xi,i}d(\xi)
\phi_\xi(\cE_{\sigma_0}^\tau(0)_{(\xi,i),(\xi,i)}^*\cE_{\varphi_{\chi,\pm}}(\xi)_{0,0}^*)} \\
 &=\frac{a}{2}(2+d\sum_{g\in G} \phi_{_g\rho}(\cE_{\sigma_0}^\tau(0)_{_g\rho,_g\rho}^*
 \cE_{\varphi_{\chi,\pm}}(_g\rho)_{0,0}^*)\\
 &=\frac{a}{2}(2+d\sum_{g\in G} \frac{\pm \inpr{g}{\chi}}{d})=a.
\end{align*}

\begin{align*}
\lefteqn{S_{\widetilde{\varphi_{\chi,\pm}},\widetilde{\sigma_g}^\tau,}
=\frac{d(\varphi_{\chi,\pm})}{\Lambda}\sum_{\xi}d(\xi)
\phi_\xi(\cE_{\sigma_g}^\tau(0)_{\xi,\xi}^*\cE_{\varphi_{\chi,\pm}}(\xi)_{0,0}^*)} \\
 &=\frac{a}{2}(\cE_{\sigma_g}^\tau(0)_{g,g}^*\cE_{\varphi_{\chi,\pm}}(g)_{0,0}^*
 +\cE_{\sigma_g}^\tau(0)_{-g,-g}^*\cE_{\varphi_{\chi,\pm}}(-g)_{0,0}^*
 +d\sum_{h\in G} \phi_{_h\rho}(\cE_{\sigma_0}^\tau(0)_{_h\rho,_h\rho}^*
 \cE_{\varphi_{\chi,\pm}}(_h\rho)_{0,0}^*)\\
 &=\frac{a}{2}(\cE_{\varphi_{\chi,\pm}}(g)_{0,0}^*+\cE_{\varphi_{\chi,\pm}}(-g)_{0,0}^*+d\sum_{h\in G} \frac{\pm \inpr{h}{\chi}}{d})
 =a\inpr{g}{\chi}.
\end{align*}
\end{proof}

For $k,l\in G_2$, we have 
$$S_{k,l}=\frac{\epsilon_k(l)\epsilon_l(k)}{\Lambda}=\frac{a-b}{2}\epsilon_k(l)\epsilon_l(k),$$
$$T_{k,k}=\epsilon_k(k).$$

For $k\in G_2$ and an irreducible object $X$ in $\cZ(\cC)$, 
we set 
$$s(k,X)=\frac{S_{k,X}}{|S_{k,X}|}.$$
Since $k$ is an invertible object in $\cZ(\cC)$, 
we get 
$$S_{kX,Y}=s(k,Y)S_{X,Y},$$
$$T_{kX,kX}=\overline{s(k,X)}T_{k,k}T_{X,X},$$
and so 
$$S_{kX,lY}=s(k,l)s(k,Y)s(l,X)S_{X,Y}=\epsilon_k(l)\epsilon_l(k)s(k,Y)s(l,X)S_{X,Y},$$
for $l\in G_2$ and $Y\in \Irr(\cZ(\cC))$. 

Note that we have 
$$s(k,\widetilde{\gamma}^i)=\cE_k(\eta)^*\cE_\gamma^i(k)_{(\eta,j),(\eta,j)}^*.$$
In particular, when $\gamma$ contains $g\in G$, we get 
$$s(k,\widetilde{\gamma}^i)=\epsilon_k(g)\cE_\gamma^i(k)_{g,g}^*,$$
and when $\gamma$ contains $\alpha_g\circ\rho$, 
$$s(k,\widetilde{\gamma}^i)=\epsilon_k(g)\cE_\gamma^i(k)_{_g\rho,_g\rho}^*.$$
These facts imply

\begin{lemma}
For $k\in G_2$,  
$$s(k,\widetilde{\pi})=1,$$
$$s(k,\widetilde{\varphi_{\chi,\pm}})=\inpr{k}{\chi},$$
$$s(k,\widetilde{\sigma_g}^\tau)=\epsilon_k(g)\inpr{k}{\tau}.$$
\end{lemma}

Let 
$$(G\times \hG)_*=(G_2\times (\hG)_*)\sqcup (G_*\times \hG).$$
Then we have 
$$G\times \hG=(G\times \hG)_2\sqcup (G\times \hG)_*\sqcup -(G\times \hG)_*.$$

\begin{theorem}\label{2}  Let $\cC\subset \End(M)$ be a generalized Haagerup category with a finite group $G$ satisfying Eq.(\ref{Q2}). 
Assume $|G_2|=2$. 
The following set exhausts the equivalence classes of the simple objects in $\cZ(\cC)$:
\begin{align*}
\lefteqn{\{k\}_{k\in G_2}\cup\{k\widetilde{\pi}\}_{k\in G_2}\cup \{k\widetilde{\varphi_{\chi_0,\varepsilon}}\}_{k\in G_2,\;\varepsilon\in \{1,-1\}}}\\
 &\cup \{\widetilde{\sigma_g}^\tau\}_{(g,\tau)\in (G\times \hG)_*}
\cup \{\widetilde{\nu_\varepsilon}^j\}_{\varepsilon\in \{1,-1\},\; j\in \{0,1\}}
\cup \{\widetilde{\mu}^i\}_{i\in I},
\end{align*}
where we use the notation $(\hG)_2=\{0,\chi_0\}$ and $I$ is an index set with $|I|=|G|^2/2$. 
Every object in $\cZ(\cC)$ is self-dual.

Let $k,l\in G_2$, and let $s(k,l)=\epsilon_k(l)\epsilon_l(k)$. 
There exist characters $G_2\ni k\mapsto s(k,\widetilde{\nu_{\pm}}^j)\in \{1,-1\}$ and 
$G_2\ni k\mapsto s(k,\widetilde{\mu}^i)\in \{1,-1\}$ satisfying 
$$S_{k,l}=s(k,l)\frac{a-b}{2},\quad S_{k.l\widetilde{\pi}}=s(k,l)\frac{a+b}{2},
\quad S_{k,l\widetilde{\varphi_{\chi_0,\pm}}}=s(k,l)\frac{a}{2}\inpr{k}{\chi_0},$$
$$ S_{k,\widetilde{\sigma_g}^\tau}=a\inpr{k}{\tau}\epsilon_k(g),
\quad S_{k,\widetilde{\nu_\pm}^j}=s(k,\widetilde{\nu_\pm}^j)\frac{b}{2},\quad S_{k,\widetilde{\mu}^i}=s(k,\widetilde{\mu}^i)b,$$
$$S_{k\widetilde{\pi},l\widetilde{\pi}}=s(k,l)\frac{a-b}{2},
\quad S_{k\widetilde{\pi},l\widetilde{\varphi_{\chi_0,\pm}}}=s(k,l)\frac{a}{2}\inpr{k}{\chi_0}, 
\quad S_{k\widetilde{\pi},\widetilde{\sigma_g}^\tau}=a\inpr{k}{\tau}\epsilon_k(g),$$
$$S_{k\widetilde{\pi},\widetilde{\nu_\pm}^j}=-s(k,\widetilde{\nu_\pm}^j)\frac{b}{2},\quad S_{k\widetilde{\pi},\widetilde{\mu}^i}=-s(k,\widetilde{\mu}^i)b,$$
$$S_{k\widetilde{\varphi_{\chi_0,\varepsilon_1}},l\widetilde{\varphi_{\chi_0,\varepsilon_2}}}
=s(k,l)\inpr{k}{\chi_0}\inpr{l}{\chi_0}(\frac{a}{2}+\frac{\varepsilon_1\varepsilon_2}{4}),$$
$$S_{k\widetilde{\varphi_{\varepsilon}},\widetilde{\sigma_g}^\tau}=a\inpr{k}{\tau}\epsilon_k(g)\inpr{g}{\chi_0},
\quad S_{k\widetilde{\varphi_{\chi_0,\varepsilon_1}},\widetilde{\nu_{\varepsilon_2}}^{j}}
=s(k,\widetilde{\nu_{\varepsilon_2}}^{j})\frac{\varepsilon_1\varepsilon_2}{4},
\quad S_{k\widetilde{\varphi_{\chi_0,\pm}},\widetilde{\mu}^{j}}=0,$$
$$S_{\widetilde{\sigma_g}^\tau,\widetilde{\sigma_h}^\theta}=a(\inpr{h}{\tau}\inpr{g}{\theta}+
\overline{\inpr{h}{\tau}\inpr{g}{\theta}}),
\quad S_{\widetilde{\sigma_g}^\tau,\widetilde{\nu_{\pm}}^k}=0,
\quad S_{\widetilde{\sigma_g}^\tau,\widetilde{\mu}^{j}}=0,$$
$$T_{k,k}=T_{k\widetilde{\pi},k\widetilde{\pi}}=\epsilon_k(k),\quad T_{k\widetilde{\varphi_{\chi_0,\pm}},k\widetilde{\varphi_{\pm}}}=\epsilon_k(k)\inpr{k}{\chi},
\quad T_{\widetilde{\sigma_g}^\tau,\widetilde{\sigma_g}^\tau}=\inpr{g}{\tau}.$$
\end{theorem}

\begin{proof} Most of the computation is similar to that in \cite{MR1832764} except for  
$$S_{\widetilde{\varphi_{\chi_0,\varepsilon}},\widetilde{\nu_+}^i}=\frac{d(\varphi_{\chi_0,\varepsilon})}{\Lambda}
\sum_{g\in 2G}d\phi_{_g\rho}(\cE_{\nu_+}^i(0)_{_g\rho,_g\rho}^*\cE_{\varphi_{\chi_0,\varepsilon}}(_g\rho)_{0,0}^*)
=\frac{\varepsilon}{2|G|}\sum_{g\in 2G}\inpr{g}{\chi_0}=\frac{\varepsilon}{4},$$
$$S_{\widetilde{\varphi_{\chi_0,\varepsilon}},\widetilde{\nu_-}^i}=\frac{d(\varphi_{\chi_0,\varepsilon})}{\Lambda}
\sum_{g\in G\setminus 2G}d\phi_{_g\rho}(\cE_{\nu_-}^i(0)_{_g\rho,_g\rho}^*\cE_{\varphi_{\chi_0,\varepsilon}}(_g\rho)_{0,0}^*)
=\frac{\varepsilon}{2|G|}\sum_{g\in G\setminus 2G}\inpr{g}{\chi_0}=-\frac{\varepsilon}{4},$$
$$S_{\widetilde{\varphi_{\chi_0,\varepsilon}},\widetilde{\mu}^j}=\frac{d(\varphi_{\chi_0,\varepsilon})}{\Lambda}
\sum_{g\in G}d\phi_{_g\rho}(\cE_\mu^j(0)_{_g\rho,_g\rho}^*\cE_{\varphi_{\chi_0,\varepsilon}}(_g\rho)_{0,0}^*)
=\frac{1}{2|G|}\sum_{g\in G}\varepsilon\inpr{g}{\chi_0}=0.$$

We have already seen in the proof of Lemma \ref{mulfree1} that $S_0^2$ restricted to $\cA_{_g\rho}$ is the identity, 
which implies that every simple object $X\in \cZ(\cC)$ with $\cF(X)$ containing $\alpha_g\circ \rho$ for some $g\in G$ 
is self-dual. 
The only simple objects not satisfying this condition are those in $G_2$, and they are again self-dual. 
\end{proof}

%

In a similar way, we can show the following.
\begin{theorem}\label{4} Let $\cC\subset \End(M)$ be a generalized Haagerup category with a finite group $G$ satisfying Eq.(\ref{Q2}). 
Assume $|G_2|=4$. 
The following set exhausts the equivalence classes of the simple objects in $\cZ(\cC)$:
$$\{k\}_{k\in G_2}\cup\{k\widetilde{\pi}\}_{k\in G_2}\cup\{k\widetilde{\varphi_{\chi,\varepsilon}}\}_{k\in G_2,\;\chi\in (\hG)_2\setminus \{0\},\;\varepsilon\in \{1,-1\}}
 \cup \{\widetilde{\sigma_g}^\tau\}_{(g,\tau)\in (G\times \hG)_*}\cup \{\widetilde{\mu}^i\}_{i\in I}$$
where $I$ is an index set with $|I|=|G|^2/2$. 
Every object in $\cZ(\cC)$ is self-dual. 

Let $k,l\in G_2$, and let $s(k,l)=\epsilon_k(l)\epsilon_l(k)$. 
There exist characters $G_2\ni k\mapsto s(k,\widetilde{\mu}^i)\in \{1,-1\}$ satisfying 
$$S_{k,l}=s(k,l)\frac{a-b}{2},\quad S_{k.l\widetilde{\pi}}=s(k,l)\frac{a+b}{2},
\quad S_{k,l\widetilde{\varphi_{\chi,\pm}}}=s(k,l)\frac{a}{2}\inpr{k}{\chi},$$
$$ S_{k,\widetilde{\sigma_g}^\tau}=a\inpr{k}{\tau}\epsilon_k(g),\quad S_{k,\widetilde{\mu}^i}=s(k,\widetilde{\mu}^i)b,$$
$$S_{k\widetilde{\pi},l\widetilde{\pi}}=s(k,l)\frac{a-b}{2},
\quad S_{k\widetilde{\pi},l\widetilde{\varphi_{\chi,\pm}}}=s(k,l)\frac{a}{2}\inpr{k}{\chi}, 
\quad S_{k\widetilde{\pi},\widetilde{\sigma_g}^\tau}=a\inpr{k}{\tau}\epsilon_k(g),$$
$$S_{k\widetilde{\pi},\widetilde{\mu}^i}=-s(k,\widetilde{\mu}^i)b,
\quad S_{k\widetilde{\varphi_{\chi_1,\varepsilon_1}},l\widetilde{\varphi_{\chi_2,\varepsilon_2}}}
=s(k,l)\inpr{k}{\chi_2}\inpr{l}{\chi_1}(\frac{a}{2}+\frac{\varepsilon_1\varepsilon_2\delta_{\chi_1,\chi_2}}{4}),$$
$$S_{k\widetilde{\varphi_{\chi,\varepsilon}},\widetilde{\sigma_g}^\tau}=a\inpr{k}{\tau}\epsilon_k(g)\inpr{g}{\chi},
\quad S_{k\widetilde{\varphi_{\chi,\pm}},\widetilde{\mu}^{j}}=0,$$
$$S_{\widetilde{\sigma_g}^\tau,\widetilde{\sigma_h}^\theta}=a(\inpr{h}{\tau}\inpr{g}{\theta}+
\overline{\inpr{h}{\tau}\inpr{g}{\theta}}),
\quad S_{\widetilde{\sigma_g}^\tau,\widetilde{\mu}^{j}}=0,$$
$$T_{k,k}=T_{k\widetilde{\pi},k\widetilde{\pi}}=\epsilon_k(k),\quad T_{k\widetilde{\varphi_{\chi,\pm}},k\widetilde{\varphi_{\chi,\pm}}}=\epsilon_k(k)\inpr{k}{\chi},
\quad T_{\widetilde{\sigma_g}^\tau,\widetilde{\sigma_g}^\tau}=\inpr{g}{\tau}.$$
\end{theorem}
%
%
The results of Theorems \ref{2} and \ref{4} are summarized in Figure \ref{stpart}.

\begin{figure}
\resizebox{\textwidth}{!}{%
$
\begin{array}{c !{\vline width 1pt} cccccc !{\vline width 1pt}}
\textbf{S} & & & & & & \\[3pt]
\Xhline{1pt}
l& s(k,l)\frac{a-b}{2} & & & & & \Tstrut \\[3pt]
l\widetilde{\pi}& s(k,l)\frac{a+b}{2} & s(k,l)\frac{a-b}{2}& & & & \\[3pt]
l\widetilde{\varphi_{\chi_2,\varepsilon_2}}  &s(k,1)\frac{a}{2}\inpr{k}{\chi_2} &  s(k,l)\frac{a}{2}\inpr{k}{\chi_2} &  s(k,l)\inpr{k}{\chi_2}\inpr{l}{\chi_1}(\frac{a}{2}+\frac{\varepsilon_1\varepsilon_2\delta_{\chi_1,\chi_2}}{4}) & & & \\[3pt]
 \widetilde{\sigma_h}^\theta & a\inpr{k}{\theta}\epsilon_k(h) & a\inpr{k}{\theta}\epsilon_k(h)& a\inpr{k}{\theta}\epsilon_k(h)\inpr{h}{\chi_1}&a(\inpr{h}{\tau}\inpr{g}{\theta}+
\overline{\inpr{h}{\tau}\inpr{g}{\theta}} & & \\[3pt]
\widetilde{\nu_{\varepsilon_2}}^j& s(k,\widetilde{\nu_{\varepsilon_2}}^j)\frac{b}{2} &
-s(k,\widetilde{\nu_{\varepsilon_2}}^j)\frac{b}{2}  & s(k,\widetilde{\nu_{\varepsilon_2}}^{j})\frac{\varepsilon_1\varepsilon_2}{4}& 0& ? & \\[3pt]
 \widetilde{\mu}^i & s(k,\widetilde{\mu}^i)b&-s(k,\widetilde{\mu}^i)b & 0 & 0&? & ? \\[3pt]
 \Xhline{1pt}
  & k & k\widetilde{\pi} & k\widetilde{\varphi_{\chi_1,\varepsilon_1}} & \widetilde{\sigma_g}^\tau& \widetilde{\nu_{\varepsilon_1}}^{j'} & \widetilde{\mu}^{i'} \Tstrut \\[3pt]
 \Xhline{1pt}
 \textbf{T} &\epsilon_k(k) & \epsilon_k(k) &  \epsilon_k(k)\inpr{k}{\chi_1} &\inpr{g}{\tau} &? & ? \Tstrut \\[3pt]
\end{array}
$
}
\caption{Partial modular data for generalized Haagerup categories for even groups satisfying Eq. (\ref{Q2}). Undetermined entries are labeled by ``?''. (The index set for $\widetilde{\nu_{\varepsilon} }$ is empty when $|G_2|=4$.)}
\label{stpart}
\end{figure}

Now we show how to decide $s(k,\widetilde{\mu}^i)$ and $s(k,\widetilde{\nu_{\pm}}^i)$ for $k\in G_2$. 
For $g\in G$ and $k\in G_2$, we set $$U_g(k)=\epsilon_k(g)(_g\rho\;k|1|k\;_g\rho).$$
Then $\{U_g(k)\}_{k\in G_2}$ forms a representation of $G_2$ in $\cA_{_g\rho}$. 
Let $X=\tilde{\gamma}^i$ be a simple object in $\cZ(\cC)$ such that $\gamma$ contains $_g\rho$ and $_h\rho$.  
Since $\cA_{_g\rho}$ is abelian, $U_g(k)e(X)_{_g\rho,_h\rho}$ is a scalar multiple of $e(X)_{_g\rho,_h\rho}$. 

\begin{lemma} Let the notation be as above. 
$$U_g(k)e(X)_{_g\rho,_h\rho}=s(k,X)e(X)_{_g\rho,_h\rho}.$$
\end{lemma}

\begin{proof} It suffices to show the statement in the case of $h=g$. 
We have already seen $s(k,X)=\epsilon_k(g)\cE_\gamma^i(k)_{_g\rho,_g\rho}^*$, and so 
$$\cE_\gamma^i(k)_{_g\rho,_g\rho}=s(k,X)\epsilon_k(g)\in (\alpha_g\rho\alpha_k,\alpha_k\alpha_g\rho)=\C.$$
Thus 
\begin{align*}
\lefteqn{e(X)_{_g\rho,_g\rho}=\frac{d(X)}{\Lambda d}\sum_{\xi}d(\xi)(_g\rho\;\xi|\cE_\gamma^i(\xi)_{_g\rho,_g\rho}|\xi\;_g\rho)} \\
 &=\frac{d(X)}{\Lambda d}(
 \sum_{k\in G_2}(_g\rho\;k|\cE_\gamma^i(k)_{_g\rho,_g\rho}|k\;_g\rho))+
 \sum_{h\in G}d(_g\rho\;_h\rho|\cE_\gamma^i(_h\rho)_{_g\rho,_g\rho}|_h\rho\;_g\rho))\\ 
 &=\frac{d(X)}{\Lambda d}(
 \sum_{k\in G_2}s(k,X)U_g(k)+
 \sum_{h\in G}d(_g\rho\;_h\rho|\cE_\gamma^i(_h\rho)_{_g\rho,_g\rho}|_h\rho\;_g\rho)),
\end{align*}
which shows the statement. 
\end{proof}

From the definition of $U_g(k)$, we can see that the representation $\{U_g(k)\}_{k\in G_2}$ of $G_2$ in 
$\cA_{_g\rho,_g\rho}$ is a multiple of the regular representation, that is, each character of $G_2$ occurs with 
the same multiplicity. 
For the same reason, we see that the left multiplication of $U_g(k)$ on $\cA_{_g\rho,_h\rho}$ gives rise to a representation 
of $G_2$, which is a multiple of the regular representation. 
On the other hand, the above formula tells the multiplicity of each character in $Z_G\cA_{_g\rho}$ and 
$Z_G\cA_{_g\rho,_h\rho}$. 
Thus we know the multiplicity of each character in $(1-Z_G)\cA_{_g\rho}$ and $(1-Z_G)\cA_{_g\rho,_h\rho}$. 

Recall that we can canonically decompose $G$ as $G_e\times G_o$ where $G_e$ is a 2-group and $G_o$ is an odd group. 

\begin{cor} Assume $|G_2|=2$, and let $a_0$ be the unique non-trivial element of $G_2$. 
Let $I$ be the index set of the half-braidings of $\mu$, and let 
$$I_\pm=\{i\in I;\; s(a_0,\widetilde{\mu}^i)=\pm1\}.$$
\begin{itemize}
\item [(1)] Assume that $G_e$ is $\Z_2$. 
Then $|I_+|=|G|^2/4-1$, $|I_-|=|G|^2/4+1$, and $s(a_0,\widetilde{\nu_{\pm}}^j)=1$. 
\item [(2)] Assume that $G_e$ is not $\Z_2$. 
Then $|I_\pm|=|G|^2/4$ and $s(a_0,\widetilde{\nu_\pm}^j)=-1$. 
\end{itemize}
\end{cor}

\section{Computing the $(\nu/\mu) \times (\nu/\mu) $ corner of the modular data} \label{corner}
\subsection{Outline of the method} \label{corner1}
The table in Figure \ref{stpart} gives formulas for the modular data of the center of a generalized Haagerup category for an even group satisfying Eq.(\ref{Q2}), except for the corners of $S$ and $T$ indexed by $\widetilde{\nu_{\varepsilon}}^j$ (for $|G_2|
=2$) and $\widetilde{\mu}^i $. Since $\nu $ and $\mu $ do not have any invertible simple summands, we are unable to prove a general formula for the modular data in terms of the characters $ \epsilon$.  However, we can compute the missing corner directly from the tube algebra using the full data of the category $(\epsilon, A)$ for some small examples (and we will also formulate conjectures for the general case).  

To perform the computation, we need to find the minimal projections $e(\widetilde{\nu_{\varepsilon}}^j)_{{}_g\rho,{}_g \rho} $  and  $e(\widetilde{\mu}^i )_{{}_g\rho,{}_g \rho} $ in $\mathcal{A}_{{}_g \rho} $ for each $g$. 
The projections $e(\cdot)_{{}_g\rho,{}_g \rho }$ for different $g$ can then be added together to find the minimal central projections  $z(\widetilde{\nu_{\varepsilon}}^j)$  
and  $z(\widetilde{\mu}^i )$. Then the modular data can be computed using Eq.(\ref{Sformula0}) or Eq.(\ref{Sformula}).

First, we find formulas for the minimal central projections $z(\cdot) $ corresponding to each of $k\widetilde{\pi} $, $k\widetilde{\varphi_{\chi,\pm}} $, and $\widetilde{\sigma_{g, \tau}} $ as follows. Let $$\mathcal{A}_G=\bigoplus_{g,h \in G}\limits\mathcal{A}_{g,h},$$ and similarly for $\mathcal{A}_{{}_G\rho} $, etc.  For each minimal central projection $p$ in $\mathcal{A}_G $ and each $h \in G$ such that $p\mathcal{A}_{G,{}_h \rho} $ is nontrivial, we choose a minimal subprojection $p'$ of $p$ and a partial isometry $j(p',h) $ in $\mathcal{A}_{G,{}_h \rho} $ such that $j(p',h)j^*(p',h)=p' $; then $j(p',h)j^*(p',h) $ is the corresponding minimal projection in $\mathcal{A}_{{}_h \rho} $.
In this way we can find all of the minimal central projections $z(\cdot)$ in the tube algebra such that $1_Gz(\cdot)\neq 0 $. Summing the $z(\cdot) $ for all of the $k$, $k\widetilde{\pi} $, $k\widetilde{\varphi_{\chi,\pm}} $, and $\widetilde{\sigma_{g, \tau}} $, we obtain an expression for $Z_G$.

Next, we diagonalize the action of $\textbf{t}_{{}_g\rho}$ on $\mathcal{A}_{{}_g \rho} $ for a representative $g$ in each coset of $2G$ (since $1_{{}_g \rho} $ is equivalent to $1_{{}_{g+2h} \rho} $). We use Mathematica to perform the calculations. It turns out that even for relatively small examples, it is difficult to calculate the eigenvalues directly, since the arithmetic takes place in a complicated number field. So instead we first find the eigenvalues numerically and guess their exact values. Then we use this guess to construct the minimal polynomial $q$ of $\textbf{t}_{{}_g\rho}$. We then verify that the guess is correct by showing through exact calculation that $ q(\textbf{t}_{{}_g\rho})$ vanishes, and that no proper factor of $q $ vanishes on $\textbf{t}_{{}_g\rho}$. 

Once we have the eigenvalues of $\textbf{t}_{{}_g\rho}$, we can compute the projections onto the eigenspaces. For each eigenvalue $\lambda $, set
$$p_g^\lambda=\displaystyle \frac {q_\lambda(\textbf{t}_{{}_g\rho}) }{q_\lambda(\lambda)}, \quad \text{ where } q_{\lambda} \text{ is the polynomial } q_{\lambda}(z)=\displaystyle \frac{q (z)}{z-\lambda}.$$
For each eigenvalue $\lambda $, let $(p_g^\lambda)'=(1-Z_G)p_g^\lambda $.

If the number of nonzero $(p_g^\lambda)' $ in $\mathcal{A}_{{}_g \rho} $ is equal to the dimension of $(1-Z_G)\mathcal{A}_{g \rho}$, then the $(p_g^\lambda)'$ must be exactly the $e(\widetilde{\nu_{\varepsilon}}^j)_{{}_g\rho,{}_g \rho} $ and  $e(\widetilde{\mu}^i)_{{}_g\rho,{}_g \rho} $ that we are looking for.

If the number of nonzero $(p_g^\lambda)' $ in $\mathcal{A}_{{}_g \rho} $ is less than the dimension of $(1-Z_G)\mathcal{A}_{g \rho}$, then  some of the  $(p_g^\lambda )'$ have rank greater than $1$ and we need to split them up. In the small examples that we consider, the highest rank that comes up for $(p_g^\lambda)'$ is $2$. We can then split up a given projection $(p_g^\lambda)' $ by finding an element $x$ in $\mathcal{A}_{{}_g \rho} $ such that $x (p_g^\lambda)' $ is not a scalar multiple of  $(p_g^\lambda)' $. Then we can write down a linear relation among $(p_g^\lambda)' $, $x (p_g^\lambda)' $,  and $x^2 (p_g^\lambda)' $ and find the minimal subprojections of $(p_g^\lambda)' $.

%
%
%
%
%
%
%

Once we have found all of the $e(\cdot)_{{}_g \rho,{}_g \rho} $, we need to match them up and add them together and find the $z(\cdot) $. Matching minimal projections of $\mathcal{A}_{{}_g\rho } $ and $\mathcal{A}_{{}_h\rho } $ which share a common $T$-eigenvalue can be done by looking at the action of  $\mathcal{A}_{{}_g\rho,  {}_h\rho}$. However, in the small examples we consider, there is only one matching that gives consistent modular data.

%
%
%

\begin{remark}
\begin{enumerate}
\item The first step, finding an expression for $Z_G $, is not strictly necessary if one is only interested in finding the missing corner of modular data. If we skip this step, there may be a bit more work when we diagonalize the action of $\textbf{t}_{{}_g\rho}$ in figuring out which eigenvalues correspond to  the $\widetilde{\nu_{\varepsilon}}^j$ and the $\widetilde{\mu}^i $, and in splitting up the corresponding eigenprojections if some of the $\widetilde{\nu_{\varepsilon}}^j$ or  $\widetilde{\mu}^i $ share eigenvalues with subprojections of $Z_G $.
\item To find the partial isometries $j(p',h) $ in $\mathcal{A}_{g,{}_h \rho} $, we can simply multiply $p'$ by elements of a basis for $\mathcal{A}_{g,{}_h \rho} $ until we find something nonzero. Since $p'$ is minimal and $\mathcal{A}_{{}_h \rho} $ is Abelian, any nonzero element in $p'\mathcal{A}_{g,{}_h \rho} $ can be rescaled to a partial isometry. 
\item If we use Eq. (\ref{Sformula0}) to compute the $S$-matrix, we can take advantage of the fact that the expression
$$ \psi(S_0 (\cdot) , \cdot) $$ 
is bilinear to simplify the calculation in the case that the  $\textbf{t}$-eigenvalues for the $\widetilde{\mu}^i $ are multiplicity free. In this case, we can write each $\widetilde{\mu}^i $ as a linear combination of powers of $\textbf{t}$. Then we can find the corresponding values of the $S$-matrix by first 
computing $$ \psi(S_0( \textbf{t}^m) , \textbf{t}^n) $$   
for various $m$ and $n$,
and then taking an appropriate linear combination of those values. The advantage is that the tube algebra calculations now take place with simpler numbers, and the more complicated coefficients are only introduced at the last step.
\end{enumerate}
\end{remark}

\subsection{Examples with $|G_2|$=2}

\begin{example}

For $G=\Z_2$ there is a unique generalized Haagerup category $\mathcal{C} $, which is the even part of the $A_7$ subfactor (or the quantum group category $PSU(2) $ at level $6$). The structure constants $(\epsilon,A)$ for this category are given in \cite[Section 9.1]{MR3827808}.

Then $G=G_2$, $(G \times \hat{G})_* $ is empty, $|I|=|I_+|=2 $, and $\mathcal{Z}(\mathcal{C}) $ has rank $14$.
The $T$-eigenvalues for the six objects $(\widetilde{\nu_{+}} ^0, \widetilde{\nu_{+}}^1,\widetilde{\nu_{-}}^0,\widetilde{\nu_-}^{1} ,\widetilde{ \mu}^0,\widetilde{ \mu}^1 ) $ are given by $$(i,-i,i,-i,\dfrac{1+i}{\sqrt{2}},\dfrac{1-i}{\sqrt{2}} ) ,$$ and the corresponding block of the $S$-matrix is
$$\dfrac{1}{8}\left(
\begin{array}{cccccc}
 -2+\sqrt{2} & 2+\sqrt{2} & 2+\sqrt{2} & -2+\sqrt{2} & -2 \sqrt{2} & 2 \sqrt{2} \\
 2+\sqrt{2} & -2+\sqrt{2} & -2+\sqrt{2} & 2+\sqrt{2} & 2 \sqrt{2} & -2 \sqrt{2} \\
 2+\sqrt{2} & -2+\sqrt{2} & -2+\sqrt{2} & 2+\sqrt{2} & -2 \sqrt{2} & 2 \sqrt{2} \\
 -2+\sqrt{2} & 2+\sqrt{2} & 2+\sqrt{2} & -2+\sqrt{2} & 2 \sqrt{2} & -2 \sqrt{2} \\
 -2 \sqrt{2} & 2 \sqrt{2} & -2 \sqrt{2} & 2 \sqrt{2} & 0 & 0 \\
 2 \sqrt{2} & -2 \sqrt{2} & 2 \sqrt{2} & -2 \sqrt{2} & 0 & 0 \\
\end{array}
\right) .$$

Let $a_0$ be the non-trivial element of $G$, which we also use to label the corresponding invertible object in  $\mathcal{Z}(\mathcal{C}) $. Since $\epsilon_{a_0}(a_0)=-1 $, $a_0 $ is a fermion (which means its twist is $ -1$), and $(\mathcal{Z}(\mathcal{C}) ,k_0)$ is a spin-modular category in the sense of \cite{MR2163566}. 

The trivial component $\mathcal{Z}(\mathcal{C})_0 $ with respect to the $ \Z_2$-grading associated to $a_0$ is the supermodular tensor subcategory generated by the objects $k$, $k\widetilde{\pi}$, and and $\widetilde{\nu_\epsilon}^j $. Then $\mathcal{Z}(\mathcal{C}) $ is a modular closure of $\mathcal{Z}(\mathcal{C})_0 $, and by \cite[Theorem 5.4]{MR3613518} there are exactly $16$ different modular closures of $\mathcal{Z}(\mathcal{C})_0 $. The modular data for $8$ of these can be computed by the zesting formula in \cite[Theorem 3.15]{MR3641612}.
\end{example}

\begin{example}
For $G=\Z_4$, there is a unique generalized Haagerup category satisfying Eq. (\ref{Q2}). The structure constants $(\epsilon,A) $ for this category are given in \cite[Section 9.3]{MR3827808}. In this case $|I_+|=|I_-|=4 $. We index the $\widetilde{\mu}^i $ by $\{+,- \} \times \{ 1,2 \}  \times \{ -1,1\}$ (with the sign corresponding to $I_{\pm} $). 

We can compute the missing corner of the modular data following the outline in Section \ref{corner1}. Some of the $\textbf{t}$-eigenvalues for the $\widetilde{\mu}^i $ have multiplicity, so it is necessary to split the corresponding eigenprojections by brute force as explained above.

 Then the $12 \times 12$ block corresponding to the $$(\widetilde{\nu_+},\widetilde{\nu_-},\widetilde{\mu_+},\widetilde{\mu_-} )$$ is as follows. The eigenvalues of the $T$-matrix are $$(i,-i,i,-i, \zeta_5^2, \zeta_5^2 ,\zeta_5^{-2} ,\zeta_5^{-2} ,i \zeta_5^2, -i\zeta_5^2 ,i\zeta_5^{-2} ,-i\zeta_5^{-2} ) ,$$ where $\zeta_r^n=e^{\frac{2\pi i n }{r}} $, and the corresponding block of the $S$-matrix is

$$\left(
\begin{array}{cccccccccccc}
 c_3 & c_2 & c_1 & c_4 & -1 & 1 & -1 & 1 & 1 & -1 & 1 & -1 \\
 c_2 & c_3 & c_4 & c_1 & -1 & 1 & -1 & 1 & -1 & 1 & -1 & 1 \\
 c_1 & c_4 & c_3 & c_2 & -1 & 1 & -1 & 1 & 1 & -1 & 1 & -1 \\
 c_4 & c_1 & c_2 & c_3 & -1 & 1 & -1 & 1 & -1 & 1 & -1 & 1 \\
 -1 & -1 & -1 & -1 & c_3 & c_3 & c_1 & c_1 & c_3 & c_3 & c_1 & c_1 \\
 1 & 1 & 1 & 1 & c_3 & c_3 & c_1 & c_1 & c_2 & c_2 & c_4 & c_4 \\
 -1 & -1 & -1 & -1 & c_1 & c_1 & c_3 & c_3 & c_1 & c_1 & c_3 & c_3 \\
 1 & 1 & 1 & 1 & c_1 & c_1 & c_3 & c_3 & c_4 & c_4 & c_2 & c_2 \\
 1 & -1 & 1 & -1 & c_3 & c_2 & c_1 & c_4 & c_2 & c_3 & c_4 & c_1 \\
 -1 & 1 & -1 & 1 & c_3 & c_2 & c_1 & c_4 & c_3 & c_2 & c_1 & c_4 \\
 1 & -1 & 1 & -1 & c_1 & c_4 & c_3 & c_2 & c_4 & c_1 & c_2 & c_3 \\
 -1 & 1 & -1 & 1 & c_1 & c_4 & c_3 & c_2 & c_1 & c_4 & c_3 & c_2 \\
\end{array}
\right)$$

where $c_k=2 \cos \frac{k \pi}{5} \in \frac{1}{2}\{ \pm 1 \pm \sqrt{5}  \}$.

This information can be summarized by the table in Figure \ref{stpart2}.

\begin{figure}
$
\begin{array}{c !{\vline width 1pt} ccc!{\vline width 1pt}}
\sqrt{n^2+4} \cdot \textbf{S} & & &  \\[3pt]
\Xhline{1pt}
\widetilde{\nu_{\varepsilon'_1}}^{\varepsilon'_2}& \frac{\varepsilon_2 \varepsilon_2' }{2}-\frac{\varepsilon_1\varepsilon_2\varepsilon_1'\varepsilon_2' }{4}\sqrt{n^2+4}& & \Tstrut \\[3pt]
\widetilde{\mu_{+}}^{l',\varepsilon'}&-\varepsilon'  & -2 \cos \frac{4 \pi a l l'}{r} &  \\[3pt]
\widetilde{\mu_{-}}^{l',\varepsilon'}  & \varepsilon'_2 \varepsilon &  -\varepsilon 2 \cos \frac{4 \pi a l l'}{r}   & \varepsilon \varepsilon' 2 \cos \frac{4 \pi a l l'}{r}    \\[3pt]
 
 \Xhline{1pt}
  & \widetilde{\nu_{\varepsilon_1}}^{\varepsilon_2} & \widetilde{\mu_{+}}^{l,\varepsilon} & \widetilde{\mu_{-}}^{l,\varepsilon} \\[3pt]
 \Xhline{1pt}
 \textbf{T} &\varepsilon_2  i& \zeta_r^{al^2} 
  & \varepsilon i \zeta_r^{al^2}\ \Tstrut \\[3pt]
\end{array}
$
\caption{The missing corner of modular data for the generalized Haagerup category for $\Z_4 $. Here $n=|G|=4 $, $r=(n^2+4)/4=5 $, $l$ ranges from $1$ to $(r-1)/2=2 $, and $a=2$ satisfies $(\frac{a}{r})=-1 $, where $(\frac{a}{r})$ is the Jacobi symbol.}
\label{stpart2}
\end{figure}
\end{example}
\begin{example}
For $G=\Z_8$, there is a generalized Haagerup category satisfying Eq. (\ref{Q2}), whose data $(\epsilon,A) $ is given in the accompanying Mathematica notebook \texttt{solutions.nb}. It is too difficult to compute the modular data exactly, but we have checked numerically that the modular data appears to conform to the table in Figure \ref{stpart2} as well (for $n=8$). 
\end{example}

\begin{example}
For each of $G=\Z_6$ and $G=\Z_{10}$, there are exactly two generalized Haagerup categories which satisfy Eq. (\ref{Q2}). The data $(\epsilon, A) $ for these categories is given in the accompanying Mathematica notebook \texttt{solutions.nb}. We did not compute the modular data exactly, but numerical calculations led to a conjecture summarized by the table in Figure \ref{stpart3}. In each case, the two generalized Haagerup categories for $G$ correspond to the two different possible values of the Jacobi symbol $(\frac{a}{r}) $. 

For these examples, we also have a fermion in $\cZ(\cC) $, so in each case there are $16$ different modular closures of the supermodular subcategory, as above.

\begin{figure}
\resizebox{\textwidth}{!}{%
$
\begin{array}{c !{\vline width 1pt} ccc!{\vline width 1pt}}
\sqrt{n^2+4} \cdot \textbf{S} & & &  \\[3pt]
\Xhline{1pt}
\widetilde{\nu_{\varepsilon'_1}}^{\varepsilon'_2}& \frac{\varepsilon_2 \varepsilon_2' }{2}-\frac{\varepsilon_1\varepsilon_2\varepsilon_1'\varepsilon_2' }{4}\sqrt{n^2+4}& & \Tstrut \\[3pt]
\widetilde{\mu_{+}}^{l',s'}&(-1)^{s'+1}  & (-1)^{ss'+1}2 \cos \frac{4 \pi a l l'}{r} &  \\[3pt]
\widetilde{\mu_{-}}^{m',\varepsilon'}  & -(\frac{a}{r})\varepsilon' \varepsilon_2 &  f(s)\varepsilon^s 2 \cos \frac{4 \pi a m' l}{r}   & -(\frac{a}{r})(-1)^{\frac{(1-\varepsilon)(1-\varepsilon')}{4}} 2 \sin \frac{4 \pi a m m'}{r}    \\[3pt]
 
 \Xhline{1pt}
  & \widetilde{\nu_{\varepsilon_1}}^{\varepsilon_2} & \widetilde{\mu_{+}}^{l,s} & \widetilde{\mu_{-}}^{m,\varepsilon} \\[3pt]
 \Xhline{1pt}
 \textbf{T} &\varepsilon_2  i&  \zeta^s_4\zeta^{al^2}_r  &  \zeta^{(\frac{a}{m}+2)\varepsilon}_8\zeta^{am^2}_r    \Tstrut \\[3pt]
\end{array}
$
}
\caption{Conjecture for the missing corner of modular data for the generalized Haagerup category for $\Z_{4m+2}$. Here $n=|G| $; $r=(n^2+4)/8$; $1 \leq l \leq (r-1)/2 $; $0 \leq s \leq 3$; $0 \leq m \leq r-1$; $f(0)=-1 $, $f(1)=-(\frac{a}{r}) $, $f(2)=1 $, and $f(3)=(\frac{a}{r})$.}
\label{stpart3}
\end{figure}

\end{example}
\section{Tensor product factorization} \label{factor}

In this section we consider a generalized Haagerup category satisfying Eq.(\ref{Q2}) such that $G_2=\Z_2 \times \Z_2 $. In this case, the $ \text{Vec}_{\Z_2 \times \Z_2 }$ subcategory lifts to the center by Lemma \ref{lifts}, and if the braiding on this subcategory is non-degenerate, we can apply M\"uger factorization.
\subsection{M\"uger factorization}
We fix a generalized Haagerup category $\cC\subset \End(M)$ with a finite abelian group $G$ satisfying Eq.(\ref{Q2}). 
We assume that $|G_2|=4$ (i.e. $G_2\cong \Z_2\times \Z_2)$, and the symmetric bicharacter $s(k,l)=\epsilon_k(l)\epsilon_l(k)$ on $G_2\times G_2$ 
is non-degenerate. 
Then the subcategory of $\cZ(\cC)$ generated by $G_2$, which we still denote by $G_2$ for simplicity, 
is a modular tensor category, whose modular data are 
$S^{G_2}_{k,l}=\frac{s(l,k)}{2}$, $T^{G_2}_{k,k}=\epsilon_k(k)$. 
Let 
$$\cZ(\cC)\cap G_2'=\{X\in \cZ(\cC);\;\forall k\in G_2, \cE(k,X)=\cE(X,k)^{-1}\},$$
where $\cE$ is the braiding of $\cZ(\cC)$. 
Then M\"uger's factorization theorem \cite[Theorem 4.2]{MR1990929}) says that this is a modular tensor category with 
$$\cZ(\cC)\cong G_2\boxtimes (\cZ(\cC)\cap G_2').$$

We denote by $(S',T')$ the modular data of $\cZ(\cC)\cap G_2'$. 
Then we have the tensor product factorization $S=S^{G_2}\otimes S'$, $T=T^{G_2}\otimes T'$. 

A simple object $X\in \cZ(\cC)$ belongs to $\Irr(\cZ(\cC)\cap G_2')$ if and only if $s(k,X)=1$ for any $k\in G_2$. 
For $X,Y\in \Irr(\cZ(\cC)\cap G_2')$, we have $S'_{X,Y}=2S_{X,Y}$, $T'_{X,X}=T_{X,X}$. 
 
Let 
$$J_1=\{(k,\chi)\in G_2\times ((\hG)_2\setminus\{0\}\});\;\forall l\in G_2, s(l,k)\inpr{l}{\chi}=1\},$$
$$J_2=\{(g,\tau)\in (G\times \hG)_*;\; \forall k\in G_2, \inpr{k}{\tau}\epsilon_k(g)=1\},$$
$$I_0=\{i\in I;\;\forall k\in G_2, s(k,\widetilde{\mu}^i)=1\}.$$
Then Theorem \ref{4} implies 
$$\Irr(\cZ(\cC)\cap G_2')=\{0,\widetilde{\pi}\}\cup \{k\widetilde{\varphi_{\chi,\varepsilon}}\}_{(k,\chi)\in J_1,\varepsilon\in \{1,-1\}}
\cup\{\widetilde{\sigma_g}^\tau\}_{(g,\tau)\in J_2}\cup \{\widetilde{\mu}^i\}_{i\in I_0}.$$
Since $|I|=4|I_0|$, we have $|I_0|=|G|^2/8$. 

The modular data $(S,T)$ are determined by $(S',T')$, and the latter have been already decided except for the
$I_0\times I_0$ entries. 
In concrete examples, we can often compute them by diagonalizing the multiplication operators of $U_g(k)$ and $\mathbf{t}_{_g\rho}$ on $\cA_{_g\rho}$. 
\subsection{$G_e=\Z_2\times \Z_2$}

We assume that $G=\Z_2\times \Z_2\times G_o$ with odd $G_o$, and that the symmetric bicharacter $s(k,l)$ on 
$G_2\times G_2$ is non-degenerate. 
In this case, we can easily identify the index sets $J_1$ and $J_2$ as follows. 

We identify $\hG$ with $(\hG)_2\times \widehat{G_o}$. 
Then for each $l\in G_2$, there exists a unique $\chi_l\in (\hG)_2$ satisfying $s(k,l)=\inpr{k}{\chi_l}$ 
for any $k\in G_2$. 
We set $\Phi_{k,\varepsilon}=k\widetilde{\varphi_{\chi_k,\varepsilon}}$ for $k\in G_2\setminus \{0\}$. 

For the sets $G_*$ and $\hG_*$, we may assume that there exist subsets $G_{o*}\subset G_o$ and 
$\hG_{o*}\subset \hG_o$ satisfying 
$G_*=(G_2\times G_{o*})$ and $\hG_*=(\hG_2\times \hG_{o*})$. 
Let $\tG=G_2\times G_o\times \hG_o$, and let 
$$\tG_*=(G_2\times \{0\}\times \hG_{0*})\sqcup (G_2\times G_{o*}\times \hG_o)\subset G_2\times G_o\times \hG_o,$$
which satisfies
$$\tG=\tG_2\sqcup \tG_*\sqcup -\tG_*. $$
We identify $G_2$ with $\tG_2$. 

For $(k,h,\tau)\in \tG_*$, we set $\Sigma_{k,h,\tau}=\widetilde{\sigma_{(k,h)}}^{(\chi_k+\epsilon_k,\tau)}$.  
We set $a'=2a=1/2|G_o|=1/\sqrt{|\tG|}$, $b'=2b=1/\sqrt{4|G_o|^2+1}$. 

\begin{theorem} Let the notation be as above. 
The set 
$$\{0,\tpi\}\cup\{\Phi_{k,\varepsilon}\}_{k\in \tG_2\setminus\{0\},\varepsilon\in \{1,-1\}}\cup
\{\Sigma_{k,h,\tau}\}_{(k,h,\tau)\in \tG_*}\cup\{\tilde{\mu}^i\}_{i\in I_0}$$
exhausts all simple objects in $\cZ(\cC)\cap G_2'$. 
Except for $\tmu^i$-$\tmu^{i'}$ entries, the modular data of $\cZ(\cC)\cap G_2'$ are given by
$$S'_{0,0}=S^{\widetilde{\cC}_1}_{\tpi,\tpi}=\frac{a'-b'}{2},\quad S'_{0,\tpi}=\frac{a'+b'}{2},$$
$$S'_{0,\Phi_{k,\varepsilon}}=S'_{\tpi,\Phi_{k,\varepsilon}}=\frac{a'}{2},\quad 
S'_{0,\Sigma_{k,h,\tau}}=S'_{\tpi,\Sigma_{k,h,\tau}}=a',$$ 
$$S'_{0,\tmu^j}=b',\quad S'_{\tpi,\tmu^j}=-b',$$
$$S'_{\Phi_{k,\varepsilon},\Phi_{k',\varepsilon'}}
=s(k,k')\frac{a'+\varepsilon\varepsilon'\delta_{k,k'}}{2},\quad
S'_{\Phi_{k,\varepsilon},\Sigma_{l,h,\tau}}
=s(k,l)a',\quad 
S'_{\Phi_{k,\varepsilon},\tmu^j}=0,$$
$$S'_{\Sigma_{k,h,\tau},\Sigma_{k',h',\tau'}}=s(k,k')(\inpr{h}{\tau'}\inpr{h'}{\tau}+\overline{\inpr{h}{\tau'}\inpr{h'}{\tau}})a',\quad 
S'_{\Sigma_{k,h,\tau},\tmu^j}=0,$$
$$T'_{0,0}=T'_{\tpi,\tpi}=1,\quad 
T'_{\Phi_{k,\varepsilon},\Phi_{k,\varepsilon}}=\epsilon_k(k),\quad 
T'_{\Sigma_{k,h,\tau},\Sigma_{k,h,\tau}}=\epsilon_k(k)\inpr{h}{\tau}.$$
\end{theorem}

\begin{figure}
\resizebox{\textwidth}{!}{%
$
\begin{array}{c !{\vline width 1pt} ccccc!{\vline width 1pt}}
\textbf{S} & & & & &  \\[3pt]
\Xhline{1pt}
0 & \frac{a'-b'}{2} & & & & \\[3pt]
 \widetilde{\pi}  & \frac{a'+b'}{2} &  \frac{a'-b'}{2}& & &  \\[3pt]
  \Phi_{k', \varepsilon'} & \frac{a'}{2} & \frac{a'}{2}   & s(k,k')\frac{a'+\varepsilon \varepsilon' \delta_{k,k'}}{2} & &  \\[3pt]
  \Sigma_{l',h',\tau'} &a' & a' & s(k,l')a' & s(l,l') (\langle h,\tau'  \rangle \langle h,\tau'  \rangle+\overline{\langle h,\tau'  \rangle \langle h,\tau'  \rangle}) a'&  \\[3pt]
   \widetilde{\mu}^{i'} & b' & -b' & 0 & 0 & ?  \\[3pt]
 \Xhline{1pt}
  & 0 &  \widetilde{\pi}  & \Phi_{k, \varepsilon} & \Sigma_{l,h,\tau} & \widetilde{\mu}^i \\[3pt]
 \Xhline{1pt}
 \textbf{T} & 1 & 1 & \epsilon_{k}(k) & \epsilon_l(l) (\langle h,\tau  \rangle & ?  \Tstrut \\[3pt]
\end{array}
$
}
\caption{Modular data for the commutant of $G_2$ for $G_e=G_2=\Z_2  \times \Z_2  $, with the entries labeled by ``?'' undetermined. }
\label{stpart4}
\end{figure}

\begin{example} Let $G= \Z_2 \times \Z_2 $. There is a unique generalized Haagerup category for $G$. The structure constants $(\epsilon, A) $ are given in \cite[Section 9.4]{MR3827808}.  Looking at $\epsilon $, we find that
$$S^{\Z_2\times \Z_2}=\frac{1}{2}\left(
\begin{array}{cccc}
1 &1 &1 &1  \\
1 &1 &-1 &-1  \\
1 &-1 &1 &-1  \\
1 &-1 &-1 &1 
\end{array}
\right)$$
and
$$T^{\Z_2\times \Z_2}=\mathrm{Diag}(1,-1,-1,-1).$$
Since we have $({S^{\Z_2\times \Z_2}})^2=I$ and $(S^{\Z_2\times \Z_2}T^{\Z_2\times \Z_2})^3=-I$, 
the modular group relation for $S'$ and $T'$ take the form ${S'}^2=I$, $(S'T')^3=-I$. 

We can compute the $I_0 \times I_0 $ entries from the tube algebra, following the outline in Section \ref{corner1}. We have $|I_0|=2 $. The two eigenvalues for $\widetilde{\mu}^i $ are $\zeta_5^{\pm 1}$, and the corresponding block of the $S$-matrix is 
$$ 
\frac{1}{10} \left(
\begin{array}{cc}
5+\sqrt{5} & -5+\sqrt{5} \\
-5+\sqrt{5}  & 5+\sqrt{5}  \\
\end{array}
\right).
$$ 

It was pointed out to us by Marcel Bischoff that this example is related to a simple current extension of $SU(5)_5 $ - see \cite[Section 3.3]{MR3764563}.

\end{example}

\subsection{$G=\Z_{2m}\times \Z_2$ with $\epsilon_{(km,l)}(i,j)=(-1)^{kj}$.} 
We assume that $G=\Z_{2m}\times \Z_2$ with $m\geq 2$ and $\epsilon_{(km,l)}(i,j)=(-1)^{kj}$. 
In this case, we have 
$$G_2=\{(0,0),(m,0),(0,1),(m,1)\},$$
$$s((mi,j),(mi',j'))=\epsilon_{(mi,j)}((mi',j'))\epsilon_{(mi',j')}((mi,j))=(-1)^{ij'+i'j},$$
which is non-degenerate on $G_2\times G_2$, and $T_{(mi,j),(mi,j)}=(-1)^{ij}$. 
Thus the modular data $(S^{G_2},T^{G_2})$ are those of $\cZ(\mathrm{Vec}_{\Z_2})$. 
We will show below that $m$ must be even if such a generalized Haagerup category $\cC$ exists. 

We can identify the two index sets $J_1$ and $J_2$ as follows. 
We identify $\widehat{\Z_n}$ with $\Z_n$ via $\inpr{j}{k}=\zeta_n^{jk}$ where 
$\zeta_n=e^{2\pi i/n}$. 
Let $\tG=\Z_{2m}^2$, and let 
$$\tG_*=\{(i,j)\in \Z_{2m}^2;\;i\in \{0,m\}, 0<j<m\}
\cup\{(i,j)\in \Z_{2m}^2;\;0<i<m\},$$
which satisfies 
$$\tG=\tG_2\sqcup \tG_*\sqcup -\tG_*.$$

For $(i,j)\in \Z_2^2\setminus\{0\}$, we set 
$\Phi_{(i,j),\varepsilon}=(mj,mi)\widetilde{\varphi_{(mi,j),\varepsilon}}$. 
For $(i,j)\in \tG_*$, we set 
$\Sigma_{(i,j)}=\widetilde{\sigma}_{(i,j)}^{(j,0)}$. 
Let $a'=2a=\frac{1}{2m}=\frac{1}{\sqrt{|\tG|}}$, $b'=2b=\frac{1}{\sqrt{4m^2+1}}$. 

\begin{theorem}\label{even} With the above assumption the natural number $m$ is always even. 
The following set exhausts the simple objects of $\cZ(\widetilde{\cC}\cap G_2')$: 
$$\{0,\;\tpi\}\cup\{\Phi_{(i,j,\varepsilon)}\}_{(i,j)\in \Z_2^2\setminus \{0\},\varepsilon\in \{1,-1\}}\cup 
\{\Sigma_{(i,j)}\}_{(i,j)\in \tG_*}\cup \{\tmu^i\}_{i\in I_0},$$
and they are all self-conjugate.  
Except for $\tmu^i$-$\tmu^{i'}$ entries, the modular data are given as 
$$S'_{0,0}=S'_{\tpi,\tpi}=\frac{a'-b'}{2},\quad S'_{0,\tpi}=\frac{a'+b'}{2},
\quad S'_{0,\Phi_{(i,j),\varepsilon}}=S'_{\tpi,\Phi_{(i,j),\varepsilon}}=\frac{a'}{2},$$
$$S'_{0,\Sigma_{(i,j)}}=S'_{\tpi,\Sigma_{(i,j)}}=a',\quad 
S'_{0,\tmu^i}=b',\quad 
S'_{\tpi,\tmu^j}=-b',$$
$$S'_{\Phi_{(i,j),\varepsilon},\Phi_{(i',j'),\varepsilon'}}=\frac{a'+\delta_{i,i'}\delta_{j,j'}\varepsilon\varepsilon'}{2},
\quad S'_{\Phi_{(i,j),\varepsilon},\Sigma_{(i',j')}}=(-1)^{ii'+jj'}a',
\quad S'_{\Phi_{(i,j),\varepsilon},\tmu^k}=0,$$

$$S'_{\Sigma_{(i,j)},\Sigma_{(i',j')}}=2a'\cos \frac{(ij'+i'j)\pi}{m},\quad 
S'_{\Sigma_{(i,j)},\tmu^k}=0,$$
$$T'_{0,0}=T'_{\tpi,\tpi}=T'_{\Phi_{(i,j),\varepsilon},\Phi_{(i,j),\varepsilon}}=1,\quad 
T'_{\Sigma_{(i,j)},\Sigma_{(i,j)}}=\zeta_{2m}^{ij}.$$
\end{theorem}

\begin{proof} The above formulae for the modular data follow from Theorem \ref{4} except that it instead gives 
$$S'_{\Phi_{(i,j),\varepsilon},\Phi_{(i',j'),\varepsilon'}}=(-1)^{m(ij'+i'j)}\frac{a'+\delta_{i,i'}\delta_{j,j'}\varepsilon\varepsilon'}{2},$$
$$T'_{\Phi_{(i,j),\varepsilon},\Phi_{(i,j),\varepsilon}}=(-1)^{mij}.$$

Note that we have $(ST)^3=I$ because every object in $\cZ(\cC)$ is self-conjugate thank to Theorem \ref{4}. 
Since $(S^{G_2}T^{G_2})^3=I$, we get $(S'T')^3=I$. 
To show that $m$ is even, we compute the $\Phi_{(i,j),\varepsilon}$-$\Phi_{(i',j'),\varepsilon'}$ entries of 
$$S'T'S'=(T'S'T')^{-1}=\overline{T'S'T'}.$$
We have 
\begin{align*}
\lefteqn{(S'T'S')_{\Phi_{(i,j),\varepsilon},\Phi_{(i',j'),\varepsilon'}}=\sum_{x}S_{\Phi_{(i,j),\varepsilon},x}T'_{x,x}S_{\Phi_{(i',j'),\varepsilon'},x} } \\
 &=\frac{a'^2}{4}+\frac{a'^2}{4}\\
 &+\sum_{(i'',j'',\varepsilon'')\in J_1\times \{1,-1\}}(-1)^{m(ij''+ji''+i'j''+j'i''+i''j'')}
 \frac{(a'+\varepsilon\varepsilon''\delta_{i,i''}\delta_{j,j''})(a'+\varepsilon'\varepsilon''\delta_{i',i''}\delta_{j',j''})}{4}\\
 &+a'^2\sum_{(p,q)\in J_2}(-1)^{ip+jq+i'p+j'q}\zeta_{2m}^{pq}\\
 &=\frac{a'^2}{2}
 +\sum_{(i'',j'')\in \Z_2^2\setminus\{0\}}(-1)^{m(ij''+ji''+i'j''+j'i''+i''j'')}
 \frac{a'^2+\varepsilon\varepsilon'\delta_{i,i''}\delta_{j,j''}\delta_{i',i''}\delta_{j',j''}}{2}\\
 &+\frac{a'^2}{2}\sum_{(p,q)\in \tG\setminus \tG_2}(-1)^{ip+jq+i'p+j'q}\zeta_{2m}^{pq}\\
 &=\sum_{(i'',j'')\in \Z_2^2}(-1)^{m((i+i')j''+(j+j')i''+i''j'')}
 \frac{a'^2+\varepsilon\varepsilon'\delta_{i,i''}\delta_{j,j''}\delta_{i',i''}\delta_{j',j''}}{2}\\
 &+\frac{a'^2}{2}(-\sum_{(p,q)\in \tG_2}(-1)^{(i+i')p+(j+j')q}\zeta_{2m}^{pq}
 +\sum_{(p,q)\in \tG}(-1)^{(i+i')p+(j+j')q}\zeta_{2m}^{pq})\\
 &=\frac{(-1)^{mij}\varepsilon\varepsilon'\delta_{i,i'}\delta_{j,j'}}{2}
 +\frac{a'^2(-1)^{m(i+i')(j+j')}}{2}\sum_{(p,q)\in \tG}\zeta_{2m}^{(p+m(j+j'))(q+m(i+i'))}\\
 &=(-1)^{m(i+i')(j+j')}\frac{a'+(-1)^{mij}\varepsilon\varepsilon'\delta_{i,i'}\delta_{j,j'}}{2}.\\
\end{align*}
This coincides with $(\overline{T'S'T'})_{\Phi_{(i,j),\varepsilon},\Phi_{(i',j'),\varepsilon'}}$ if and only if $m$ is even. 
\end{proof}

\begin{figure}
$
\begin{array}{c !{\vline width 1pt} ccccc!{\vline width 1pt}}
\textbf{S} & & & & &  \\[3pt]
\Xhline{1pt}
0 & \frac{a'-b'}{2} & & & & \\[3pt]
 \widetilde{\pi}  & \frac{a'+b'}{2} &  \frac{a'-b'}{2}& & &  \\[3pt]
  \Phi_{(h',j'),\varepsilon'} & \frac{a'}{2} & \frac{a'}{2}   &\frac{a'+ \delta_{(h,j),(h',j')}\varepsilon \varepsilon' }{2} & &  \\[3pt]
  \Sigma_{(k',l')} &a' & a' &(-1)^{hk'+jl'}a' &2a' \cos \frac{(kl'+k'l)\pi}{m} &  \\[3pt]
   \widetilde{\mu}^{i'} & b' & -b' & 0 & 0 & ?  \\[3pt]
 \Xhline{1pt}
  & 0 &  \widetilde{\pi}  & \Phi_{(h,j),\varepsilon} & \Sigma_{(k,l)} & \widetilde{\mu}^i \\[3pt]
 \Xhline{1pt}
 \textbf{T} & 1 & 1 &1 & \zeta_{2m}^{ij} & ?  \Tstrut \\[3pt]
\end{array}
$
\caption{Modular data for the commutant of $G_2=\Z_2  \times \Z_2$ for $G=\mathbb{Z}_{2m} \times \mathbb{Z}_2$ with $\epsilon_{(km,l)}(i,j)=(-1)^{kj} $,  with the entries labeled by ``?'' undetermined. }
\label{stpart5}
\end{figure}

Since $\epsilon_{(0,1)}((0,1))=1 $, we can de-equivariantize $\cC$ to get another fusion category
$\widetilde{\cC}_1$ realized as endomorphisms of $P=M\rtimes_{\alpha_{(0,1)}}\Z_2$, with implementing unitary  $\lambda \in P$. 
Then 
 the set 
$$\{\alpha'_{(g,0)}\}_{g\in \Z_{2m}}\cup \{\alpha'_{(g,0)}\circ \rho'\}_{g\in \Z_{2m}}$$
exhausts all simple objects in $\widetilde{\cC}_1$ as we have $\alpha'_{(0,1)}=\Ad \lambda$, and we have 
$$\alpha'_{(g,0)}\circ \alpha'_{(h,0)}=\alpha'_{(g+h,0)},$$
$$\alpha'_{(-g,0)}\circ \rho'=\rho'\circ \alpha'_{(g,0)},$$
$$\rho'^2(x)=sxs^*+\sum_{g\in \Z_{2m}}(t_{(g,0)}\alpha'_{(g,0)}\circ\rho'(x)t_{(g,0)}^*+t_{(g,1)}\lambda
\alpha'_{(g,0)}\circ\rho'(x)(t_{(g,1)}\lambda)^*).$$ 
We denote by $\cP_{4m+1}$ the $C^*$-algebra generated by
$$\{s\}\cup \{t_{(g,0)}\}_{g\in \Z_{2m}}
\cup\{t_{(g,1)}\lambda\}_{g\in \Z_{2m}},$$
which satisfies the Cuntz algebra relations. 
Note that $\alpha'$ and $\rho'$ globally preserve $\cP_{4m+1}$. 
 
Let $\beta$ be the dual action of $\alpha_{(0,1)}$, which is a period 2 automorphism of $P$ satisfying 
$\beta(x)=x$ for any $x\in M$, and $\beta(\lambda)=-\lambda$. 
Let $\gamma=\alpha'_{(m,0)}\circ \beta$. Then $\gamma $ commutes with $\alpha' $ and $\rho' $, and therefore induces a $\Z_2 $-action on $\widetilde{\cC}_1$. The equivariantization of $\widetilde{\cC}_1$ with respect to this action is equivalent to $\cC $. 

Let $\widetilde{\cC}_2$ the fusion category generated by $\gamma$, 
which is equivalent to $\mathrm{Vec}_{\Z_2}$, 
and let $\widetilde{\cC}$ be the fusion category generated by $\widetilde{\cC}_1$ and $\widetilde{\cC}_2$. 
Then $\widetilde{\cC}$ is the crossed product category for the $\Z_2 $-action on $\widetilde{\cC}_1$ induced by $\gamma $. Therefore $\widetilde{\cC}$ is Morita equivalent to $\cC$, and their Drinfeld centers are braided equivalent. 

\begin{theorem} \label{ahthm}
The fusion category $\widetilde{\cC}$ is equivalent to $\widetilde{\cC}_1\boxtimes \widetilde{\cC}_2$. 
In consequence, 
$$\cZ(\cC)\cong \cZ(\widetilde{\cC}_1)\boxtimes \cZ(\mathrm{Vec}_{\Z_2}),$$
and the modular data of $\cZ(\widetilde{\cC}_1)$ are $(S',T')$. 
\end{theorem}

\begin{proof} 
Every intertwiner between products of objects of objects $\alpha'_{(g,0)}$ and $\rho'$ belongs to $\cP_{4m+1}$, and 
$\gamma$ acts on $\cP_{4m+1}$ trivially. 
On the other hand, every intertwiner between products of $\gamma$ is a scalar, and so we get 
the splitting 
$$ \widetilde{\cC}\cong \widetilde{\cC}_1\boxtimes \widetilde{\cC}_2.$$
\end{proof}

As a consequence of Theorem \ref{ahthm}, the modular data of the Drinfeld center of the de-equivariantization $\widetilde{\cC}_1 $ is determined by Theorem \ref{even}, except for the $I_0 \times I_0 $ corner. 

\begin{example} Let $G=\Z_4\times \Z_2$. It was shown in \cite{MR3859276} that there is a generalized Haagerup category $\cC$ for $G$ with $\epsilon_{(2k,l)}(i,j)=(-1)^{kj}$ such that the even parts of the Asaeda-Haagerup subfactor are Morita equivalent to the $\Z_2 $-de-equivariantization $\widetilde{\cC}_1 $. Therfore the modular data for the Drinfeld center of the Asaeda-Haagerup categories is given in part by the table in Figure \ref{stpart5}. To find the missing corner, we can work in the tube algebra of the generalized Haagerup category $\cC$, or directly in the tube algebra of the de-equivariantization $\widetilde{\cC}_1 $.

We have $|I_0|=8 $, and we find that the $T$-eigenvalues for $\widetilde{\mu}^i $ are of the form $\zeta_{17}^{3i^2} $, for $ 1 \leq i \leq 8 $. Since these numbers are distinct (and also different from the $T$-eigenvalues for the other objects), we can compute the corresponding projections in the tube algebra as eigenprojections of $\textbf{t} $. Then the corresponding block of the $S$-matrix is given by the formula
$$S_{\widetilde{\mu}^i,\widetilde{\mu}^{i'} }= -\frac{2}{\sqrt{17}} \cos \frac{12 \pi i i'}{17}  .$$
\end{example}
We do not know whether there exist further examples of ``generalized Asaeda-Haagerup categories'', i.e. generalized Haagerup categories for $\Z_{4m} \times \Z_2 $ with $\epsilon_{(2mk,l)}(i,j)=(-1)^{kj}$ for $m > 1 $.

\section{$\Z_2$-de-equivariantization} \label{deeq}
In this section we consider $ \Z_2$-de-equivariantizations of generalized Haagerup categories for cyclic groups. In all known examples of generalized Haagerup categories for even cyclic groups, $\epsilon $ is non-trivial. So we can only de-equivariantize with respect to $G_2$ when $|G|$ is divisible by $4$.

Therefore we consider a generalized Haagerup category $\cC $ for  $G=\Z_{4m}$ and assume $\epsilon_{2m}(g)=(-1)^g$, meaning $\alpha_{2m}(t_g)=(-1)^gt_g$. 
In this case, we may assume $\alpha_h(t_0)=t_{2h}$, $\alpha_h(t_1)=t_{1+2h}$  for $0\leq h<2m$. 
Since $\epsilon_{2m}(2m)=(-1)^{2m}=1$, we can perform de-equivariantization by $\alpha_{2m}$. 
We extend $\rho\in \End(M)$ and $\alpha_g\in \Aut(M)$ to the crossed product 
$M\rtimes_{\alpha_{2m}}\Z_2=M\vee \{\lambda\}$ by 
$\rho'(\lambda)=\lambda$ and $\alpha'_g(\lambda)=(-1)^g\lambda$. 
We denote by $\cD$ the fusion category generated by $\rho'$, which is a $\Z_2 $-de-equivariantization of $\cC $. 
Note that we have $\alpha'_{2m}=\Ad \lambda$, and the set 
$$\{\alpha'_g\}_{0\leq g<2m}\cup \{\alpha'_g\circ \rho'\}_{0\leq g<2m}$$
exhausts the equivalence classes of irreducible objects in $\cD$. 
We have decomposition 
$${\rho'}^2(x)=sxs^*+\sum_{g=0}^{2m-1}(t_g\alpha'_g\circ \rho'(x)t_g^*+t_{g+2m}\lambda \alpha'_g\circ \rho'(x)(t_{g+2m}\lambda)^*)$$
 
Let $d=d(\rho')=2m+\sqrt{4m^2+1}$, and let $\Lambda'$ be the global dimension of $\cD$: 
$$\Lambda'=2m(1+d^2)=2m(2+4md)=4m(4m^2+1+2m\sqrt{4m^2+1}).$$
Setting $a'=1/2m$ and $b'=1/\sqrt{4m^2+1}$, we get $1/\Lambda'=(a'-b')/2$. 

To represent an element in the cyclic group $\Z_{2m}$, we always use a number $0\leq g<2m$. 
We denote $(g\;h|1|h\;g)=(\alpha'_g\; \alpha'_h|1|\alpha'_h\;\alpha'_g)$ for short.
Since 
$$(g\;h|1|h\;g)(g\;k|1|k\;g)=\left\{
\begin{array}{ll}
(g\;h+k|1|h+k\;g) , &\quad h+k<2m \\
(-1)^g(g\;h+k-2m|1|h+k-2m\;g) , &\quad h+k\geq 2m
\end{array}
\right.,
$$
the map $$\Z_{2m}\ni h\mapsto \zeta_{4m}^{gh}(g\;h|1|h\;g)\in \cA_g$$ is a representation, where 
$\zeta_k=e^{\frac{2\pi i}{k}}$. 
Let 
$$p(g,k)=\frac{1}{2m}\sum_{h=0}^{2m-1}\zeta_{4m}^{gh}\zeta_{2m}^{hk}(g\;h|1|h\;g).$$
Then $p(g,k)\in \cA_g$ is a projection. 

For $\cA_0$, we have  
$$(0\;h|1|h\;0)(0\;\rho'|1|\rho'\;0)=(0\;\rho'|1|\rho'\;0)(0\;2m-h|1|2m-h\;0),$$
$$(0\;\rho'|1|\rho'\;0)^2=1_{\rho'}+\sum_{h=0}^{2m-1}2(0\;h|1|h\;0)(0\;\rho'|1|\rho'\;0)
=1_{\rho'}+4mp(0,0)(0\;\rho'|1|\rho'\;0)).$$ 
Thus we have $p(0,k)(0\;\rho'|1|\rho'\;0)=(0\;\rho'|1|\rho'\;0)p(0,2m-k)$ for $0<k<m$, 
and the linear span of 
$$\{p(0,0),\;p(0,m),\;p(0,0)(0\;\rho'|1|\rho'\;0),\;p(0,m)(0\;\rho'|1|\rho'\;0)\}$$
is a subalgebra of $\cA_0$ isomorphic to $\C^4$. 
Let 
$$z(\widetilde{\id})=\frac{1}{\Lambda'}\sum_{g\in G}((0\;g|1|g\;0)+d(0\;_g\rho'|1|_g\rho'\;0))
=\frac{2m}{\Lambda'}p(0,0)(1_0+d(0\;\rho'|1|\rho'\;0)),$$
$$E(0,0)=\frac{2md}{\Lambda'}p(0,0)(d1_0-(0\;\rho'|1|\rho'\;0)),$$
$$E(0,m)_{\pm}=\frac{1}{2}p(0,m)(1_0\pm (0\;\rho'|1|\rho'\;0)).$$
Then these are the minimal projections of the subalgebra. 
For $0< k<m$, we set 
$$E(0,k)_{11}=p(0,k),\quad E(0,k)_{22}=p(0,2m-k),$$
$$E(0,k)_{12}=p(0,k)(0\;\rho'|1|\rho'\;0),$$
$$E(0,k)_{21}=p(0,2m-k)(0\;\rho'|1|\rho'\;0),$$
and set $\cA_0^k=\mathrm{span}\{E(0,k)_{ij}\}_{1\leq i,j,\leq 2}$. 
Then $\cA_0^k$ is isomorphic to the 2 by 2 matrix algebra with a system 
of matrix units $\{E(0,k)_{ij}\}_{1\leq i,j\leq 2}$. 
Now we have 
$$\cA_{0}=\C z(\widetilde{\id})\oplus \C E(0,0) \oplus \C E(0,m)_+\oplus \C E(0,m)_{-}\oplus 
\bigoplus_{0<k<m}\cA^k_0.$$

The algebra $\cA_0$ acts on $\cA_{0,_g\rho'}$ by left multiplication, and 
$p(0,k)\cA_{0,_g\rho'}$ is a 2-dimensional space with a basis 
$$\{p(0,k)(0\;\rho'|t_{-g}|\rho'\; _g\rho'),\; p(0,k)(0\;\rho'|t_{2m-g}\lambda|\rho'\;_g\rho')\}.$$
Note that we have $(p(0,m)(0\;\rho'|1|\rho'\;0))^2=p(0,m)$. 

\begin{lemma} $(0\;\rho'|1|\rho'\;0)$ acts on $p(0,m)\cA_{0,_g\rho'}$ as multiplying by $(-1)^g$. 
\end{lemma}

\begin{proof} Since the two elements  
$$\{p(0,m)(0\;\rho'|t_{-g}|\rho'\; _g\rho'), p(0,m)(0\;\rho'|t_{2m-g}\lambda|\rho'\; _g\rho')\},$$
are exchanged, up to scalar multiple, by right multiplication of $(_g\rho'\;m|\lambda|m\;_g\rho')$, 
it suffices to show 
$$(0\;\rho'|1|\rho'\;0)p(0,m)(0\;\rho'|t_{-g}|\rho'\; _g\rho')=(-1)^g(0\;\rho'|t_{-g}|\rho'\; _g\rho').$$
Indeed, 
\begin{align*}
\lefteqn{(0\;\rho'|1|\rho'\;0)p(0,m)(0\;\rho'|t_{-g}|\rho'\; _g\rho')=p(0,m)(0\;\rho'|1|\rho'\;0)(0\;\rho'|t_{-g}|\rho'\; _g\rho')} \\
 &=p(0,m)\sum_{h=0}^{2m-1}(0\;_h\rho'|t_h^*\rho(t_{-g})t_h
 +\lambda^{-1}t_{h+2m}^*\rho(t_{-g})t_{h+2m}\lambda|_h\rho'\;_g\rho') \\
 &=p(0,m)\sum_{h=0}^{2m-1}(0\;h|1|h\;0)(0\;\rho'|\alpha_{-h}(t_h^*\rho(t_{-g})t_h
 +(-1)^gt_{h+2m}^*\rho(t_{-g})t_{h+2m})|\rho'\;_g\rho') \\
  &=p(0,m)\sum_{h=0}^{2m-1}(-1)^h\epsilon_h(-g)(0\;\rho'|t_{-h}^*\rho(t_{2h-g})t_{-h}
 +(-1)^gt_{-h+2m}^*\rho(t_{2h-g})t_{-h+2m}|\rho'\;_g\rho') \\
 &=p(0,m)(0\;\rho'|\sum_{h=0}^{4m-1}(-1)^h\epsilon_h(-g)t_{-h}^*\rho(t_{2h-g})t_{-h}|\rho'\;_g\rho').
\end{align*}
Here we have 
\begin{align*}
\lefteqn{\sum_{h=0}^{4m-1}(-1)^h\epsilon_h(-g)t_{-h}^*\rho(t_{2h-g})t_{-h}} \\
 &= \sum_{h=0}^{4m-1}(-1)^h\epsilon_h(-g)\epsilon_{2h-g}(g-2h) t_{-h}^*\alpha_{g-2h}\rho(t_{g-2h})t_{-h}\\
 &=\sum_{h=0}^{4m-1}(-1)^h\epsilon_h(-g)\epsilon_{2h-g}(g-2h)A_{g-2h}(h-g,h-g)t_{-g}.
\end{align*}
Since 
$$A_{g-2h}(h-g,h-g)=\epsilon_{g-h}(-h)A_{g-2h}(g-h,0)=\delta_{g,h}-\frac{\epsilon_{g-h}(-g)}{d-1},$$
we get 
\begin{align*}
\lefteqn{\sum_{h=0}^{4m-1}(-1)^h\epsilon_h(-g)\epsilon_{2h-g}(g-2h)A_{g-2h}(h-g,h-g)} \\
 &=(-1)^g-\frac{1}{d-1}\sum_{h=0}^{4m-1}(-1)^h\epsilon_h(-g)\epsilon_{2h-g}(g-2h)\epsilon_{g-h}(-g)\\
 &=(-1)^g-\frac{1}{d-1}\sum_{h=0}^{4m-1}(-1)^h\epsilon_h(-g)\epsilon_h(g-2h) \epsilon_{h-g}(g)\epsilon_{g-h}(-g)\\
 &=(-1)^g-\frac{\epsilon_{-g}(g)\epsilon_g(-g) }{d-1}\sum_{h=0}^{4m-1}(-1)^h\\
 &=(-1)^g,
\end{align*}
which shows the statement. 
\end{proof}

Now we treat the case with $g\neq 0$. 
Note that $p(g,2m)$ is well defined and equal to $p(g,0)$. 

\begin{lemma} For $0<g< 2m$ and $0\leq k\leq 2m$, 
$$p(g,k)(g\;\rho|\lambda|\rho\;2m-g)=(g\;\rho|\lambda|\rho\;2m-g)p(2m-g,2m-k).$$
\end{lemma}

\begin{proof} 
On one hand, we have 
\begin{align*}
\lefteqn{p(g,k)(g\;\rho'|\lambda|\rho'\;2m-g)} \\
 &=\frac{1}{2m}\sum_{h=0}^{2m-1}\zeta_{4m}^{gh}\zeta_{2m}^{hk}(g\;h|1|h\;g)(g\;\rho'|\lambda|\rho'\;2m-g) \\
 &=\frac{1}{2m}\sum_{h=0}^{2m-1}\zeta_{4m}^{gh}\zeta_{2m}^{hk}(g\;_h\rho'|\alpha'_h(\lambda)|_h\rho'\;2m-g) \\
 &=\frac{1}{2m}\sum_{h=0}^{2m-1}\zeta_{4m}^{gh}\zeta_{2m}^{hk}(-1)^h(g\;_h\rho'|\lambda|_h\rho'\;2m-g). 
\end{align*}
On the other hand, 
\begin{align*}
\lefteqn{(g\;\rho'|\lambda|\rho'\;2m-g)p(2m-g,2m-k)} \\
 &=\frac{1}{2m}\sum_{h=0}^{2m-1}\zeta_{4m}^{(2m-g)h}\zeta_{2m}^{h(2m-k)}(g\;\rho'|\lambda|\rho'\;2m-g)(2m-g\;h|1|h\;2m-g)  \\
 &=\frac{1}{2m}(g\;\rho'|\lambda|\rho'\; 2m-g)1_{2m-g}\\
 &+\frac{1}{2m}\sum_{h=1}^{2m-1}\zeta_{4m}^{(2m-g)h}\zeta_{2m}^{h(2m-k)}
 (g\;_{2m-h}\rho'|\lambda^{-1}\lambda\alpha'_g(\lambda)|_{2m-h}\rho'\;2m-g) \\
 &=\frac{1}{2m}(g\;\rho'|\lambda|\rho'\; 2m-g)\\
 &+\frac{1}{2m}\sum_{h=1}^{2m-1}\zeta_{4m}^{(2m-g)(2m-h)}\zeta_{2m}^{(2m-h)(2m-k)}(-1)^g
 (g\;_h\rho'|\lambda|_h\rho'\;2m-g) \\
 &=\frac{1}{2m}\sum_{h=0}^{2m-1}\zeta_{4m}^{gh}\zeta_{2m}^{hk}(-1)^h
 (g\;_h\rho'|\lambda|_h\rho'\;2m-g), \\
\end{align*}
which shows the statement.
\end{proof}

For $g=m$, the linear span of 
$$\{p(m,0),\;p(m,m),\; p(m,0)(m\;\rho'|\lambda|\rho'\;m),\;p(m,m)(m\;\rho'|\lambda|\rho'\;m)\},$$ 
is a commutative subalgebra of $\cA_m$ isomorphic to $\C^4$. 
Note that we have 
\begin{align*}
\lefteqn{(m\;\rho'|\lambda|\rho'\;m)^2} \\
 &=(m\;0|s^*\rho'(\lambda)\lambda\alpha_m(s)|0\;m)\\
 &+\sum_{h=0}^{2m-1}(m\;_h\rho'|t_h^*\rho'(\lambda)\lambda\alpha_m(t_h)
 +\lambda^{-1}t_{h+2m}^*\rho'(\lambda)\lambda\alpha'_m(t_{h+2m}\lambda)|_h\rho'\;m) \\
 &=1_m+\sum_{h=0}^{2m-1}(m\;_h\rho'|t_h^*  \alpha_m(t_h)
 +\lambda^{-1}t_{h+2m}^*\alpha_m(t_{h+2m})\lambda|_h\rho'\;m)\\
 &=1_m.
\end{align*}
Let
$$E(m,0)_{\pm}=\frac{1}{2}p(m,0)(1_m \pm (m\;\rho'|\lambda|\rho'\;m)),$$
$$E(m,m)_{\pm}=\frac{1}{2}p(m,m)(1_m \pm (m\;\rho'|\lambda|\rho'\;m)).$$
Then they are the minimal projections of the subalgebra. 
For $0<k<m$, we set 
$$E(m,k)_{11}=p(m,k),\quad E(m,k)_{22}=p(m,2m-k),$$
$$E(m,k)_{12}=p(m,k)(m\;\rho|\lambda|\rho\;m),$$
$$E(m,k)_{21}=p(m,2m-k)(m\;\rho|\lambda|\rho\;m).$$
and set $\cA_m^k=\mathrm{span}\{E(m,k)_{ij}\}_{1\leq i,j,\leq 2}$. 
Then $\cA_m^k$ is isomorphic to the 2 by 2 matrix algebra with a system 
of matrix units $\{E(m,k)_{ij}\}_{1\leq i,j\leq 2}$. 
Now we have 
$$\cA_m=\C E(m,0)_+\oplus \C E(m,0)_- \oplus \C E(m,m)_+\oplus \C E(m,m)_{-}\oplus 
\bigoplus_{0<k<m}\cA^k_m.$$

The algebra $\cA_m$ acts on $\cA_{m,_g\rho'}$ by left multiplication, and 
$p(m,k)\cA_{m,_g\rho'}$ is a 2-dimensional space with a basis 
$$\{p(m,k)(m\;\rho'|t_{m-g}|\rho'\; _g\rho'),\; p(m,k)(m\;\rho'|t_{3m-g}\lambda|\rho'\;_g\rho')\}.$$

\begin{lemma} Let the notation be as above. 
\begin{itemize}
\item[(1)] The action of $(m\;\rho'|\lambda|\rho'\;m)$ on $p(m,0)\cA_{m,_g\rho'}$ has eigenvalues both 1 and -1.
\item[(2)] The action of $(m\;\rho'|\lambda|\rho'\;m)$ on $p(m,m)\cA_{m,_g\rho'}$ has eigenvalues both 1 and -1. 
\end{itemize}
\end{lemma}

\begin{proof} Since $(m\;\rho'|\lambda|\rho' \;m)$ acts as an invertible transformation of period 2, 
it suffices to show that it is not a scalar. 
Indeed, it is easy to show that $(m\;\rho'|\lambda|\rho' \;m)$ switches the two basis elements 
(up to scalar multiple). 
\end{proof}

Finally, for $0<g<m$ and $0\leq k<2m$, we set 
$$E(g,k)_{11}=p(g,k),\quad E(g,k)_{22}=p(2m-g,2m-k),$$
$$E(g,k)_{12}=p(g,k)(g\;\rho'|\lambda|\rho'\;2m-g),$$
$$E(g,k)_{21}=p(2m-g,2m-k)(2m-g\;\rho'|\lambda|\rho'\;g),$$
and $\cB_g^k=\mathrm{span}\{E(g,k)_{ij}\}_{1\leq i,j,\leq 2}$. 
Then $\cB_g^k$ is isomorphic to the 2 by 2 matrix algebra with a system 
of matrix units $\{E(g,k)_{ij}\}_{1\leq i,j\leq 2}$. 
Now we have 
$$\cA_g\oplus \cA_{g,2m-g}\oplus \cA_{2m-g,g}\oplus \cA_{2m-g}= 
\bigoplus_{0\leq k<2m}\cB^k_g.$$

Let 
$$\pi=\id\oplus 2\bigoplus_{g=0}^{2m-1}\alpha'_g\circ \rho',$$
$$\varphi_+=\id\oplus2\bigoplus_{g=0}^{m-1}\alpha'_{2g}\circ \rho',$$
$$\varphi_-=\id\oplus2\bigoplus_{g=0}^{m-1}\alpha'_{2g+1}\circ \rho',$$
$$\psi=\alpha'_m\oplus \bigoplus_{g=0}^{2m-1}\alpha'_g\circ\rho',$$
$$\mu=2\bigoplus_{g=0}^{2m-1}\alpha'_g\circ\rho',$$
$$\sigma_0=\id\oplus \id\oplus 2\bigoplus_{g=0}^{2m-1} \alpha'_g\circ\rho'.$$
For $0<g\leq m$, let 
$$\sigma_g=\alpha'_g\oplus \alpha'_{2m-g}\oplus 2\bigoplus_{h=0}^{2m-1} \alpha'_h\circ\rho'.$$
Summing up the above argument, we get 

\begin{lemma}
\begin{itemize}
\item [(1)] $\id$ has a unique half-braiding $\cE_0(\xi)=1$.  

\item [(2)]  $\pi$ has a unique half-braiding, which gives $e(\tpi)_{0,0}=E(0,0)$, and 
$$\cE_\pi(h)_{0,0}=1,$$
$$\cE_\pi(_h\rho')_{0,0}=-\frac{1}{d^2}.$$
 
\item [(3)] Each of $\varphi_+$ and $\varphi_-$ has a unique half-braiding, which gives $e(\widetilde{\varphi_\pm})_{0,0}=E(0,m)_\pm$ and
$$\cE_{\varphi_\pm}(h)_{0,0}=(-1)^h,$$
$$\cE_{\varphi_\pm}(_h\rho')_{0,0}=\frac{\pm(-1)^h}{d}.$$

\item [(4)] $\psi$ has exactly 4 half-braidings parametrized by the set 
$$\{(+,+),\;(+,-),\;(-,+),\;(-,-)\},$$
which gives 
$e(\tpsi^{(\varepsilon_1,\varepsilon_2)})_{m,m}=E(m,\frac{1-\varepsilon_1}{2}m)_{\varepsilon_2}$, and 
$$\cE_{\psi}^{(\varepsilon_1,\varepsilon_2)}(h)_{m,m}=(\varepsilon_1 i)^h,$$
$$\cE_{\psi}^{(\varepsilon_1,\varepsilon_2)}(_h\rho')_{m,m}=\frac{\varepsilon_2 (-\varepsilon_1 i)^h}{d}\lambda,$$
$$\cE_{\psi}^{(\varepsilon_1,\varepsilon_2)}(m)_{_g\rho',_g\rho'}=(-1)^m\varepsilon_2a^{\varepsilon_1}(g)(\varepsilon_1i)^{m+g} \lambda.$$
Here we identify the symbols $+$ with $1$ and $-$ with $-1$ in an appropriate way. 
The number $a^{\varepsilon_1}(g)\in \{1,-1\}$ satisfies $a^\varepsilon(g+2)=a^\varepsilon(g)$. 

\item [(5)] $\sigma_0$ has exactly $m-1$ half-braidings parametrized by $0<k<m$, which gives 
$e(\widetilde{\sigma_0}^k)_{(0,s),(0,t)}=E(0,k)_{st}$, and  
$$\cE_{\sigma_0}^k(h)_{(0,1),(0,1)}=\zeta_{2m}^{kh},$$
$$\cE_{\sigma_0}^k(_h\rho')_{(0,s),(0,s)}=0.$$

\item [(6)] $\sigma_m$ has exactly $m-1$ half-braidings parametrized by $0<k<m$, 
which give $e(\widetilde{\sigma_m}^k)_{(m,s),(m,t)}=E(m,k)_{st}$, and 
$$\cE_{\sigma_m}^k(h)_{(m,1),(m,1)}=i^h\zeta_{2m}^{kh},$$
$$\cE_{\sigma_m}^k(_h\rho')_{(m,s),(m,s)}=0.$$
 
\item [(7)] For $0<g<m$, 
$\sigma_g$ has exactly $2m$ half-braidings parametrized by $0\leq k<2m$, which give 
$e(\widetilde{\sigma_g}^k)_{g,g}=p(g,k)$, $e(\widetilde{\sigma_g}^k)_{2m-g,2m-g}=p(2m-g,2m-k)$, and 
$$\cE_{\sigma_g}^k(h)_{g,g}=\zeta_{4m}^{gh}\zeta_{2m}^{kh},$$
$$\cE_{\sigma_g}^k(h)_{2m-g,2m-g}=\zeta_{4m}^{(2m-g)h}\zeta_{2m}^{(2m-k)h},$$
$$\cE_{\sigma_g}^k(_h\rho')_{g,g}=0.$$ 
\end{itemize}
\end{lemma}

\begin{proof} 
The only statement that we haven't shown yet is about $\cE_\psi^{(\varepsilon_1,\varepsilon_2)}(m)_{_g\rho',_g\rho'}$. 
We first note that $i^{m+g}(_g\rho'\;m|\lambda|m\;_g\rho')$ is a period two unitary in $\cA_{_g\rho'}$ satisfying
\begin{align*}
\lefteqn{p(m,\frac{1-\varepsilon}{2}m)(m\;\rho'|t_{m-g}|\rho'\;_g\rho')i^{m+g}(_g\rho'\;m|\lambda|m\;_g\rho')} \\
 &=i^{m+g}p(m,\frac{1-\varepsilon}{2}m)(m\;_m\rho'|\lambda^{-1}\rho'(\lambda)t_{m-g}\alpha'_m(\lambda)|_m\rho'\;_g\rho') \\
 &=(-1)^mi^{m+g}p(m,\frac{1-\varepsilon}{2}m)(m\;_m\rho'|t_{m-g}\lambda|_m\rho'\;_g\rho') \\
  &=(-1)^mi^{m+g}p(m,\frac{1-\varepsilon}{2}m)(m\;m|1|m\;m)(m\;\rho'|{\alpha'}_{m}^{-1}(t_{m-g}\lambda)|\rho'\;_g\rho') \\
  &=(-1)^mi^{m+g}p(m,\frac{1-\varepsilon}{2}m)(-i)^m(-1)^{\frac{1-\varepsilon}{2}}(m\;\rho'|(-1)^m\epsilon_{3m}(m-g)t_{3m-g}\lambda|\rho'\;_g\rho') \\
   &=\varepsilon\epsilon_{3m}(m-g)i^gp(m,\frac{1-\varepsilon}{2}m)(m\;\rho'|t_{3m-g}\lambda|\rho'\;_g\rho').
\end{align*}
Since the right multiplication of $i^{m+g}(_g\rho'\;m|\lambda|m\;_g\rho')$ and the left multiplication of $(m\;\rho'|\lambda|\rho'\;m)$ on 
the 2-dimensional space $p(m,\frac{1-\varepsilon}{2})\cA_{m,_g\rho'}$ are commuting period two transformations that are not scalars, 
they coincide up to sign, and 
\begin{align*}
\lefteqn{(m\;\rho'|\lambda|\rho'\;m)p(m,\frac{1-\varepsilon}{2}m)(m\;\rho'|t_{m-g}|\rho'\;_g\rho')} \\
 &=i^gb^{\varepsilon}(g)p(m,\frac{1-\varepsilon}{2}m)(m\;\rho'|t_{3m-g}\lambda|\rho'\;_g\rho'),
\end{align*}
with $b^{\varepsilon}(g)\in \{1,-1\}$.

Since $e(\tpsi^{(\varepsilon_1,\varepsilon_2)})_{m,m} (m\;\rho'|t_{m-g}|\rho'\;_g\rho')$ is a multiple of a partial isometry with 
range projection $e(\tpsi^{(\varepsilon_1,\varepsilon_2)})_{m,m}$, there exists a positive number $c$ satisfying 
\begin{align*}
\lefteqn{e(\tpsi^{(\varepsilon_1,\varepsilon_2)})_{_g\rho',_g\rho'}}\\
 &=c(m\;\rho'|t_{m-g}|\rho'\;_g\rho'))^*p(m,\frac{1-\varepsilon_1}{2}m)\frac{1}{2}(1_m+\varepsilon_2(m\;\rho'|\lambda|\rho'\;m))
 (m\;\rho'|t_{g-m}|\rho'\;_g\rho') \\
 &=\frac{c\epsilon_{m-g}(g-m)}{2}(_g\rho'\;\rho'|t_{g-m}^*|\rho'\;m)p(m,\frac{1-\varepsilon_1}{2}m)
 (m\;\rho'|t_{g-m}+\varepsilon_2b^{\varepsilon_1}(g)i^g t_{3m-g}\lambda|\rho'\;_g\rho')\\
 &=\frac{c\epsilon_{m-g}(g-m)}{4m}\sum_{h=0}^{2m-1}i^h\varepsilon_1^h(_g\rho'\;\rho'|t_{m-g}^*|\rho'\;m)
 (m\;_h\rho'|\alpha'_h(t_{m-g}+\varepsilon_2b^{\varepsilon_1}(g)i^g t_{3m-g}\lambda)|_h\rho'\;_g\rho').
\end{align*}
On the other hand, we have\begin{align*}
\lefteqn{e(\tpsi^{(\varepsilon_1,\varepsilon_2)})_{_g\rho',_g\rho'}} \\
 &=\frac{1}{4md}1_{_g\rho'}
+\frac{1}{4md}(_g\rho'\;m|\cE_\psi^{(\varepsilon_1,\varepsilon_2)}(m)_{_g\rho',_g\rho'}|m\;_g\rho') \\
 &+\frac{1}{4m}\sum_{h=0}^{2m-1}(_g\rho'\;_h\rho'|\cE_\psi^{(\varepsilon_1,\varepsilon_2)}(_h\rho')_{_g\rho',_g\rho'}|_h\rho'\;_g\rho')),
\end{align*}
and so,
\begin{align*}
1&=cd\epsilon_{m-g}(g-m)s^*\rho'(t_{g-m}+\varepsilon_2b^{\varepsilon_1}(g)i^g t_{3m-g}\lambda)t_{m-g}^*\alpha'_g\circ\rho'(s) \\
 &=cds^*\alpha'_{m-g}\circ\rho'(t_{m-g})t_{m-g}^*\circ\rho'(s)=c, \\
\end{align*}
\begin{align*}
\lefteqn{\cE_\psi^{(\varepsilon_1,\varepsilon_2)}(m)_{_g\rho',_g\rho'}} \\
 &=cd\epsilon_{m-g}(g-m)i^m\varepsilon_1^m\lambda^{-1}s^*\rho'(\alpha'_m(t_{m-g}+\varepsilon_2b^{\varepsilon_1}(g)i^g t_{3m-g}\lambda))t_{g-m}^* \alpha'_g\circ\rho'(s\lambda) \\
 &=d\epsilon_{m-g}(g-m)\epsilon_m(3m-g)i^m\varepsilon_1^m\varepsilon_2b^{\varepsilon_1}(g)i^g (-1)^{m+g}\lambda^{-1}s^*\rho'(t_{m-g})\lambda 
 t_{g-m}^* \rho(s)\lambda \\
 &=i^{m+g}\varepsilon_1^m\varepsilon_2b^{\varepsilon_1}(g)\epsilon_m(3m-g)\lambda.
\end{align*}
Setting $a^\varepsilon(g)=(-1)^mb^\varepsilon(g)\varepsilon^{g}\epsilon_m(3m-g)$, we get 
$$\cE_\psi^{(\varepsilon_1,\varepsilon_2)}(m)_{_g\rho',_g\rho'}=(-1)^m(\varepsilon_1i)^{m+g}\varepsilon_2a^{\varepsilon_1}(g)\lambda.$$


Let $0\leq g<m$. 
Since $(_{2g}\rho'\;g|1|g\;\rho')^*(_{2g}\rho'\;g|1|g\;\rho')=1_{\rho'}$ and $(_{2g}\rho'\;g|1|g\;\rho')(_{2g}\rho'\;g|1|g\;\rho')^*=1_{_{2g}\rho'}$, we have 
$$e(\tpsi^{(\varepsilon_1,\varepsilon_2)})_{_{2g}\rho',_{2g}\rho'}=(_{2g}\rho'\;g|1|g\;\rho')e(\tpsi^{(\varepsilon_1,\varepsilon_2)})_{\rho',\rho'}(_{2g}\rho'\;g|1|g\;\rho')^*,$$
and so \begin{align*}
\lefteqn{(_{2g}\rho'\;m|\cE_\psi^{(\varepsilon_1,\varepsilon_2)}(m)_{_{2g}\rho',_{2g}\rho'}|m\;_{2g}\rho')} \\
 &=(_{2g}\rho'\;g|1|g\;\rho')
 (\rho'\;m|\cE_\psi^{(\varepsilon_1,\varepsilon_2)}(m)_{\rho',\rho'}|m\;\rho')
 (\rho'\;2m-g|1|2m-g\;_{2g}\rho') \\
 &=(_{2g}\rho'\;g+m|\alpha'_g(\cE_\psi^{(\varepsilon_1,\varepsilon_2)}(m)_{\rho',\rho'})|g\;\rho')(\rho'\;2m-g|1|2m-g\;_{2g}\rho')\\
 &=(_{2g}\rho'\;m|\lambda^{-1}\alpha'_g(\cE_\psi^{(\varepsilon_1,\varepsilon_2)}(m)_{\rho',\rho'})\alpha'_{2g}\circ\rho'(\lambda)|m\;_{2g}\rho') \\
 &=(_{2g}\rho'\;m|(-1)^g\cE_\psi^{(\varepsilon_1,\varepsilon_2)}(m)_{\rho',\rho'}|m\;_{2g}\rho').
\end{align*}
This shows $a^{\varepsilon}(2g)=a^{\varepsilon}(0)$. In the same way, we can show $a^{\varepsilon}(2g+1)=a^{\varepsilon}(1)$. 
\end{proof}

In view of the above result, we introduce the following index sets: $J_1=\{+,-\}$, 
$$J_2=\{(+,+),\;(+,-),\;(-,+),\;(-,-)\}.$$
Let $\tG=\Z_{2m}\times \Z_{2m}$, and let 
$$\tG_*=\{(j,k)\in \tG;\;j\in \{0,m\},\; 0<k<m\}\cup\{(g,k)\in \tG;\; 0<g<m\}.$$
Then $\tG=\tG_2\sqcup \tG_*\sqcup -\tG_*$. 

\begin{lemma} For any $0\leq k<2m$,  
$$\cA_{_k\rho'}=\bigoplus_{(\varepsilon_1,\varepsilon_2)\in J_2}\C e(\tpsi^{(\varepsilon_1,\varepsilon_2)})_{_k\rho',_k\rho'}\oplus M_2(\C)^{\oplus 4m^2}.$$
\end{lemma}

\begin{proof}
We show the statement for $k=0$ as our standing assumptions for $\rho$ and $\alpha_k\circ \rho$ are equivalent. 

Let $\mathbf{t}_{\rho'}=d(\rho'\;\rho'|ss^*|\rho'\;\rho')$, $U=i^m (\rho'\;m|\lambda|m\;\rho')$, 
$x_{g,h}=(\rho'\;_g\rho'|t_{h+g}t_{h-g}^*|_g\rho'\;\rho')$ for $0\leq g<2m$, $0\leq h<4m$. 
Then 
$$\{1_{\rho'},\;U,\;\mathbf{t}_{\rho'},\;U\mathbf{t}_{\rho'}\}
\cup\{x_{g,h},Ux_{g,h}\}_{0\leq g<2m,\;0\leq h<4m}$$
forms a basis of $\cA_{\rho'}$, and  $\dim \cA_{\rho'}=4+16m^2$. 
Note that $\mathbf{t}_{\rho'}$ is central, and $U$ is a unitary of period two. 
Let $\cA_{\rho'}^0$ be the linear span of 
$$\{1_{\rho'},\;\mathbf{t}_{\rho'}\}
\cup\{x_{g,h}\}_{0\leq g<2m,\;0\leq h<4m}.$$
Then $\cA_{\rho'}^0$ is a $*$-subalgebra of $\cA_{\rho'}$. 
Since  
\begin{align*}
\lefteqn{Ux_{g,h}U^{-1}} \\
 &=(-1)^m(\rho'\;m|\lambda|m\;\rho')(\rho'\;_g\rho'|t_{h+g}t_{h-g}^*|_g\rho'\;\rho')(\rho'\;m|\lambda|m\;\rho') \\
 &= \left\{
\begin{array}{ll}
(-1)^m (\rho'\;_{g+m}\rho'|\alpha_m(t_{h+g}t_{h-g}^*)\lambda|_{g+m}\rho'\;\rho')(\rho'\;m|\lambda|m\;\rho'), &\quad 0\leq g<m,  \\
(-1)^{m+g}(\rho'\;m|\lambda|m\;\rho')(\rho'\;_{g-m}\rho'|\lambda t_{h+g}t_{h-g}^*|_{g-m}\rho'\;\rho') , &\quad m\leq g<2m,
\end{array}
\right.
\\
 &=(-1)^g (\rho'\;_g\rho'|\lambda\alpha_m(t_{h+g}t_{h-g}^*)\lambda|_g\rho'\;\rho')\\
&=(-1)^g\epsilon_m(h+g)\epsilon_m(h-g)x_{g,h+2m},
\end{align*}
we see that $U$ normalizes $\cA_{\rho'}^0$. 

In the same way as in the proof of Lemma \ref{mulfree1}, we can prove that $\cA_{\rho'}^0$ is abelian by showing 
that the restriction of $S_0^2$ to $\cA_{\rho'}^0$ is the identity. 
It is easy to show $S_0^2(1_{\rho'})=1_{\rho'}$ and $S_0^2(\mathbf{t}_{\rho'})=\mathbf{t}_{\rho'}$. 
For $(\rho'\;_g\rho'|x_{g,h}|_g\rho'\;\rho')$, 
\begin{align*}
\lefteqn{S_0^2((\rho'\;_g\rho'|x_{g,h}|_g\rho'\;\rho'))=S_0((_g\rho'\;\rho'|ds^*\alpha'_g\circ\rho'(x_{g,h}\rho'(s))|\rho'\;_g\rho'))} \\
 &=(\rho'\;_g\rho'|d^2s^*\rho'(s^*\alpha'_g\circ\rho'(x_{g,h}\rho'(s))\alpha'_g\circ\rho'(s))|_g\rho'\;\rho') \\
 &=(\rho'\;_g\rho'|d^2s^*\rho'(s^*\alpha'_g\circ\rho'(x_{g,h}\rho'(s)s))|_g\rho'\;\rho'), 
\end{align*}
which is $x_{g,h}$ thanks to the proof of Lemma \ref{mulfree1}. 
Thus the claim is shown. 

Since $\cA_{\rho'}^0$ is abelian and normalized by $U$, and $\cA_{\rho'}=\cA_{\rho'}^0+U\cA_{\rho'}^0$, any simple component of 
$\cA_{\rho'}$ is either $\C$ or $M_2(\C)$. 
We already known that 
$$\bigoplus_{(\varepsilon_1,\varepsilon_2)\in J_2}\C e(\tpsi^{(\varepsilon_1,\varepsilon_2)})_{_k\rho',_k\rho'}\cong \C^4$$
is a direct summand of $\cA_{\rho'}$.  
On the other hand, let 
$$\cA_{\rho'}^{\varepsilon_1,\varepsilon_2}=\{x\in \cA_{\rho'};\; Ux=\varepsilon_1 x,\; xU=\varepsilon_2 x\}=\frac{1+\varepsilon_1 U}{2}\cA_{\rho'}\frac{1+\varepsilon_2 U}{2}.$$
Then it is easy to show 
$$\dim \cA_{\rho'}^{+,+}=\dim \cA_{\rho'}^{-,-}=2+4m^2,$$
$$\dim \cA_{\rho'}^{+,-}=\dim \cA_{\rho'}^{-,+}=4m^2,$$
which shows that $\cA_{\rho'}$ contains $M_2(\C)^{\oplus 4m^2}$ as a direct summand. 
Thus we get the statement. 
\end{proof}

\begin{lemma} Let 
$$\mu=2\bigoplus_{g=0}^{2m-1}\alpha'_g\rho'.$$
Then $\mu$ has exactly $2m^2$ half-braidings $\{\cE_\mu^{j}\}_{j=1}^{2m^2}$, which give the remaining simple objects in $\cZ(\cD)$.   
\end{lemma}

\begin{proof}
Let 
$$z=z(\widetilde{id})+z(\tpi)+\sum_{\varepsilon\in J_1}z(\widetilde{\varphi_\varepsilon})
+\sum_{(\varepsilon_1,\varepsilon_2)\in J_2}z(\tpsi^{(\varepsilon_1,\varepsilon_2)})
+\sum_{(g,k)\in \tG_*}z(\widetilde{\sigma_g}^k),$$
which is a central projection of $\mathrm{Tube}(\cD)$. 
Note that we have $|\tG_*|=2(m^2-1)$. 
Thus thanks to the previous lemma. we have 
$$(1-z)\cA_{_k\rho'}\cong M_2(\C)^{\oplus 2m^2}.$$Since $(_{2k}\rho'\;k|1|k\;\rho')^*(_{2k}\rho'\;k|1|k\;\rho')=1_{\rho'}$ and $(_{2k}\rho'\;k|1|k\;\rho')(_{2k}\rho'\;k|1|k\;\rho')^*=1_{_{2k\rho'}}$ 
for any $0<k<m$, we have 
$$\dim (1-z)\cA_{_{2k}\rho',\rho'}=\dim (1-z)\cA_{\rho'}=8m^2.$$ 
In the same way, 
$$\dim (1-z)\cA_{_{2k+1}\rho',_1\rho'}=\dim(1-z)\cA_{_1\rho'}=8m^2.$$
Direct counting shows $\dim \cA_{_1\rho',\rho'}=16m^2$. 
On the other hand, we can write down the basis of $z\cA_{_1\rho',\rho'}$ coming from $\tpi$, $\tpsi^{(\varepsilon_1,\varepsilon_2)}$, 
and $\widetilde{\sigma_g}^k$, showing $\dim z\cA_{_1\rho',\rho'}=8m^2$. 
Thus we get $\dim (1-z)\cA_{_1\rho',\rho'}=8m^2$, and 
$$(1-z)\mathrm{Tube}(\cD)\cong M_{4m^2}(\C)^{\oplus 2m^2}.$$
This shows the statement. 
\end{proof}

Let 
$$I=\{1,2,\ldots,2m^2\}.$$

\begin{lemma}  With the above notation, 
$$\overline{\tpsi^{(\varepsilon_1,\varepsilon_2)}}=\tpsi^{(\varepsilon_1,(-1)^m\varepsilon_2)},\quad 
\overline{\tmu^i}=\tmu^i.$$ 
\end{lemma}

\begin{proof} Direct computation shows $S_0^2(p(m,k))=p(m,2m-k)$ and $S_0^2((m\;\rho'|\lambda|\rho'\;m))=(-1)^m(m\;\rho'|\lambda|\rho'\;m)$, 
which implies  
$S_0^2(z(\tpsi^{(\varepsilon_1,\varepsilon_2)}))=z(\tpsi^{(\varepsilon_1,(-1)^m\varepsilon_2)})$, and the first statement. 

Recall we have 
$$\cA_{\rho'}=\bigoplus_{(\varepsilon_1,\varepsilon_2)\in J_2}\C e(\tpsi^{(\varepsilon_1,\varepsilon_2)})_{\rho',\rho'}\oplus \cA_{\rho'}^1,
\quad \cA_{\rho'}^1\cong M_2(\C)^{\oplus 4m^2}.$$
We already know $S_0^2(e(\tpsi^{(\varepsilon_1,\varepsilon_2)})_{\rho,\rho})=e(\tpsi^{(\varepsilon_1,(-1)^m\varepsilon_2)})_{\rho,\rho}$. 
In the proof of the previous lemma, we saw that the subalgebra $\cA_{\rho'}^0$ includes all the central projection of 
$\cA_{\rho'}^1$, and $S_0^2$ acts on $\cA_{\rho'}^0$ trivially. 
Thus we get $S_0^2(z(\tmu^i))=z(\tmu^i)$.
\end{proof}

\begin{theorem} Let the notation be as above. 
The following set exhausts the simple objects of the Drinfeld center $\cZ(\cD)$: 
$$\{0,\;\tpi\}\cup\{\widetilde{\varphi_\varepsilon}\}_{\varepsilon\in J_1}\cup
\{\tpsi^{(\varepsilon_1,\varepsilon_2)}\}_{(\varepsilon_1,\varepsilon_2)\in J_2}\cup 
\{\widetilde{\sigma_g}^k\}_{(g,k)\in \tG_*}\cup \{\tmu^i\}_{i\in I}.$$
We have $\overline{\tpsi^{(\varepsilon_1,\varepsilon_2)}}=\tpsi^{(\varepsilon_1,(-1)^m\varepsilon_2)}$, and the others are self-conjugate.  
Except for $\tmu^i$-$\tmu^{i'}$ entries, the $S$-matrix and $T$-matrix are given as 
$$S_{0,0}=S_{\tpi,\tpi}=\frac{a'-b'}{2},\quad S_{0,\tpi}=\frac{a'+b'}{2},$$
$$S_{0,\widetilde{\varphi_\pm}}=S_{0,\tpsi^{(\varepsilon_1,\varepsilon_2)}}
=S_{\tpi,\widetilde{\varphi_\pm}}=S_{\tpi,\tpsi^{(\varepsilon_1,\varepsilon_2)}}=\frac{a'}{2},$$
$$S_{0,\widetilde{\sigma_g}^k}=S_{\tpi,\widetilde{\sigma_g}^\tau}=a',\quad 
S_{0,\tmu^i}=b',\quad 
S_{\tpi,\tmu^i}=-b',$$
$$S_{\widetilde{\varphi_{\varepsilon}},\widetilde{\varphi_{\varepsilon'}}}=
\frac{a'+\varepsilon\varepsilon'}{2},
\quad S_{\widetilde{\varphi_\pm},\tpsi^{(\varepsilon_1,\varepsilon_2)}}
=\frac{(-1)^ma'}{2},$$
$$S_{\widetilde{\varphi_\pm},\widetilde{\sigma_g}^k}=(-1)^g a',\quad S_{\widetilde{\varphi_\pm},\tmu^k}=0,$$
$$S_{\tpsi^{(\varepsilon_1,\varepsilon_2)},\tpsi^{(\varepsilon_1',\varepsilon_2')}}
=\frac{(-\varepsilon_1\varepsilon_1')^ma'+\varepsilon_2\varepsilon_2'(\varepsilon_1i)^m\delta_{\varepsilon_1,\varepsilon_1'}}{2},$$
$$S_{\tpsi^{(\varepsilon_1,\varepsilon_2)},\widetilde{\sigma_g}^k}
=(-\varepsilon_1)^g(-1)^ka',\quad 
S_{\tpsi^{(\varepsilon_1,\varepsilon_2)},\tmu^i}=0,$$
$$S_{\widetilde{\sigma_g}^k,\widetilde{\sigma_{g'}}^{k'}}=2a'\cos \frac{(gg'+gk'+g'k)\pi}{m},\quad 
S_{\widetilde{\sigma_g}^k,\tmu^i}=0,$$
$$T_{0,0}=T_{\tpi,\tpi}=T_{\widetilde{\varphi_{\pm}},\widetilde{\varphi_{\pm}}}=1,\quad 
T_{\tpsi^{(\varepsilon_1,\varepsilon_2)},\tpsi^{(\varepsilon_1,\varepsilon_2)}}=(\varepsilon_1i)^m,\quad 
T_{\widetilde{\sigma_g}^k,\widetilde{\sigma_g}^k}=\zeta_{4m}^{g^2+2kg}.$$
\end{theorem}

\begin{proof} The only statements that do not directly follow from the previous arguments are about  
$S_{\tpsi^{(\varepsilon_1,\varepsilon_2)},\tpsi^{(\varepsilon_1',\varepsilon_2')}}$ 
and $S_{\tpsi^{(\varepsilon_1,\varepsilon_2)},\tmu^i}$. 
Direct computation shows 
$$S_{\tpsi^{(\varepsilon_1,\varepsilon_2)},\tpsi^{(\varepsilon_1',\varepsilon_2')}}
=\frac{(-\varepsilon_1\varepsilon_1')^m}{4m}+\varepsilon_2\varepsilon_2'(\varepsilon_1'i)^m
\frac{a^{\varepsilon_1'}(0)+a^{\varepsilon_1'}(1)\varepsilon_1\varepsilon_1'}{4}.$$
Since $S$ is a symmetric matrix, we have 
$S_{\tpsi^{(\varepsilon_1,\varepsilon_2)},\tpsi^{(\varepsilon_1',\varepsilon_2')}}
=S_{\tpsi^{(\varepsilon_1',\varepsilon_2')},\tpsi^{(\varepsilon_1,\varepsilon_2)}}$, and 
$$(\varepsilon_1')^m(a^{\varepsilon_1'}(0)+a^{\varepsilon_1'}(1)\varepsilon_1\varepsilon_1')=
(\varepsilon_1)^m(a^{\varepsilon_1}(0)+a^{\varepsilon_1}(1)\varepsilon_1\varepsilon_1').$$
This is equivalent to 
$$a^+(0)-a^+(1)=(-1)^m(a^-(0)-a^-(1)).$$
Thus either $a^+(0)=a^+(1)$, $a^{-}(0)=a^-(1)$, or 
$$-a^+(1)=-(-1)^m a^-(1)=(-1)^ma^-(0)=a^+(0).$$ 

Assume $a^+(0)=a^+(1)$, $a^{-}(0)=a^-(1)$ first. 
Then we get
$$S_{\tpsi^{(\varepsilon_1,\varepsilon_2)},\tpsi^{(\varepsilon_1',\varepsilon_2')}}
=\frac{(-\varepsilon_1\varepsilon_1')^m}{4m}+
\frac{\varepsilon_2\varepsilon_2'(\varepsilon_1i)^ma^{\varepsilon_1}(0)\delta_{\varepsilon_1,\varepsilon_1'}}{2}.$$
Since $S$ is a unitary, 
\begin{align*}
\lefteqn{1=\sum_{a}|S_{\tpsi^{(\varepsilon_1,\varepsilon_2)},a}|^2} \\
 &=\frac{1}{16m^2}+\frac{1}{16m^2}+\frac{1}{16m^2}+\frac{1}{16m^2}+\sum_{s,t\in \{1,-1\}}
 |\frac{(-\varepsilon_1s)^m}{4m}+
\frac{\varepsilon_2t(\varepsilon_1i)^ma^{\varepsilon_1}(0)\delta_{\varepsilon_1,s}}{2}|^2 \\
 &+\frac{|\tG_*|}{(2m)^2}+\sum_{i\in I}|S_{\tpsi^{(\varepsilon_1,\varepsilon_2)},\tilde{\mu}^j}|^2 \\
 &=\frac{1}{2}-\frac{1}{4m^2}+2\sum_{s\in \{1,-1\}}(\frac{1}{16m^2}+\frac{\delta_{\varepsilon_1,s}}{4})
 +\sum_{i\in I}|S_{\tpsi^{(\varepsilon_1,\varepsilon_2)},\tilde{\mu}^j}|^2\\
 &=1+\sum_{i\in I}|S_{\tpsi^{(\varepsilon_1,\varepsilon_2)},\tilde{\mu}^j}|^2,
\end{align*}
showing $S_{\tpsi^{(\varepsilon_1,\varepsilon_2)},\tilde{\mu}^j}=0$. 

Recall the modular group relation $(S)^2=(ST)^3=C$, $TC=CT$, where $C_{a,b}=\delta_{a,\overline{b}}$, and $\overline{b}$ 
is determined by $\overline{S_{a,b}}=S_{a,\overline{b}}$. 
Recall $\overline{\tpsi^{(\varepsilon_1,\varepsilon_2)}}=\tpsi^{(\varepsilon_1,(-1)^m\varepsilon_2)}$.  
We compare the $\tpsi^{(\varepsilon_1,\varepsilon_2)}$-$\tpsi^{(\varepsilon_1',\varepsilon_2')}$ entries  
of the both sides of 
$$STS=C{T}^{-1}{S}^{-1}{T}^{-1}=C\overline{TST}=\overline{T}S\overline{T}.$$ 
\begin{align*}
\lefteqn{(STS)_{\tpsi^{(\varepsilon_1,\varepsilon_2)},\tpsi^{(\varepsilon_1',\varepsilon_2')}}} \\
 &=\frac{1}{16m^2}+\frac{1}{16m^2}+\frac{1}{16m^2}+\frac{1}{16m^2}\\
 &+\sum_{s,t\in \{1,-1\}}(\frac{(-\varepsilon_1s)^m}{4m}+
\frac{\varepsilon_2t(\varepsilon_1i)^ma^{\varepsilon_1}(0)\delta_{\varepsilon_1,s}}{2})  
(\frac{(-\varepsilon_1's)^m}{4m}+
\frac{\varepsilon_2't(\varepsilon_1'i)^ma^{\varepsilon_1'}(0)\delta_{\varepsilon_1',s}}{2})(si)^m\\
 &+\sum_{(g,k)\in \tG_*} \frac{(-\varepsilon_1)^g(-1)^k}{2m}\frac{(-\varepsilon_1')^g(-1)^k}{2m}\zeta_{4m}^{g^2+2gk}\\
 &=\frac{1}{4m^2}+2\sum_{s\in \{1,-1\}}(\frac{(\varepsilon_1\varepsilon_1')^m}{16m^2}
 +\frac{\varepsilon_2\varepsilon_2'(-\varepsilon_1\varepsilon_1')^m a^{\varepsilon_1}(0) a^{\varepsilon_1'}(0)
\delta_{\varepsilon_1,s}\delta_{\varepsilon_1',s}}{4}) (si)^m \\
 &+\sum_{k=1}^{m-1} \frac{1}{4m^2}+\sum_{k=1}^{m-1} \frac{(\varepsilon_1\varepsilon_1')^m}{4m^2}i^{m+2k}
 +\sum_{g=1}^{m-1}\sum_{k=0}^{2m-1} \frac{(\varepsilon_1\varepsilon_1')^g}{4m^2}\zeta_{4m}^{g^2+2gk}\\
 &=\frac{1}{4m^2}+(\varepsilon_1\varepsilon_2i)^m\frac{1+(-1)^m}{8m^2}
 +\frac{\varepsilon_2\varepsilon_2'(-\varepsilon_1i)^m\delta_{\varepsilon_1,\varepsilon_1'}}{2}
 +\frac{m-1}{4m^2}-(\varepsilon_1\varepsilon_2i)^m\frac{1+(-1)^m}{8m^2}\\
 &=\frac{1}{4m}+\frac{\varepsilon_2\varepsilon_2'(-\varepsilon_1i)^m\delta_{\varepsilon_1,\varepsilon_1'}}{2}.
\end{align*}
On the other hand, 
$$(\overline{T}S\overline{T})_{\tpsi^{(\varepsilon_1,\varepsilon_2)},\tpsi^{(\varepsilon_1,\varepsilon_2)}}=
\frac{1}{4m}+
\frac{\varepsilon_2\varepsilon_2'(-\varepsilon_1i)^ma^{\varepsilon_1}(0)\delta_{\varepsilon_1,\varepsilon_1'}}{2}.
$$
Thus we get $a^\varepsilon(0)=1$.

Assume now that that the second case 
$$-a^+(1)=-(-1)^m a^-(1)=(-1)^ma^-(0)=a^+(0),$$
occurs. 
Then 
$$S_{\tpsi^{(\varepsilon_1,\varepsilon_2)},\tpsi^{(\varepsilon_1',\varepsilon_2')}}
=\frac{(-\varepsilon_1\varepsilon_1')^m}{4m}+
\frac{\varepsilon_2\varepsilon_2'a^{+}(0)i^m\delta_{\varepsilon_1,-\varepsilon_1'}}{2}.$$
In the same way as above, we get $S_{\tpsi^{(\varepsilon_1,\varepsilon_2)},\tilde{\mu}^j}=0$, and 
 \begin{align*}
\lefteqn{(STS)_{\tpsi^{(\varepsilon_1,\varepsilon_2)},\tpsi^{(\varepsilon_1',\varepsilon_2')}}} \\
 &=\frac{1}{16m^2}+\frac{1}{16m^2}+\frac{1}{16m^2}+\frac{1}{16m^2}\\
 &+\sum_{s,t\in \{1,-1\}}(\frac{(-\varepsilon_1s)^m}{4m}+
\frac{\varepsilon_2ti^ma^+(0)\delta_{\varepsilon_1,-s}}{2})  
(\frac{(-\varepsilon_1's)^m}{4m}+
\frac{\varepsilon_2'ti^ma^{+}(0)\delta_{\varepsilon_1',-s}}{2})(si)^m\\
 &+\sum_{(g,k)\in \tG_*} \frac{(-\varepsilon_1)^g(-1)^k}{2m}\frac{(-\varepsilon_1')^g(-1)^k}{2m}\zeta_{4m}^{g^2+2gk}\\
 &=\frac{1}{4m^2}+2\sum_{s\in \{1,-1\}}(\frac{(\varepsilon_1\varepsilon_1')^m}{16m^2}
 +\frac{\varepsilon_2\varepsilon_2'(-1)^m\delta_{\varepsilon_1,-s}\delta_{\varepsilon_1',-s}}{4}) (si)^m \\
 &+\sum_{k=1}^{m-1} \frac{1}{4m^2}+\sum_{k=1}^{m-1} \frac{(\varepsilon_1\varepsilon_1')^m}{4m^2}i^{m+2k}
 +\sum_{g=1}^{m-1}\sum_{k=0}^{2m-1} \frac{(\varepsilon_1\varepsilon_1')^g}{4m^2}\zeta_{4m}^{g^2+2gk}\\
 &=\frac{1}{4m^2}+(\varepsilon_1\varepsilon_2i)^m\frac{1+(-1)^m}{8m^2}
 +\frac{\varepsilon_2\varepsilon_2'(\varepsilon_1i)^m\delta_{\varepsilon_1,\varepsilon_1'}}{2}
 +\frac{m-1}{4m^2}-(\varepsilon_1\varepsilon_2i)^m\frac{1+(-1)^m}{8m^2}\\
 &=\frac{1}{4m}+\frac{\varepsilon_2\varepsilon_2'(\varepsilon_1i)^m\delta_{\varepsilon_1,\varepsilon_1'}}{2}.
\end{align*}
On the other hand, 
$$(\overline{T}S\overline{T})_{\tpsi^{(\varepsilon_1,\varepsilon_2)},\tpsi^{(\varepsilon_1,\varepsilon_2)}}=
\frac{1}{4m}+
\frac{\varepsilon_2\varepsilon_2'i^ma^+(0)\delta_{\varepsilon_1,-\varepsilon_1'}}{2},$$
which is a contradiction. 
\end{proof}

\begin{figure}
$
\begin{array}{c !{\vline width 1pt} cccccc!{\vline width 1pt}}
\textbf{S} & & & & &  \\[3pt]
\Xhline{1pt}
0 & \frac{a'-b'}{2} & & & & \\[3pt]
 \widetilde{\pi}  & \frac{a'+b'}{2} &  \frac{a'-b'}{2}& & &  \\[3pt]
 \widetilde{\phi_{\varepsilon}' }  & \frac{a'}{2} &  \frac{a'}{2}& \frac{a'+\varepsilon \varepsilon'}{2} & &  \\[3pt]
\widetilde{\psi}^{(\varepsilon_1',\varepsilon_2')} & \frac{a'}{2} & \frac{a'}{2}   & \frac{(-1)^m a'}{2} &\frac{(-\varepsilon_1 \varepsilon'_1)ma'+\varepsilon_2 \varepsilon'_2 (\varepsilon_1i)^m \delta_{\varepsilon_1,\varepsilon_1'}  }{2} &  \\[3pt]
 \widetilde{\sigma_g'}^{k'}  &a' & a' & (-1)^{g'} & (-\varepsilon_1)^{g'}(-1)^{k'}a'  & 2a' \cos \left( \frac{(gg'+gk'+g'k)\pi}{m} \right) &  \\[3pt]
   \widetilde{\mu}^{i'} & b' & -b' & 0 & 0 & 0 & ?  \\[3pt]
 \Xhline{1pt}
  & 0 &  \widetilde{\pi}  &  \widetilde{\phi_{\varepsilon} } &  \widetilde{\psi}^{(\varepsilon_1,\varepsilon_2)} &\widetilde{\sigma_g}^k & \widetilde{\mu}^i \\[3pt]
 \Xhline{1pt}
 \textbf{T} & 1 & 1 &1 &(\varepsilon_1i)^m  &  \zeta_{4m}^{g^2+2kg} & ?  \Tstrut \\[3pt]
\end{array}
$
\caption{Modular data for the $\Z_2 $-de-equivariantization of a generalized Haagerup category for $G=\Z_{4m} $ with $\epsilon_{2m}(g)=(-1)^g $, with entries labeled by ``?'' undetermined. }
\label{stpart6}
\end{figure}

To compute the missing corner in examples, we use Mathematica and the formulas for tube algebras for a de-equivariantization of a generalized Haagerup category, included in the online appendix.

\begin{example}
For $G=\Z_4 $, let $\cC $ be the generalized Haagerup category satisfying Eq.(\ref{Q2}). Then $\epsilon_{2}(g)=(-1)^g $, and we have the $\Z_2$-de-equivariantization  $\cD$, which is the principal even part of the $2D2$ subfactor \cite{MR3394622,MR3827808}. 

We can compute the $I \times I $-corner of the modular data of the Drinfeld center by diagonalizing the action of $\textbf{t} $ on the tube algebra, using a similar method to that outlined in Section \ref{corner1}. We have $|I_0|=2 $. The two $T$-eigenvalues for $\widetilde{\mu}^i $ are $\zeta_5^{\pm 2} $, and the corresponding block of the $S$-matrix is 
$$ 
\frac{1}{10} \left(
\begin{array}{cc}
-5+\sqrt{5} & 5+\sqrt{5} \\
5+\sqrt{5}  & -5+\sqrt{5}  \\
\end{array}
\right).
$$ 

This $S$-matrix looks similar to that of the commutant of $G_2$ in the center of the generalized Haagerup category for $\Z_2 \times \Z_2 $, which also has rank $10$, but with differences in several blocks.

\end{example}

\begin{example}
The generalized Haagerup category $ \cC$ for $\Z_8 $ with $(\epsilon,A) $ given in the Mathematica notebook \texttt{solutions.nb} satisfies $\epsilon_{4}(g)=(-1)^g $.
Again we can compute the $I \times I $-corner of the modular data of the Drinfeld center of the $\Z_2 $-de-equivariantization $\cD $ by diagonalizing the action of $\textbf{t} $ on the tube algebra. 
We have $|I|=8 $, and we find that the missing corner is the same as that of the Drinfeld center of the Asaeda-Haagerup categories:  the eigenvalues of the $\widetilde{\mu}^i $ are $\zeta_{17}^{3i^2} $, for $ 1 \leq i \leq 8 $, with
$$S_{\widetilde{\mu}^i ,\widetilde{\mu}^{i'} }= -\frac{2}{\sqrt{17}} \cos  \frac{12 \pi i i'}{17}  .$$

\end{example}
\section{$\Z_3$-equivariantization} \label{4442}

In this section we compute the modular data for the Drinfeld center of the even part of the $4442$ subfactor. The $4442$ subfactor was first constructed in \cite{MR3314808}. It is self dual, and its even part is a $\Z_3 $-equivariantization of the generalized Haagerup category for $G=\Z_2 \times \Z_2 $ \cite[Corollary 9.5]{MR3827808}. 

The structure constants $(\epsilon,A) $ of the generalized Haagerup category $ \cC$ are given in \cite[Theorem 9.4]{MR3827808} in terms of a sign $s$ and a fourth root of unity $\dfrac{z}{\sqrt{d}} $. We fix $s=1$ and $z=\sqrt{d} $.
We denote the elements of $\Z_2 \times \Z_2 $ by $\{ 0,a,b,c \} $, in the same order as the matrices $(\epsilon,A) $.

Let  $\theta $ be the automorphism of $\Z_2 \times \Z_2 $
satisfying $$\theta(a)=b, \ \theta(b)=c, \theta(c)=a .$$
Then $\epsilon $ and  $A$ are invariant under $\theta $, so we can define an automorphism $\gamma(s)=s $ and $\gamma(t_g)=t_{\theta(g)} , g \in G$. Then the even part of the $4442$ subfactor is equivalent to the equivariantization $ \cC^{\Z_3}$ with respect to the action generated by $\gamma $. 

To describe the Drinfeld center of $ \cC^{\Z_3}$, it is easier to work instead with the Morita equivalent category  $ \cC \rtimes \Z_3$, which is generated by $\cC $ and $\gamma $. The category $\cC \rtimes \Z_3$ leaves the Cuntz algebra generated by $ s$ and $t_g, g \in G $ invariant, so we can do all of our calculations in terms of this Cuntz algebra.

Note that $H=\text{Inv}(\cC \rtimes \Z_3) = (\Z_2 \times \Z_2) \rtimes_{\theta} \Z_3$ is isomorphic to the alternating group on four letters. We denote a typical element of $H$ by $$(i,g)=\gamma^i \circ \alpha_g,  \ i \in \{0,1,2\}, \ g \in \{0,a,b,c \} .$$

The tube algebra of $\mathcal{C} \rtimes_{\gamma} \Z_3 $ inherits the $\Z_3$-grading of the category, and so we look for the minimal central projections in the graded components of the tube algebra separately. For detailed calculations within the tube algebra, such as diagonalization of $\textbf{t} $, we use Mathematica and the tube algebra formulas for an equivariantization of a generalized Haagerup category, which are included in the online appendix.

 We first look at the trivially-graded component of $\Tube (\mathcal{C} \rtimes_{\gamma} \Z_3)$, which contains $\Tube \cC $.
 The group part $\cA _G$ of $\Tube \cC $ is Abelian, and following the notation of Section \ref{restrict}, its minimal projections are $z(k) $ and $E(k,\epsilon_k) $, for $k \in G $; and $E(k,\epsilon_l)_{\pm} $, for $k \neq l \in G $.
 
In the larger algebra $\Tube  (\mathcal{C} \rtimes_{\gamma} \Z_3) $, we can break up the trivially-graded group part as
$$\cA_{(0,G)}=\cA_{(0,0)} \oplus \cA_{(0,G\backslash 0)}  ,$$ with $\text{dim}(\mathcal{A}_{(0,0)})=24 $ and $\text{dim}(\cA_{(0,G\backslash 0)})=72$.

For $g,h \neq 0 \in G $, we have that $ 1_g $ is equivalent to $1_h $, since $g$ and $h$ are in the same conjugacy class in $H$. Therefore, there are $8$ minimal central projections in $\mathcal{A}_{(0,G \backslash \{0\})} $ which all have rank three. They are:
 $$z(a)+z(b)+z(c), \quad E(a,\epsilon_a) +E(b,\epsilon_b) +E(c,\epsilon_c) $$
 and
 $$ E(a,\epsilon_{0} )_{\varepsilon}+E(b,\epsilon_{0} )_{\varepsilon}+E(c,\epsilon_{0} )_{\varepsilon} ,  \quad E(a,\epsilon_{b} )_{\varepsilon}+E(b,\epsilon_{c} )_{\varepsilon}+E(c,\epsilon_{a} )_{\varepsilon}, $$ $$ E(a,\epsilon_{c} )_{\varepsilon}+E(b,\epsilon_{a} )_{\varepsilon}+E(c,\epsilon_{b} )_{\varepsilon} , \quad  \varepsilon \in \{ \pm \}.$$
 
 To find the center of $\mathcal{A}_{(0,0)} $, we also consider the projections $$p^{\omega} =\frac{1}{3}\sum_{i=0}^{2} \omega^i ( (0,0) \;  (i,0) | 1  | (i,0) \; (0,0)  ),$$
 for $\omega  $ a cube root of unity.

Then the minimal central projections of $\mathcal{A}_{(0,0)} $ are $$p^{\omega} z(0) \text{ and }  p^{\omega}E(0,0), \quad \omega \in \{1,\zeta_3,\zeta_3^{-1} \},$$
which each have rank one, and $$E(0,\epsilon_a)_{\varepsilon}+E(0,\epsilon_b)_{\varepsilon}+E(0,\epsilon_c)_{\varepsilon}, \ \varepsilon \in \{ \pm \},  $$
 which each have rank three.
 
 Therefore there are $16$ minimal central projections in $\mathcal{A}_{(0,G)} $. To find the corresponding minimal central projections in  $\Tube (\mathcal{C} \rtimes_{\gamma} \Z_3)$, we follow a similar procedure as in Section \ref{corner}. Namely, for each minimal central projection $p$ in $\mathcal{A}_{(0,G)} $, we choose a minimal subprojection $p'$ and a basis $\{ j_s \}_{s \in S} $ of mutually orthogonal partial isometries for $p' \mathcal{A}_{(0,G),{}_{(0,G)} \rho}$ (this is not difficult since for a fixed $p'$ and $h\in G$, the space $p' \mathcal{A}_{(0,G),{}_{(0,h)} \rho}$ turns out to be at most $3$-dimensional). Then the corresponding minimal central projection in the tube algebra is $$p+ \sum_{s \in S} \limits j_s^*j_s.$$ 

After computing the $16$ minimal central projections of  $\Tube (\mathcal{C} \rtimes_{\gamma} \Z_3)$ which have nontrivial component in $\mathcal{A}_{(0,G)} $, we can list the corresponding objects in the Drinfeld center $\cZ (\mathcal{C} \rtimes_{\gamma} \Z_3)$.

 \begin{lemma} \label{l1}
 \begin{enumerate}
 \item  The (identity) object $(0,0) $
	and the object $(0,0) \oplus \bigoplus_{ g \in G} \limits {}_{(0,g)} \rho $ 
	each have three irreducible half-braidings.
 
 \item The objects $ \bigoplus_{0 \neq g \in G} \limits  (0,g)$ and
	$\bigoplus_{0 \neq g \in G} \limits ( (0,g) \oplus 3{}_{(0,g)} \rho   )\oplus 3 {}_{(0,0)} \rho $
	each have a unique irreducible half-braiding.

\item The objects  $\bigoplus_{0 \neq g \in G} \limits ( (0,g) \oplus {}_{(0,g)} \rho   ) \oplus 3 {}_{(0,0)} \rho $
	and $\bigoplus_{0 \neq g \in G} \limits ( (0,g) \oplus 2{}_{(0,g)} \rho   ) $ each have three irreducible half-braidings.
	
	 \item The objects $3(0,0)\oplus 2 \bigoplus_{0 \neq g \in G} \limits {}_{(0,g)} \rho $	
	and $3 ((0,0) \oplus {}_{(0,0)} \rho  ) \oplus \bigoplus_{0 \neq g \in G} \limits {}_{(0,g)} \rho $ each have a unique irreducible half-braiding.

	 \end{enumerate}
 \end{lemma}

  We can now find the remaining minimal central projections in the $0$-graded component of $\Tube (\mathcal{C} \rtimes_{\gamma} \Z_3)$ as follows. Let $Z_{(0,G)} $ be the sum of the $16$ minimal central projections with non-trivial component in $ \cA_{(0,G)}$.
  The dimension of $\mathcal{A}_{{}_{(0,0)} \rho} $ is $72$, and for $h \neq 0$ the dimensions of 
  $\mathcal{A}_{{}_{(0,h)} \rho} $ and  $\mathcal{A}_{ {}_{(0,0)} \rho,  {}_{(0,h)} \rho} $ are
  $56$ and $48$, respectively. On the other hand, the dimensions of  $Z_{(0,G)}\mathcal{A}_{{}_{(0,0)} \rho} $, $Z_{(0,G)}\mathcal{A}_{{}_{(0,h)} \rho} $, and  $Z_{(0,G)}\mathcal{A}_{ {}_{(0,0)} \rho,  {}_{(0,h)} \rho} $ are $48$, $32 $, and $24$, respectively. This imples that  
  $$\text{dim}( (1-Z_{(0,G)})\mathcal{A}_{{}_{(0,0)} \rho}) =\text{dim}((1-Z_{(0,G)})\mathcal{A}_{{}_{(0,h)} \rho})=\text{dim}( (1-Z_{(0,G)})\mathcal{A}_{ {}_{(0,0)} \rho,  {}_{(0,h)} \rho} )=24.$$
  Therefore all of the subalgebras $(1-Z_{(0,G)})\mathcal{A}_{{}_{(0,h)} \rho}) $ are $24$-dimensional, and the corresponding projections $ (1-Z_{(0,G)})1_{{}_{(0,h)}\rho}$ are equivalent in the tube algebra.
  
  To find the minimal central projections in $(1-Z_{(0,G)})\mathcal{A}_{{}_{(0,h)} \rho} $, we diagonalize $\textbf{t}_{{}_{(0,h)} \rho}$.  For each $h \in G$ the minimal polynomial of $\textbf{t}_{{}_{(0,h)} \rho}$ is 
  
  $$q(x)=(x^2-1)(x^2-\zeta_5^2 )(x^2-\zeta_5^{-2 }) .$$

We let $$q_{\lambda}(x)=\frac{q(x)}{x -\lambda}  \text{ and } p^{\lambda}_h=\frac{q_{\lambda}( \textbf{t}_{{}_{(0,h)} \rho} ) }{q_{\lambda}(\lambda)} , \quad \lambda \in \{ \pm\zeta_5^{\pm 1} \}.$$

We find that each $p^{\lambda}_h $ is a rank three projection, which is a minimal central projection in $ \mathcal{A}_{{}_{(0,h)} \rho}$ for $\lambda=-\zeta_5^{\pm 1}$, but splits as a sum of three minimal central projections for $ \lambda=\zeta_5^{\pm 1}$. Then we can match up the minimal subprojections of the $p^{\lambda}_h $ for $\lambda= \zeta_5^{\pm 1}$ for different $h$ to find the corresponding minimal central projections in $\Tube (\mathcal{C} \rtimes_{\gamma} \Z_3)$.
  
%
%
%

\begin{lemma} \label{l2}
\begin{enumerate}
 \item The object $\bigoplus_{h \in G}  \limits {}_{(0,h)} \rho $ has $6$ irreducible half-braidings.
 \item  The object $3 \bigoplus_{h \in G}  \limits {}_{(0,h)} \rho$ has $2$ irreducible half-braidings.
\end{enumerate}
\end{lemma}

Next we consider the non-trivially-graded components of the tube algebra. Let $\tau \in \{ 1,2\} $.  Then we have $\mathcal{A}_{(\tau,g)} $ is isomorphic to 
  $\mathcal{A}_{(\tau,h)} $ and similarly $\mathcal{A}_{{}_{(\tau,g)} \rho} $ is isomorphic to $\mathcal{A}_{{}_{(\tau,h)} \rho} $ for all $g$ and $h$ in $G$, so it suffices to consider $\mathcal{A}_{(\tau,0)} $ and
  $\mathcal{A}_{{}_{(\tau,0)} \rho} $, which have dimensions $6$ and $54$ respectively. The space  $\mathcal{A}_{(\tau,0), {}_{(\tau,0)} \rho} $ has dimension $12$.

  The algebra $\mathcal{A}_{(\tau,0)} $ is Abelian, and the $\omega$-eigenspace of $\textbf{t}_{(\tau,0)} $ is $2$-dimensional for each cube root of unity $\omega $. For each $\omega $, let $$r^{\omega}_\tau=\frac{1}{3} 
  \sum_{i=0}^2 \omega^i ( (\tau,0 ) \;  (i,0) | 1 |(i,0) \; (\tau,0) )$$
  and $$s^{\omega}_\tau=\frac{1}{3} 
  \sum_{i=0}^2 \omega^i ( (\tau,0 ) \;  {}_{(i,0)} \rho | 1 |{}_{(i,0)} \rho \; (\tau,0) ) .$$
  Then $(s^{\omega}_\tau)^2 =r^{\omega}_\tau+s^{\omega}_\tau$ and
  the six minimal projections of $\mathcal{A}_{(\tau,0)} $ are given by
  $$p^{\omega,0}_{\tau}=\frac{5+\sqrt{5}}{10}r^{\omega}_\tau-\frac{1}{\sqrt{5}}s^{\omega}_\tau   $$
  and
  $$p^{\omega,1}_\tau= \frac{5-\sqrt{5}}{10}r^{\omega}_\tau+\frac{1}{\sqrt{5}}s^{\omega}_\tau,$$
  for the three choices of cube root of unity $\omega $.
  
  The $\textbf{t}_{(\tau,0)} $-eigenvalue for each $p_\tau^{\omega,i} $ is $\omega $. We can find the corresponding minimal central projections in $\Tube (\mathcal{C} \rtimes_{\gamma} \Z_3)$ by looking at the action of $ \mathcal{A}_{(\tau,0)} $ on $ \mathcal{A}_{(\tau,0),{}_{(\tau,0)} \rho} $, in a similar way to the $0$-graded case above. 
  
%
%
%

  \begin{lemma} \label{l3}
  For each $\tau $ in $\{ 1,2\} $,
    the objects $\bigoplus_{g \in G} \limits ( \tau,g) \oplus  {}_{(\tau,g)}  \rho ) $ and $\bigoplus_{g \in G} \limits ( \tau,g) \oplus 3 {}_{(\tau,g)}  \rho ) $ each have three irreducible half-braidings
    
%
  
  \end{lemma}
  
 Finally, we can determine the remaining minimal central projections in  $ \mathcal{A}_{{}_{(\tau,G)} \rho}$ by diagonalizing $\textbf{t}_{{}_{(\tau,0)}\rho} $, in a similar way to the $0$-graded case. We find that 
 $\textbf{t}_{{}_{(\tau,0)}\rho} $ has six additional eigenvalues, which are $\{ \omega \zeta_5^{\pm1} \} $, for $\omega $ a cube root of unity. The corresponding eigenprojections all have rank two and are central.  
%
  \begin{lemma} \label{l4}
  For each $\tau $ in $\{ 1,2\} $,
  the object $2\bigoplus_{g \in G} \limits   {}_{(\tau,g)}  \rho  $ has six irreducible half-braidings.
  \end{lemma}
  
   Now that we have found the $48$ minimal central projections of  $\Tube (\mathcal{C} \rtimes_{\gamma} \Z_3)$ and their $T$-eigenvalues, we can compute the $S$-matrix using Eq.(\ref{Sformula0}). 
   To display the modular data, we group the simple objects in $\cZ (\mathcal{C} \rtimes_{\gamma} \Z_3)$ into eight blocks of sizes $6$, $2$, $6$, $2$, $6$, $2$, $12$, and $12$, respectively, corresponding to the enumerations in Lemmas \ref{l1},  \ref{l2},  \ref{l3},  and \ref{l4}.  
Within each block we use the following indexing convention. We factor the block size into a product of a power of two and a power of three. Then we index each factor of size three by a cube root of unity $\omega $ and each factor of size two by a sign $ \varepsilon$ (or $\varepsilon_1$,$\varepsilon_2 $). 

\begin{theorem}
With notation as above and appropriate ordering within each block, the modular data for the Drinfeld center of the even part of the $4442$ subfactor is given by the table in Figure \ref{st4442take4}. 
\end{theorem}

\begin{figure}
\resizebox{\textwidth}{!}{%
$
\begin{array}{c !{\vline width 1pt} ccccccccccc!{\vline width 1pt}}
\textbf{S} & & & & & & & & & & &  \\[3pt]
\Xhline{1pt}

   & \alpha - \varepsilon \varepsilon' \beta & & & & & & & & &  \\[3pt]
  3& \alpha - \varepsilon \varepsilon' \beta& \varepsilon \varepsilon' \beta - \alpha& & & & & & & &  \\[3pt]

\frac{1}{8}  & 1&-1 &-1-2\varepsilon \varepsilon'+8\delta_{\varepsilon, \varepsilon' }\delta_{\omega, \omega'}& & & & & &  \\[3pt]
  
\frac{1}{8}  &1 &3 & -1-2\varepsilon \varepsilon'& 2\varepsilon, \varepsilon'-1& & & & &  \\[3pt]
  
 \frac{1}{6\sqrt{5}}  &\varepsilon & 3 \varepsilon & 0 & 0 &2\cos \frac{(2-\varepsilon \varepsilon')\pi}{5} & & & &  \\[3pt]
   
 \frac{1}{2\sqrt{5}}  &\varepsilon  &-\varepsilon  & 0& 0& 2\cos \frac{(2-\varepsilon \varepsilon')\pi}{5}  & 2\cos \frac{(3+\varepsilon \varepsilon')\pi}{5}  & & &  \\[3pt]
   
   \frac{1}{3\sqrt{5}}   &2\cos \frac{(3-\varepsilon \varepsilon_1')\pi}{10} \omega^{\varepsilon_2'}  & 0 &0 &0 & -\varepsilon_1'\omega^{-\varepsilon_2'}&0 &2\cos \frac{(3-\varepsilon_1 \varepsilon_1')\pi}{10} (\omega\omega')^{- \varepsilon_2 \varepsilon_2'}  & &  \\[3pt]
      
  \frac{1}{3 \sqrt{5}}&\varepsilon  \omega^{\varepsilon_2'}  &0 & 0& 0& 2\cos \frac{(2-\varepsilon_1' \varepsilon)\pi}{5} \omega^{-\varepsilon_2'}&0 &\varepsilon_1 (\omega\omega')^{- \varepsilon_2' \varepsilon_2} & 2 \cos \frac{(3+ \varepsilon_1 \varepsilon_1')\pi}{5} (\omega\omega')^{- \varepsilon_2 \varepsilon_2'} \\[3pt]

 \Xhline{1pt}
 \textbf{T}  & 1 & -1 &1  & -1& \zeta_5^\varepsilon &  -\zeta_5^\varepsilon & \omega & \omega \zeta_5^{\varepsilon_1}  \Tstrut \\[3pt]
\end{array}
$
}
\caption{The modular data for the Drinfeld center of the $4442$ fusion category. Here $\alpha=\frac{5}{120} $ and $\beta=\frac{2\sqrt{5}}{120} $. The eight blocks have sizes $6$, $2$, $6$, $2$, $6$, $2$, $12$, and $12$. The indexing of each block is as indicated in the text and primes are used for indices corresponding to rows. The number to the left of each row is a multiplicative factor which applies to each entry in that row.}
\label{st4442take4}
\end{figure}

\appendix
\newcommand{\urlprefix}{}
\bibliographystyle{alpha}
\bibliography{newbib}

\end{document}